\documentclass{article}

\usepackage{fullpage}
\usepackage{color}
\usepackage{float, ulem}

\usepackage[utf8]{inputenc}
\usepackage{geometry}
\usepackage{amssymb,latexsym,amsmath,amsfonts,amsthm}
\usepackage{bm}
\usepackage{graphicx}
\usepackage{epsfig}
\usepackage{tikz}
\usepackage{subfigure}
\usepackage{comment}
\usepgflibrary{decorations.markings}
\usetikzlibrary{decorations.markings}

\usepackage{titlesec}

\titleclass{\subsubsubsection}{straight}[\subsection]

\newcounter{subsubsubsection}[subsubsection]
\renewcommand\thesubsubsubsection{\thesubsubsection.\arabic{subsubsubsection}}

\titleformat{\subsubsubsection}
  {\normalfont\normalsize\bfseries}{\thesubsubsubsection}{1em}{}
\titlespacing*{\subsubsubsection}
{0pt}{3.25ex plus 1ex minus .2ex}{1.5ex plus .2ex}

\makeatletter
\renewcommand\paragraph{\@startsection{paragraph}{5}{\z@}%
  {3.25ex \@plus1ex \@minus.2ex}%
  {-1em}%
  {\normalfont\normalsize\bfseries}}
\renewcommand\subparagraph{\@startsection{subparagraph}{6}{\parindent}%
  {3.25ex \@plus1ex \@minus .2ex}%
  {-1em}%
  {\normalfont\normalsize\bfseries}}
\def\toclevel@subsubsubsection{4}
\def\toclevel@paragraph{5}
\def\toclevel@paragraph{6}
\def\l@subsubsubsection{\@dottedtocline{4}{7em}{4em}}
\def\l@paragraph{\@dottedtocline{5}{10em}{5em}}
\def\l@subparagraph{\@dottedtocline{6}{14em}{6em}}
\makeatother

\newcommand{\footremember}[2]{%
    \footnote{#2}
    \newcounter{#1}
    \setcounter{#1}{\value{footnote}}%
}
\newcommand{\footrecall}[1]{%
    \footnotemark[\value{#1}]%
}

\theoremstyle{plain}
\newtheorem{theorem}{Theorem}[section]
\newtheorem{lemma}[theorem]{Lemma}
\newtheorem{corollary}[theorem]{Corollary}
\newtheorem{proposition}[theorem]{Proposition}
\newtheorem{problem}[theorem]{Riemann-Hilbert Problem}

\theoremstyle{definition}

\newtheorem{remark}[theorem]{Remark}

\def \emph#1{ {\rm #1}}

\newcommand{\ds}{\displaystyle}
\newcommand{\C}{{\mathbb C}}

\newcommand{\BigO}[1]{\ensuremath{\operatorname{O}\left(#1\right)}}

\newcommand*\pFqskip{8mu}
\catcode`,\active
\newcommand*\pFq{\begingroup
        \catcode`\,\active
        \def ,{\mskip\pFqskip\relax}%
        \dopFq
}
\catcode`\,12
\def\dopFq#1#2#3#4#5{%
        {}_{#1}F_{#2}\biggl(\genfrac..{0pt}{}{#3}{#4} \hspace{1mm}\biggr|\hspace{1mm} #5\biggr)%
        \endgroup
}

\renewcommand{\l}{\lambda}
\newcommand{\ra}{\rightarrow}
\newcommand{\1}{{\bf 1}}
\newcommand{\hf}{\frac 12}

\newcommand\blfootnote[1]{%
  \begingroup
  \renewcommand\thefootnote{}\footnote{#1}%
  \addtocounter{footnote}{-1}%
  \endgroup
}

\begin{document}

\numberwithin{equation}{subsection}

\title{Diagonalization of the finite Hilbert transform on two adjacent intervals: the Riemann-Hilbert approach}

\author{M. Bertola\footnote{Department of Mathematics and Statistics, Concordia University, 1455 de Maisonneuve W., Montr\'eal, Qu\'ebec, Canada H3G 1M8.} \footnote{SISSA, International School for Advanced Studies, via Bonomea 265, Trieste, Italy.}, E. Blackstone\footnote{Department of Mathematics, KTH Royal Institute of Technology, Lindstedtsv\"{a}gen 25, SE-114 28 Stockholm, Sweden.}, A. Katsevich\footremember{UCF}{Department of Mathematics, University of Central Florida, P.O. Box 161364, 4000 Central Florida Blvd, Orlando, FL 32816-1364, USA.}, A. Tovbis\footrecall{UCF} \blfootnote{The research of M. B. was supported in part by the Natural Sciences and Engineering Research Council of Canada grant RGPIN-2016-06660.  The research of E. B., A. K., and A. T. was supported in part by NSF grant DMS-1615124.}}

\date{}

\maketitle

\begin{abstract}

    In this paper we study the spectra of bounded self-adjoint linear operators that are related to finite Hilbert transforms $\mathcal{H}_L:L^2([b_L,0])\to L^2([0,b_R])$ and $\mathcal{H}_R:L^2([0,b_R])\to L^2([b_L,0])$.  These operators arise when one studies the interior problem of tomography.  The diagonalization of $\mathcal{H}_R,\mathcal{H}_L$ has been previously obtained, but only asymptotically when $b_L\neq-b_R$.  We implement a novel approach based on the method of matrix Riemann-Hilbert problems (RHP) which diagonalizes $\mathcal{H}_R,\mathcal{H}_L$ explicitly.  We also find the asymptotics of the solution to a related RHP and obtain error estimates.

\end{abstract}

\tableofcontents

\section{Introduction}

\newcommand{\CH}{\mathcal{H}}
\newcommand{\Iin}{I_{\text{1}}}
\newcommand{\Iex}{I_{\text{2}}}
\newcommand{\br}{{\mathbb R}}

Let $\Iin,\Iex\subset\br$ be multiintervals, i.e. the unions of finitely many non-overlapping intervals (but the sets $\Iin, \Iex$ can overlap). The intervals can be bounded or unbounded. Consider the Finite Hilbert transform (FHT) and its adjoint:
\begin{equation}\label{hilb-orig}
\CH:\,L^2(\Iin)\to L^2(\Iex)\,\ (\CH f)(y)=\frac1\pi \int_{\Iin}\frac{f(x)}{x-y}dx;\ 
\CH^*:\, L^2(\Iex)\to L^2(\Iin),\ (\CH^* g)(x)=\frac1\pi \int_{\Iex}\frac{g(y)}{x-y}dy.
\end{equation}
The general problem we consider is to study the spectral properties of $\CH$ (and the associated self-adjoint operators $\CH^*\CH$, $\CH\CH^*$) depending on the geometry of the sets $\Iin,\Iex$. The properties we are interested in include finding the spectrum, establishing the nature of the spectrum (e.g., discrete vs. continuous), and finding the associated resolution of the identity. 
In the case when $\Iin=\Iex$, these problems were studied starting in the 50s and 60s, see e.g.  \cite{koppinc59, kopp60, widom60, kopp64, pincus64, putnam65, rosenblum66}.

More recently, the problem of the diagonalization of $\CH^*\CH$ and $\CH\CH^*$ occurred when solving the problem of image reconstruction from incomplete tomographic data, e.g. when solving the interior problem of tomography \cite{yyww-07, yyw-07b, yyw-08, kcnd, cndk-08, aak14}. In these applications, a significant diversity of different arrangements of $\Iin$, $\Iex$ was encountered: the two sets can be disjoint (i.e., $\text{dist}(\Iin, \Iex)>0$), touch each other (i.e., have a common endpoint), or overlap over an interval. Each of these arrangements leads to different spectral properties of the associated FHT. Generally, if $\text{dist}(\Iin, \Iex)>0$, the operators $\CH^*\CH$ and $\CH\CH^*$ are Hilbert-Schmidt with discrete spectrum \cite{kat10c, BKT16}.  In one particular case of overlap, when $\Iin=[a_1,a_3]$, $\Iex=[a_2,a_4]$, $a_1<a_2<a_3<a_4$, the spectrum is discrete, but has two accumulation points: $\lambda=0$ and $\lambda=1$ \cite{aak14, adk15}, where $\lambda$ denotes the spectral parameter.  In two cases when $\Iin$, $\Iex$ touch each other, the spectrum is purely absolutely continuous. The case $\Iin=[a_1,a_2]$, $\Iex=[a_2,a_3]$, $a_1<a_2<a_3$, was considered in \cite{KT16}. The case when $\Iin$, $\Iex$ are unions of more than one  sub-interval each and $\Iin\cup\Iex=\br$ was considered in \cite{kbt18}. 

To obtain the above mentioned results, three methods have been employed. In very few exceptional cases, e.g. when $\Iin\cup\Iex=\br$, it is possible to diagonalize the FHT explicitly via some ingenious transformations.  When $\Iin$, $\Iex$ each consist of a single interval (the intervals can be separated, touch, or overlap), the method of a commuting differential operator is used \cite{kat10c, kat11, aak14, adk15, KT16}. Here the associated singular functions (and kernels of the unitary transformations) are obtained as solutions of the special Sturm-Liouville problems for the differential operator that commutes with the FHT. As is seen, both of these approaches are fairly limited.

In \cite{BKT16}, a new powerful approach based on the method of the Riemann-Hilbert Problem (RHP) and the nonlinear steepest descent method of Deift-Zhou was proposed. This approach allows one to treat fairly general cases of $\Iin$, $\Iex$. However, the limitation of this approach so far  was the assumption $\text{dist}(\Iin, \Iex)>0$, which ensured that the associated operators $\CH^*\CH$ and $\CH\CH^*$ are Hilbert-Schmidt. In the case when $\Iin$, $\Iex$ touch each other, the RHP approach encountered several challenges that have not been studied before. Here are the three main ones: 
\begin{enumerate}
\item In the case of a purely discrete spectrum, residues are used to compute singular functions in \cite{BKT16}. However, in the limit $\text{dist}(\Iin, \Iex)\ra 0$ the spectra of the operators change from discrete to continuous. Thus, a different technique is needed to extract the spectral properties of the FHT from the solution of the corresponding RHP; 
\item Construction of small $\lambda$ approximations of the RHP solution near the common endpoints (parametrices) was not known, and it had to be developed; 
\item The proper boundary conditions near the common endpoints in the formulation of the RHP was not well understood. 
\end{enumerate}

As stated earlier, in the particular case when $\Iin=[a_1,a_2]$, $\Iex=[a_2,a_3]$, $a_1<a_2<a_3$, the spectral analysis of the FHT was performed in \cite{KT16}. Since the method in \cite{KT16} is based on a commuting differential operator, this method cannot be extended to more general situations of touching intervals,
for example, to the case  when $\Iex$ consists of two disjoint intervals. 
The goal of this paper is to extend the RHP and Deift-Zhou approach from \cite{BKT16} to the case when $\Iin$, $\Iex$ touch each other. This problem in its most general setting is very complicated. For example, if $\Iin$, $\Iex$ touch each other at several points, the continuous spectrum may have multiplicity greater than 1 (see \cite{kbt18}).  If there are also pieces of $\Iin$ that are at a positive distance from $\Iex$ (or, vice versa), then there will be discrete spectrum accumulating at $\lambda=0$ embedded in the continuous spectrum. In this paper we build the foundation for using the RHP/Deift-Zhou method by studying the case $\Iin=[a_1,a_2]$, $\Iex=[a_2,a_3]$, $a_1<a_2<a_3$, as a model example.  Our results include: 
\begin{enumerate}
\item Formulating the corresponding RHP {(including the proper boundary conditions)} and explicitly calculating its solution $\Gamma(z;\lambda)$ in terms of hypergeometric functions; 
\item Complete spectral analysis of the operators  $\CH^*\CH$ and $\CH\CH^*$ as well as {diagonalization} of the operators $\CH$, $\CH^*$; 
\item Calculating the leading order asymptotics of $\Gamma(z;\lambda)$  in the limit $\lambda\ra 0$ in various regions of the complex $z$ plane. These regions include, in particular, a small annulus centered at the {common endpoint} $a_2$. 
\item Finally, we also show that the spectral asymptotics of \cite{KT16} match the explicit spectral results of this paper.
\end{enumerate}

The asymptotics in the annulus around $a_2$ from Item 3 allows us to use  $\Gamma(z;\lambda)$  as a parametrix near any {common endpoint} for {more} general multiintervals $\Iin$, $\Iex$. This parametrix is the key missing link that is needed to construct the leading order asymptotics of $\Gamma(z;\lambda)$ for the general $\Iin$, $\Iex$. {This will be the subject of future research.}

The paper is organized as follows. In Section \ref{secKandRHP} we introduce the integral operator  $$
\hat{K}:L^2([b_L,b_R])\to L^2([b_L,b_R]),\ b_L<0<b_R,
$$
which, when restricted to $L^2([b_L,0])$ and $L^2([0,b_R])$, coincides with $\frac{1}{2i}\CH_L$ and $\frac{1}{2i}\CH_R$, respectively (see equation \eqref{hatK-restr}). Here $\CH_L$ is the FHT from $L^2([b_L,0])\to L^2([0,b_R])$, and $\CH_R$ is the FHT from $L^2([0,b_R])\to L^2([b_L,0])$. In the spirit of the above notation, we have $a_1=b_L$, $a_2=0$, and $a_3=b_R$.
It can easily be shown that the knowledge of the spectrum of $\hat K^2$ allows one to find the spectra  of $\CH_L^*\CH_L$ and $\CH_R\CH_R^*$.  Similarly to the case of disjoint intervals studied in \cite{BKT15}, the key observation is that the kernel $K$ of $ \hat K$, see \eqref{hatK}, 
is a kernel of integrable type in the sense of \cite{IIKS90}. Therefore, the kernel of the resolvent $\hat R$ of $\hat K$ 
is readily available in terms of the solution $\Gamma(z;\lambda)$ of the  matrix RHP \ref{Gamma4RHP}.  {Thus, RHP \ref{Gamma4RHP} plays a fundamental role in our paper.  Because the jump matrix of this RHP is constant in $z$, its solution  $\Gamma(z;\l)$ satisfies a Fuchsian system of linear differential equations with three singular points and, therefore, can be 
expressed in terms of hypergeometric functions. An explicit expression for $\Gamma(z;\l)$ is obtained in Theorem \ref{ThmRHPSol}. Using this
expression, we show that the matrix $\Gamma(z;\lambda)$ is analytic for $\lambda\in\overline{\mathbb{C}}\setminus[-1/2,1/2]$ and has analytic continuations across $(-1/2,0)$ and $(0,1/2)$ from above and from below (Proposition \ref{GammaNoLambdaPoles}). }

In Section \ref{secRHPDiag} we explicitly find the unitary operators  $U_L: L^2([b_L,0]) \to L^2([0,1],\sigma_{\chi_L})$ and ${U_R: L^2([0,b_R])\to L^2([0,1],\sigma_{\chi_R})}$, which diagonalize the operators $\mathcal{H}^*_L\mathcal{H}_L$ and  $\mathcal{H}^*_R\mathcal{H}_R$, respectively, see Theorem \ref{HstarHdiagHHstardiag}. Here $\sigma_{\chi_L}$, $\sigma_{\chi_R}$ denote the corresponding spectral measures. We prove in Theorems \ref{ResOfIdThm} and \ref{SimpSpec} that the spectrum of the operators   $\mathcal{H}^*_L\mathcal{H}_L$ and  $\mathcal{H}^*_R\mathcal{H}_R$  
is simple, purely absolutely continuous, and consists of the interval $[0,1]$. Our approach is based on the resolution of the identity  
Theorem \ref{AGdiag}  (see, for example, \cite{AG80}) and explicit calculation of the jump of the kernel
of the resolvent $\hat R$  over the spectral set in terms of the hypergeometric functions (Theorem \ref{resJump}). 
In the process, we find that $\Gamma(z;\lambda)$ satisfies the jump condition in the spectral variable $\l$ over the segment $[-\hf,\hf]$, see Theorem \ref{GammaJumpLambdaThm}, which, in some sense, is dual to the jump condition in the $z$ variable, see  RHP  \ref{Gamma4RHP}. {This observation is a RHP analogue of certain  bispectral problems \cite{Gru01}.}

The spectrum and diagonalization of the operators $\mathcal{H}_L, \mathcal{H}_R$ was studied in  \cite{KT16}, where the authors used a second order differential operator $L$, see \eqref{DiffOpL}, which commutes with $\mathcal{H}_L$ and $\CH_R$. The Titchmarsh-Weyl theory was utilized in  \cite{KT16} to obtain the small-$\l$  asymptotics of the unitary operators $U_1,U_2$ that diagonalize $\CH_L$, $\CH_R$, i.e. $U_2\CH_L U_1^*$ and $U_1\CH_R U_2^*$ are multiplication operators. In Section \ref{secODEDiag}, we construct the  operators $U_1,U_2$, see \eqref{U1}, \eqref{U2}, explicitly in terms of the hypergeometric functions everywhere on the continuous spectrum. We also show, see Theorem \ref{diagEquiv},  that the diagonalization of $\mathcal{H}^*_L\mathcal{H}_L$ and  $\mathcal{H}^*_R\mathcal{H}_R$, obtained in Section \ref{secRHPDiag}, is equivalent to the diagonalization obtained through the operators $U_1,U_2$.

In Section \ref{secGammaAsymp} we obtain the leading order behavior of  $\Gamma(z;\l)$ as $\l\ra 0$ in different regions of the complex $z$-plane,
see Theorem \ref{GammaAsympMainResult}.
The main tool we use here is the Deift-Zhou nonlinear steepest descent method combined with the $g$-function mechanism, which reduces (asymptotically)
the original RHP \ref{Gamma4RHP} for  $\Gamma(z;\l)$ to the so called model RHP \ref{modelRHP} for $\Psi(z)$. The latter represents the 
leading order approximation of $\Gamma(z;\l)$ on compact subsets of $\C\setminus [-\hf,\hf]$. Since the jump matrices in RHP \ref{modelRHP} 
commute with each other, the model RHP has a simple algebraic solution (Theorem \ref{model_sol}). However, due to a singularity at the common endpoint $z=0$, this solution is not unique. In order to 
select the appropriate $\Psi(z)$, we need to match it with the  leading order behavior of 
$\Gamma(z;\l)$ as $\l\ra 0$ in a small annulus centered at $z=0$, which is derived in Theorem \ref{GammaAsmptotics}. This, in turn,
requires calculating the leading order approximation of the hypergeometric function $F(a,b;c;\eta)$, where the parameters $a,b,c$ go to infinity in a certain way as $\l\ra 0$. Moreover,  this approximation must be uniform in a certain {large radius annulus $\Omega$ in the complex $\eta$-plane.}  Such asymptotics was recently obtained in \cite{Par13} based on the saddle point method
for complex integrals of the type of \eqref{int2}, but error estimates and uniformity needed for our purposes were not addressed there. 
Thus, we state and prove Theorem \ref{result} for such complex integrals.

Details of the construction of  $\Gamma(z;\l)$, the proof of Theorem \ref{result},  and other auxiliary  material can be found in the Appendix.


\section{Integral operator $\hat{K}$ and RHP}\label{secKandRHP}

Let us begin by defining the finite Hilbert transforms $\mathcal{H}_L:L^2([b_L,0])\to L^2([0,b_R])$ and $\mathcal{H}_R:L^2([0,b_R])\to L^2([b_L,0])$ by
\begin{equation}\label{HLHRdef}
\mathcal{H}_L[f](y):=\frac{1}{\pi}\int_{b_L}^0\frac{f(x)}{x-y}dx,~~ \mathcal{H}_R[g](x):=\frac{1}{\pi}\int_{0}^{b_R}\frac{g(y)}{y-x}dy.
\end{equation}
Notice that the adjoint of $\mathcal{H}_L$ is $-\mathcal{H}_R$.

\subsection{Definition and Properties of $\hat{K}$}

We define the integral operator $\hat{K}:L^2([b_L,b_R])\to L^2([b_L,b_R])$ by the requirements 
\begin{equation}\label{hatK-restr}
    \hat{K}\big|_{L^2([b_L,0])}=\frac{1}{2i}\mathcal{H}_L, ~~~ \hat{K}\big|_{L^2([0,b_R])}=\frac{1}{2i}\mathcal{H}_R.
\end{equation}
Explicitly,
\begin{equation}\label{hatK}
    \hat{K}[f](z):=\int_{b_L}^{b_R} K(z,x)f(x)~dx, ~~ \text{ where } ~~ K(z,x):=\frac{\chi_L(x)\chi_R(z)+\chi_R(x)\chi_L(z)}{2\pi i(x-z)}
\end{equation}
and $\chi_{L},\chi_R$ are indicator functions on $[b_L,0],[0,b_R]$, respectively.

\begin{proposition}\label{hatKProp}
The integral operator $\hat{K}:L^2([b_L,b_R])\to L^2([b_L,b_R])$ is self-adjoint and bounded.
\end{proposition}

\begin{proof}
    The boundedness of $\hat{K}$ follows from the boundedness of the Hilbert transform on $L^2(\mathbb{R})$ and we can see that $\hat{K}$ is self-adjoint because $K(z,x)=\overline{K(x,z)}$.
\end{proof}

\subsection{Resolvent of $\hat{K}$ and the Riemann-Hilbert Problem}

The operator $\hat{K}$ falls within the class of ``integrable kernels'' (see \cite{IIKS90}) and it is known that its spectral properties are intimately related to a suitable Riemann-Hilbert problem.  In particular, the kernel of the resolvent integral operator $\hat{R}=\hat{R}(\lambda):L^2([b_L,b_R])\to L^2([b_L,b_R])$, defined by
\begin{equation}\label{resolvent}
(I+\hat{R}(\lambda))\left(I-\frac{1}{\lambda}\hat{K}\right)=I,
\end{equation}
can be expressed through the solution $\Gamma(z;\lambda)$ of the following RHP.

\begin{problem}\label{Gamma4RHP}

Find a $2\times2$ matrix-function $\Gamma(z;\lambda)$, $\lambda\in\overline{\mathbb{C}}\setminus[-1/2,1/2]$, analytic for $z\in\overline{\mathbb{C}}\setminus [b_L,b_R]$ and satisfying
\begin{align}
\Gamma(z_+;\lambda)=&\Gamma(z_-;\lambda)\begin{bmatrix} 1 & -\frac{i}{\lambda} \\ 0 & 1 \end{bmatrix}, ~~ z\in[b_L,0], \label{rhpI1} \\
\Gamma(z_+;\lambda)=&\Gamma(z_-;\lambda)\begin{bmatrix} 1 & 0 \\ \frac{i}{\lambda} & 1 \end{bmatrix}, ~~ z\in[0,b_R], \label{rhpI3} \\
\Gamma(z;\lambda)= &\begin{bmatrix} \BigO{1} & \BigO{\log(z-b_L)}\end{bmatrix}, ~~ z\to b_L, \label{RHP2endpt} \\
\Gamma(z;\lambda)= &\begin{bmatrix} \BigO{\log(z-b_R)} & \BigO{1}\end{bmatrix}, ~~ z\to b_R, \label{RHP2endpt2} \\
\Gamma(z;\lambda)\in&L^2([b_L,b_R]), \label{RHP2endpt3} \\
\Gamma(z;\lambda) = & \1 + \BigO{z^{-1}}, ~~ z\to\infty, \label{rhpC1}
\end{align}
where $\1$ denotes the identity matrix.  The endpoint behavior of $\Gamma(z;\lambda)$ is described column-wise, the intervals $(b_L,0)$ and $(0,b_R)$ are positively oriented, and $z_\pm$ denotes the values on the positive/negative side of the jump contour $(b_L,b_R)$ respectively.

\end{problem}

As described in Appendix \ref{appendixGammaConstruct}, we are able to construct the solution of RHP \ref{Gamma4RHP} in terms of the hypergeometric functions
\begin{align}
    h_\infty(z)&:=e^{a\pi i}z^{-a}\pFq{2}{1}{a,a+1}{2a+2}{\frac{1}{z}} \implies h_\infty'(z)=-a e^{a\pi i}z^{-a-1}\pFq{2}{1}{a+1,a+1}{2a+2}{\frac{1}{z}}, \label{hhprime} \\
    s_\infty(z)&:=-\frac{z^{a+1}}{e^{a\pi i}}\pFq{2}{1}{-a-1,-a}{-2a}{\frac{1}{z}}\implies s_\infty'(z)=\frac{a+1}{-e^{a\pi i}}z^a\pFq{2}{1}{-a,-a}{-2a}{\frac{1}{z}}, \label{ssprime}
\end{align}
where $h_\infty,s_\infty$ are linearly independent solutions of the ODE 
\begin{equation}\label{hsODE}
    z(1-z)w''(z)+a(a+1)w(z)=0,
\end{equation}
and $a=a(\lambda)$, where 
\begin{equation}\label{aFunc}
a(\lambda):=\frac{1}{i\pi}\ln\left(\frac{i+\sqrt{4\lambda^2-1}}{2\lambda}\right).
\end{equation}
The principle branch of the log is taken and $\sqrt{4\lambda^2-1}=2\lambda+\BigO{1}$ as $\lambda\to\infty$, so it can be shown that $a(\lambda)$ is analytic for $\lambda\in\overline{\mathbb{C}}\setminus[-1/2,1/2]$.  This function $a(\lambda)$ will occur frequently throughout this paper so we have listed its relevant properties in Appendix \ref{aAppendix}.  We will often write $a$ in place of $a(\lambda)$ for convenience.  Recall that the standard Pauli matrices are
\begin{equation}\label{Pauli}
\sigma_1=\begin{bmatrix} 0 & 1 \\ 1 & 0 \end{bmatrix}, \hspace{3mm} \sigma_2=\begin{bmatrix} 0 & -i \\ i & 0 \end{bmatrix}, \hspace{3mm} \sigma_3=\begin{bmatrix} 1 & 0 \\ 0 & -1 \end{bmatrix}.
\end{equation}
In Appendix \ref{appendixGammaConstruct}, we describe how to construct the (unique!) solution to RHP \ref{Gamma4RHP}.

\begin{theorem}\label{ThmRHPSol}

For $\lambda\in\overline{\mathbb{C}}\setminus[-1/2,1/2]$, the unique solution to RHP \ref{Gamma4RHP} is 
\begin{equation}\label{RHPSolution}
\Gamma(z;\lambda)=\sigma_2Q^{-1}(\lambda)\hat{\Gamma}^{-1}\left(M_1(\infty)\right)\begin{bmatrix} 1 & \frac{b_Lb_R}{z(b_R-b_L)(a+1)} \\ 0 & 1 \end{bmatrix}\hat{\Gamma}\left(M_1(z)\right)Q(\lambda)\sigma_2,
\end{equation}
where $M_1(z)=\frac{b_R(z-b_L)}{z(b_R-b_L)}$ and 
\begin{equation}\label{QGammahat}
Q=\begin{bmatrix} -\tan(a\pi) & 0 \\ 0 & 4^{2a+1}e^{a\pi i}\frac{\Gamma(a+3/2)\Gamma(a+1/2)}{\Gamma(a)\Gamma(a+2)} \end{bmatrix}\begin{bmatrix} 1 & e^{a\pi i} \\ -e^{a\pi i} & 1 \end{bmatrix}, ~~ \hat{\Gamma}(z)=\begin{bmatrix} h_\infty\left(z\right) & s_\infty\left(z\right) \\ h_\infty'\left(z\right) & s_\infty'\left(z\right) \end{bmatrix}.
\end{equation}

Here $a:=a(\lambda)$ which is defined in \eqref{aFunc} and $h_\infty,s_\infty$ are defined in \eqref{hhprime},\eqref{ssprime}.

\end{theorem}

\begin{remark}\label{GammaSymmetry}
For any $\lambda\in\mathbb{C}\setminus[-1/2,1/2]$, the function $\Gamma(z;\lambda)$ has the symmetries 
\begin{equation}
\overline{\Gamma(\overline{z};\overline{\lambda})}=\Gamma(z;\lambda), \hspace{3mm} \sigma_3\Gamma(z;-\lambda)\sigma_3=\Gamma(z;\lambda)
\end{equation}
which follow from the observation that $\sigma_3\Gamma(z;-\lambda)\sigma_3$ and $\overline{\Gamma(\overline{z};\overline{\lambda})}$ also solve RHP \ref{Gamma4RHP}.
\end{remark}

\begin{proposition}\label{GammaNoLambdaPoles}
    The matrix $\Gamma(z;\lambda)$, defined in \eqref{RHPSolution}, is analytic for $\lambda\in\overline{\mathbb{C}}\setminus[-1/2,1/2]$ and can be analytically continued across the intervals $(-1/2,0)$ and $(0,1/2)$ from above and from below.
\end{proposition}

\begin{proof}
    Recall from \cite{DLMF} 15.2 that the hypergeometric function $F(a,b;c;z)$ (here $a,b,c$ are generic parameters) is an entire function of $a,b$ and meromorphic in $c$ with poles at non positive integers.  It easy to see that $\Gamma(z;\lambda)$ is analytic for $\lambda\in\overline{\mathbb{C}}\setminus[-1/2,1/2]$ because $a(\lambda)$ is analytic and $|\Re[a(\lambda)]|<\frac{1}{2}$ for $\lambda\in\mathbb{C}\setminus[-1/2,1/2]$.  Note that $a(\lambda)=0$ only for $\lambda=\infty$, so we can see that the second row of $Q(\lambda)$ has a simple zero when $\lambda=\infty$ and the second column of $\hat{\Gamma}(z,\lambda)$ has a simple pole when $\lambda=\infty$.  Thus the product $\hat{\Gamma}(z,\lambda)Q(\lambda)=\BigO{1}$ as $\lambda\to\infty$ so $\Gamma(z;\lambda)$ is analytic at $\lambda=\infty$ as well.  Using Appendix \ref{aAppendix}, we can see that $a(\lambda)\not\in\mathbb{R}$ for $\lambda\in(-1/2,0)\cup(0,1/2)$. Thus,  $\Gamma(z;\lambda)$ can be analytically continued  across the intervals  $(-1/2,0)$ and $(0,1/2)$ from above and from below (but these continuations do not coincide on  $(-1/2,0)$ and on  $(0,1/2)$). 

\end{proof}

We now show the relation between $\hat{K}$ and $\Gamma(z;\lambda)$.

\begin{theorem}\label{resolventKernel}
With the resolvent operator $\hat{R}$ defined by $\eqref{resolvent}$, let the kernel of $\hat{R}$ be denoted by $R$.  Then, 
\begin{equation}\label{res kernel}
R(z,x;\lambda)=\frac{\vec{g}_1^t(x)\Gamma^{-1}(x;\lambda)\Gamma(z;\lambda)\vec{f}_1(z)}{2\pi i\lambda(z-x)} \text{   where    } \vec{f}_1(z)=\begin{bmatrix} i\chi_L(z) \\ \chi_R(z) \end{bmatrix}, ~ \vec{g}_1(x)=\begin{bmatrix} -i\chi_R(x) \\ \chi_L(x) \end{bmatrix}.
\end{equation}
The matrix $\Gamma(z;\lambda)$ is defined in \eqref{RHPSolution} and functions $\chi_L,\chi_R$ are indicator functions on $[b_L,0],[0,b_R]$, respectively.
\end{theorem}

The proof is the same as in \cite{BKT16} (Lemma 3.16) so it will be omitted here.  An important ingredient of the proof is the observation that the jump of $\Gamma(z;\lambda)$ can be compactly written as
\begin{equation}\label{GammaJumpZStrucure}
    \Gamma(z_+;\lambda)=\Gamma(z_-;\lambda)\left(\1-\frac{1}{\lambda}\vec{f}_1(z)\vec{g}_1^t(z)\right)
\end{equation}
for $z\in[b_L,b_R]$ and 
\begin{equation}
    K(z,x)=\frac{\vec{f}_1^t(z)\vec{g}_1(x)}{2\pi i(x-z)},
\end{equation}
where $K(z,x)$ is the kernel of $\hat{K}$, see \eqref{hatK}.

\section{Spectral properties and diagonalization $\mathcal{H}_R^*\mathcal{H}_R$ and $\mathcal{H}_L^*\mathcal{H}_L$}\label{secRHPDiag}

The goal of this section is to construct unitary operators ${U_R:L^2([0,b_R])\to L^2(J,\sigma_R)}$ and $U_L:L^2([0,b_R])\to L^2(J,\sigma_L)$ such that 
\begin{equation}
    U_R\mathcal{H}^*_R\mathcal{H}_RU_R^*=\lambda^2, ~~ U_L\mathcal{H}^*_L\mathcal{H}_LU_L^*=\lambda^2
\end{equation}
where $\lambda^2$ is a multiplication operator (the space is clear by context), $J:=\{\lambda^2:0\leq\lambda^2\leq1\}$, and the spectral measures $\sigma_L,\sigma_R$ are to be determined.  This is to be understood in the sense of operator equality on $L^2(J,\sigma_R),L^2(J,\sigma_L)$, respectively.  We will begin this section with a brief summary of the spectral theory for a self-adjoint operator with simple spectrum.

\subsection{Basic facts about diagonalizing a self-adjoint operator with simple spectrum}\label{subsecDiag}

For an in-depth review of the spectral theorem for self-adjoint operators, see \cite{AG80}, \cite{Wei80}, \cite{DS57}.  We present a short summary of this topic which is directly related to the needs of this paper.  Let $\mathcal{K}$ be a Hilbert space and let $A$ be a self-adjoint operator with simple spectrum acting on $\mathcal{K}$.  Recall from \cite{AG80}, that a self-adjoint operator has simple spectrum if there is a vector $g\in\mathcal{K}$ so that the span of $\hat{E}_\Delta[g]$, where $\Delta$ runs through the set of all subintervals of the real line, is dense in $\mathcal{K}$.  Here the family of operators $\hat{E}_t$ denotes the so-called \textit{resolution of the identity} for the operator $A$, which we define in \eqref{ResOfIdGeneral}.  Define $\hat{R}$, the resolvent of $A$, via the formula 
\begin{equation}\label{stdRes}
    \hat{R}(t)=(tI-A)^{-1}
\end{equation}
for $t\not\in\mathbb{R}$.  Then, according to \cite{DS57} p.921, the resolution of the identity is computed by the formula 
\begin{equation}\label{ResOfIdGeneral}
    \hat{E}_{(\alpha,\beta)}:=\hat{E}_\beta-\hat{E}_\alpha=\lim_{\epsilon\to0^+}\lim_{\delta\to0^+}\int_{\alpha+\delta}^{\beta-\delta}\frac{-1}{2\pi i}\left[\hat{R}(t+i\epsilon)-\hat{R}(t-i\epsilon)\right]~dt,
\end{equation}
where $\alpha<\beta$.  Once we obtain $\hat{E}_t$, we can construct the unitary operators which will diagonalize $A$, as described in the following Theorem from \cite{AG80} p.279.

\begin{theorem}\label{AGdiag}
    If $A$ is a self-adjoint operator with simple spectrum, if $g$ is any generating element, and if $\sigma(t)=\langle \hat{E}_t[g],g \rangle$, then the formula
    \begin{equation}\label{unitaryOp}
        f=\int_{\mathbb{R}}\tilde{f}(t)~d\hat{E}_t[g]
    \end{equation}
    associates with each function $\tilde{f}\in L^2(\mathbb{R},\sigma)$ a vector $f\in\mathcal{K}$, and this correspondence is an isometric mapping of $L^2(\mathbb{R},\sigma)$ onto $\mathcal{K}$.  It maps the domain $D(Q)$ of the multiplication operator $Q$ in $L^2(\mathbb{R},\sigma)$ into the domain $D(A)$ of the operator $A$, and if the element $f\in D(A)$ corresponds to the function $\tilde{f}\in L^2(\mathbb{R},\sigma)$, then the element $Af$ corresponds to the function $t\tilde{f}(t)$.
    
\end{theorem}

\begin{remark}\label{diagShort}
    In short, Theorem \ref{AGdiag} says that $\sigma(t):=\langle \hat{E}_t[g],g \rangle$ defines the spectral measure ($g$ is any generating element) and the operator $U^*:L^2(\mathbb{R},\sigma)\to\mathcal{K}$ defined by
    \begin{equation}
        U^*[\tilde{f}](x):=\int_\mathbb{R}\tilde{f}(t)~d\hat{E}_t[g](x;t)
    \end{equation}
    is unitary.  Moreover,
    \begin{equation}
        UAU^*=t
    \end{equation}
    in the sense of operator equality on $L^2(\mathbb{R},\sigma)$.
\end{remark}

    Thus our immediate goal moving forward is to construct the resolution of the identity for $\mathcal{H}_R^*\mathcal{H}_R$ and $\mathcal{H}^*_L\mathcal{H}_L$.

\subsection{Resolution of the identity for $\mathcal{H}_R^*\mathcal{H}_R$ and $\mathcal{H}^*_L\mathcal{H}_L$}

From \eqref{ResOfIdGeneral}, knowledge of the resolvent operator is paramount.  We are able to express the resolvents of $\mathcal{H}_R^*\mathcal{H}_R,\mathcal{H}_L^*\mathcal{H}_L$ in terms of the resolvent of $\hat{K}$.

\begin{proposition}\label{HstarH_res}
    The resolvent of $\mathcal{H}_R^*\mathcal{H}_R$ and $\mathcal{H}_L^*\mathcal{H}_L$ is 
    \begin{align}
        \mathcal{R}_R(\lambda^2)&:=\left(I-\frac{1}{\lambda^2}\mathcal{H}_R^*\mathcal{H}_R\right)^{-1}=I+\pi_R\hat{R}(\lambda/2)\pi_R, \label{HstarHRres} \\
        \mathcal{R}_L(\lambda^2)&:=\left(I-\frac{1}{\lambda^2}\mathcal{H}_L^*\mathcal{H}_L\right)^{-1}=I+\pi_L\hat{R}(\lambda/2)\pi_L, \label{HstarHLres}
    \end{align}
    where $\pi_R:L^2([b_L,b_R])\to L^2([0,b_R])$, $\pi_L:L^2([b_L,b_R])\to L^2([b_L,0])$ are orthogonal projections (i.e. restrictions), $\hat{R}(\lambda)$ is defined by the relation \eqref{resolvent} and the kernel of $\hat{R}(\lambda)$ is computed in Theorem \ref{resolventKernel}.
    
\end{proposition}

\begin{remark}\label{resForm}
    
    The resolvents of $\mathcal{H}_R^*\mathcal{H}_R$ and $\mathcal{H}_L^*\mathcal{H}_L$, defined in \eqref{HstarHRres}, \eqref{HstarHLres}, slightly differ from the standard definition given by \eqref{stdRes}.  It is easy to see that
    \begin{equation}
        \frac{\mathcal{R}_R(\lambda^2)}{\lambda^2}=\left(\lambda^2I-\mathcal{H}^*_R\mathcal{H}_R\right)^{-1},
    \end{equation}
    and the same is true for the resolvent of $\mathcal{H}_L^*\mathcal{H}_L$.
    
\end{remark}

\begin{proof}
    In the direct sum decomposition $L^2([b_L,b_R])=L^2([b_L,0])\oplus L^2([0,b_R])$, $\hat{K}$ has the block structure 
    \begin{equation}
        \hat{K}=\begin{bmatrix} 0 & -\frac{i}{2}\mathcal{H}_L \\ -\frac{i}{2}\mathcal{H}_R & 0 \end{bmatrix}=\begin{bmatrix} 0 & \frac{i}{2}\mathcal{H}_R^* \\ -\frac{i}{2}\mathcal{H}_R & 0 \end{bmatrix}=\begin{bmatrix} 0 & -\frac{i}{2}\mathcal{H}_L \\ \frac{i}{2}\mathcal{H}_L^* & 0 \end{bmatrix}.
    \end{equation}
For $\lambda$ sufficiently large, we can write (recall that $\hat{K}$ is bounded, see Proposition \ref{hatKProp})
    \begin{align}\label{hatKres}
        I+\hat{R}(\lambda)=\left(I-\frac{1}{\lambda}\hat{K}\right)^{-1}=\sum_{n=0}^\infty\left(\frac{\hat{K}}{\lambda}\right)^n
    \end{align}
    where all the even powers in the right hand side of \eqref{hatKres} are block diagonal and all the odd powers in \eqref{hatKres} are block off-diagonal.  The result is extended to all $\lambda\not\in\mathbb{R}$ via analytic continuation.  Similarly, we can write 
    \begin{align}
        \left(I-\frac{1}{\lambda^2}\mathcal{H}_R^*\mathcal{H}_R\right)^{-1}=\sum_{n=0}^\infty \left(\frac{\mathcal{H}_R^*\mathcal{H}_R}{\lambda^2}\right)^n
    \end{align}
    and comparing with the series in \eqref{hatKres} gives our result for the resolvent of $\mathcal{H}_R^*\mathcal{H}_R$.  The proof for the resolvent of $\mathcal{H}_L^*\mathcal{H}_L$ is nearly identical.

\end{proof}

To construct the resolution of the identity (see \eqref{ResOfIdGeneral}), we need to compute the jump of the resolvent of $\mathcal{H}_R^*\mathcal{H}_R$ and $\mathcal{H}_L^*\mathcal{H}_L$ in the $\lambda$-plane.  The kernel of the resolvent is expressed in terms of $\Gamma(z;\lambda)$ (see Theorem \ref{resolventKernel}), so we need to compute the jump of $\Gamma(z;\lambda)$ in the $\lambda-$plane.

\begin{remark}\label{mobius}
    
In the remaining sections of this paper we will frequently encounter the following M\"obius transformations:
\begin{align}
        M_1(x)&=\frac{b_R(x-b_L)}{x(b_R-b_L)}, ~~ M_2(x)=\frac{b_Rb_Lx}{x(b_R+b_L)-b_Rb_L}, \\
        M_3(x)&=M_1(M_2(x))=\frac{-b_L(x-b_R)}{x(b_R-b_L)}, ~~ M_4(x)=\frac{x(b_R-b_L)}{x(b_R+b_L)-2b_Rb_L}.
\end{align}

\end{remark}

Define functions
\begin{align}\label{drdl}
    d_R(z;\lambda)&:=\alpha(\lambda)h_\infty '\left(M_1(z)\right)+\beta(\lambda)s_\infty '\left(M_1(z)\right), \\ d_L(z;\lambda)&:=-e^{a\pi i}\alpha(\lambda)h_\infty '\left(M_1(z)\right)+e^{-a\pi i}\beta(\lambda)s_\infty '\left(M_1(z)\right),
\end{align}
where $h_\infty',s_\infty'$ is defined \eqref{hhprime}, \eqref{ssprime}, $a:=a(\lambda)$ is defined in \eqref{aFunc}, and coefficients $\alpha(\lambda),\beta(\lambda)$ are  
\begin{align}\label{alphabeta}
    \alpha(\lambda):=\frac{e^{-a\pi i}\tan(a\pi)\Gamma(a)}{4^{a+1}\Gamma(a+3/2)}, ~~ \beta(\lambda):=\frac{4^{a}e^{a\pi i}\Gamma(a+1/2)}{\Gamma(a+2)}.
\end{align}
The functions $d_L,d_R$ will play a major role in this section so we have complied all of their relevant properties in Appendix \ref{dldr_appendix}.  Importantly, it is shown that both $d_L(z;\lambda)$ and $d_R(z;\lambda)$ are single valued when $\lambda\in(-1/2,0)$.  For $\lambda\in(-1/2,0)\cup(0,1/2)$, define vectors
\begin{align}\label{fgVec}
    \vec{f}_2(z,\lambda)&:=\frac{-2b_Rb_L|\lambda|(2a_-(-|\lambda|)+1)}{b_R-b_L}\begin{bmatrix} d_R(z;-|\lambda|) \\ \text{sgn}(\lambda)d_L(z;-|\lambda|) \end{bmatrix}, ~~~ \vec{g}_2(z,\lambda):=\begin{bmatrix} -\text{sgn}(\lambda)d_L(z;-|\lambda|) \\ d_R(z;-|\lambda|) \end{bmatrix}.
\end{align}
We are now ready to compute the jump of $\Gamma(z;\lambda)$ in the $\lambda-$plane.

\begin{theorem}\label{GammaJumpLambdaThm}
For $\lambda\in(-1/2,0)\cup(0,1/2)$,
\begin{align}\label{GammaJumpLambda}
    \Gamma(z;\lambda_+)&=\Gamma(z;\lambda_-)\left[\1-\frac{1}{z}\vec{f}_2(z,\lambda)\vec{g}_2^t(z,\lambda)\right],
\end{align}
where $\Gamma(z;\lambda)$ is defined in \eqref{RHPSolution}.
\end{theorem}

\begin{proof}

The proof is straightforward; we have
\begin{align}
\Gamma(z;\lambda_+)&=\Gamma(z;\lambda_-)\Gamma^{-1}(z;\lambda_-)\Gamma(z;\lambda_+)
\end{align}
and thus we compute $\Gamma^{-1}(z;\lambda_-)\Gamma(z;\lambda_+)$ to obtain the stated result.  We see from \eqref{RHPSolution} that $\Gamma(z;\lambda)$ is defined in terms of $\hat{\Gamma}, Q$ which are defined in \eqref{QGammahat}.  For $\lambda\in(-1/2,0)$ (standard orientation), we find that
\begin{align}
    \hat{\Gamma}\left(z,\lambda_+\right)&=\hat{\Gamma}\left(z,\lambda_-\right)\sigma_1, \label{hatGammaJump}\\
    Q_+(\lambda)&=\left[\frac{-\tan(a\pi)\Gamma(a)\Gamma(a+2)}{e^{2\pi ia}4^{2a+1}\Gamma(a+3/2)\Gamma(a+1/2)}\right]_-\sigma_1Q_-(\lambda). \label{Qjump}
\end{align}
Recall from Proposition \ref{aProp} that $a_+(\lambda)+a_-(\lambda)=-1$ for $\lambda\in(-1/2,0)$ thus the jump for $Q$ follows.  Then by inspection, we see from \eqref{hhprime}, \eqref{ssprime} 
\begin{equation}
    h_\infty(z,a_+(\lambda))=h_\infty(z,-1-a_-(\lambda))=s_\infty(z,a_-(\lambda))
\end{equation}
so the jump of $\hat{\Gamma}$ follows.  Now using \eqref{RHPSolution}, \eqref{hatGammaJump}, \eqref{Qjump} we obtain
\begin{align}
\Gamma^{-1}(z;\lambda_-)\Gamma(z;\lambda_+)&=\1-\frac{b_Lb_R}{z(b_R-b_L)}\left(\frac{2a+1}{a(a+1)}\right)_-M^{-1}(z,\lambda_-)\begin{bmatrix} 0 & 1 \\ 0 & 0 \end{bmatrix}M(z,\lambda_-)    
\end{align}
where we have used $\hat{\Gamma}_-(M_1(\infty))Q_-Q^{-1}_+\hat{\Gamma}_+(M_1(\infty))=c_-(\lambda)\1$ and $c_-(\lambda)$ is the scalar appearing on the right hand side of \eqref{Qjump}.  Here $M(z,\lambda):=\hat{\Gamma}\left(M_1(z)\right)Q\sigma_2$.  Let $m_{21},m_{22}$ denote the (2,1), (2,2) elements of the matrix $M$, respectively.  Then, 
\begin{align*}
    M^{-1}\begin{bmatrix} 0 & 1 \\ 0 & 0 \end{bmatrix}M&=|M|^{-1}\begin{bmatrix} m_{22}m_{21} & m_{22}^2 \\ -m_{21}^2 & -m_{21}m_{22} \end{bmatrix}=|M|^{-1}\begin{bmatrix} m_{22} \\ -m_{21} \end{bmatrix}\begin{bmatrix} m_{21} & m_{22} \end{bmatrix}
\end{align*}
and, explicitly, 
\begin{align*}
    &m_{21}(z;\lambda)=\frac{ie^{a\pi i}\Gamma(a+3/2)}{4^{-a-1}\Gamma(a)}d_L(z;\lambda), \qquad
    m_{22}(z;\lambda)=\frac{ie^{a\pi i}\Gamma(a+3/2)}{4^{-a-1}\Gamma(a)}d_R(z;\lambda), \\
    &|M(z;\lambda)|=-\frac{2e^{2\pi ia}4^{2a+1}\Gamma^2\left(a+\frac{3}{2}\right)}{\lambda a(a+1)\Gamma^2(a)}.
\end{align*}
To compute $|M|$ we have used \eqref{appendix_detHatGamma}.  Finally, we calculate  
\begin{align*}
    &\left(\frac{2a+1}{a(a+1)}\right)_-M^{-1}(z,\lambda_-)\begin{bmatrix} 0 & 1 \\ 0 & 0 \end{bmatrix}M(z,\lambda_-)=\left(\frac{2a+1}{a(a+1)|M|}\right)_-\begin{bmatrix} m_{22}(z;\lambda_-) \\ -m_{21}(z;\lambda_-) \end{bmatrix}\begin{bmatrix} m_{21}(z;\lambda_-) & m_{22}(z;\lambda_-) \end{bmatrix} \\
    &=2\lambda\left(2a_-+1\right)\begin{bmatrix} d_R(z;\lambda) \\ -d_L(z;\lambda) \end{bmatrix}\begin{bmatrix} d_L(z;\lambda) & d_R(z;\lambda) \end{bmatrix}
\end{align*}
which can be used to obtain the desired result for $\lambda\in(-1/2,0)$.

Now, to calculate the jump of $\Gamma(z;\lambda)$ when $\lambda\in(0,1/2)$, we take advantage of the symmetry $\Gamma(z;\lambda)=\sigma_3\Gamma(z;-\lambda)\sigma_3$, see Remark \ref{GammaSymmetry}.  For brevity, let 
\begin{equation}
    J(z;\lambda):=\1-\frac{1}{z}\vec{f}_2(z,\lambda)\vec{g}_2^t(z,\lambda).
\end{equation}
Using \eqref{GammaJumpLambda} for $\lambda\in(-1/2,0)$, we obtain
\begin{equation}
    \Gamma(z;\lambda_+)=\Gamma(z;\lambda_-)J(z;\lambda_-)
\end{equation}
only for $\lambda\in(-1/2,0)$.  For $\lambda\in(0,1/2)$,
\begin{align}
    \Gamma(z;\lambda_+)&=\sigma_3\Gamma(z,(-\lambda)_-)\sigma_3=\sigma_3\Gamma(z,(-\lambda)_+)J^{-1}(z;(-\lambda)_-)\sigma_3=\Gamma(z;\lambda_-)\sigma_3J^{-1}(z;(-\lambda)_-)\sigma_3.
\end{align}
It can now be verified that 
\begin{equation}
    \sigma_3J^{-1}(z;(-\lambda)_-)\sigma_3=\1-\frac{1}{z}\vec{f}_2(z,\lambda)\vec{g}_2^t(z,\lambda_-)
\end{equation}
for $\lambda\in(0,1/2)$.

\end{proof}

Recall from Proposition \ref{HstarH_res} and Theorem \ref{resolventKernel} that the resolvents of $\mathcal{H}^*_L\mathcal{H}_L$, $\mathcal{H}^*_R\mathcal{H}_R$ are expressed in terms of $\Gamma(z;\lambda)$.  In light of the previous Theorem, we can now compute the jump of the resolvents of $\mathcal{H}^*_L\mathcal{H}_L$, $\mathcal{H}^*_R\mathcal{H}_R$ in the $\lambda$ plane, which is required to construct the resolution of the identity, see \eqref{ResOfIdGeneral}.

\begin{theorem}\label{resJump}
The kernel of $\hat{R}(\lambda)$ is single valued for $\lambda\in\overline{\mathbb{C}}\setminus[-1/2,1/2]$ and satisfies the jump property 
    \begin{equation}
        R(z,x;\lambda_+)-R(z,x;\lambda_-)=\frac{\begin{bmatrix} -i\chi_R(x) & \chi_L(x) \end{bmatrix}\vec{f}_2(x,\lambda)\vec{g}_2^t(z,\lambda)\begin{bmatrix} i\chi_L(z) \\ \chi_R(z) \end{bmatrix}}{2\pi i\lambda xz}
    \end{equation}
    for $\lambda\in(-1/2,0)\cup(0,1/2)$, where $\chi_L,\chi_R$ are the characteristic functions on $(b_L,0), (0,b_R)$ and $R(z,x;\lambda)$, $\vec{f}_2(x,\lambda),\vec{g}_2(z,\lambda)$ are defined in \eqref{res kernel},\eqref{fgVec}, respectively.
\end{theorem}

\begin{proof}

    Recall, from \eqref{resolventKernel}, that 
    \begin{equation}
        R(z,x;\lambda)=\frac{\vec{g}_1^t(x)\Gamma^{-1}(x;\lambda)\Gamma(z;\lambda)\vec{f}_1(z)}{2\pi i\lambda(z-x)}.
    \end{equation}
    From Proposition \ref{GammaNoLambdaPoles} we can see that $\Gamma(z;\lambda)$ is single valued for $\lambda\in\overline{\mathbb{C}}\setminus[-1/2,1/2]$ so from \eqref{resolventKernel} it immediately follows that the kernel of $\hat{R}(\lambda)$ is also single valued for $\lambda\in\overline{\mathbb{C}}\setminus[-1/2,1/2]$.  Let $\Delta_\lambda F(\lambda):=F(\lambda_+)-F(\lambda_-)$ for any $F$.  To prove the result we need to compute $\Delta_\lambda [\Gamma^{-1}(x;\lambda)\Gamma(z;\lambda)]$.  For $\lambda\in(-1/2,0)$, we calculate (see proof of Theorem \ref{GammaJumpLambdaThm})
    \begin{align}
        \Gamma^{-1}(x;\lambda_-)\Gamma(z;\lambda_-)&=M^{-1}(x,\lambda_-)\begin{bmatrix} 1 & \frac{b_Lb_R(x-z)}{xz(b_R-b_L)(a+1)} \\ 0 & 1 \end{bmatrix}_-M(z,\lambda_-).
    \end{align}
    Using $M(z,\lambda_+)=c_-(\lambda)M(z,\lambda_-)$, where $c_-(\lambda)$ is the scalar found in the right hand side of \eqref{Qjump},
    \begin{align}
        \Gamma^{-1}(x;\lambda_+)\Gamma(z;\lambda_+)&=M^{-1}(x,\lambda_+)\begin{bmatrix} 1 & \frac{b_Lb_R(x-z)}{xz(b_R-b_L)(a+1)} \\ 0 & 1 \end{bmatrix}_+M(z,\lambda_+) \nonumber \\
        &=M^{-1}(x,\lambda_-)\begin{bmatrix} 1 & \frac{-b_Lb_R(x-z)}{axz(b_R-b_L)} \\ 0 & 1 \end{bmatrix}_-M(z,\lambda_-),
    \end{align}
    so we have that
    \begin{align}
        \Delta_\lambda [\Gamma^{-1}(x;\lambda)\Gamma(z;\lambda)]&=\left(\frac{-b_Lb_R(x-z)(2a+1)}{xza(a+1)(b_R-b_L)}\right)_-M^{-1}(x,\lambda_-)\begin{bmatrix} 0 & 1 \\ 0 & 0 \end{bmatrix}M(z,\lambda_-) \cr
        &=\frac{z-x}{xz}\vec{f}_2(x,\lambda)\vec{g}_2^t(z,\lambda).
    \end{align}
    Now for $\lambda\in(0,1/2)$, we again take advantage of the symmetry $\Gamma(z;\lambda)=\sigma_3\Gamma(z;-\lambda)\sigma_3$, see Remark \ref{GammaSymmetry}.  The process is the same as in the proof of Theorem \ref{GammaJumpLambdaThm}.  Directly from \eqref{res kernel} we have
    \begin{align}
        \Delta_\lambda R(z,x;\lambda)&=\frac{\vec{g}_1^t(x)\Delta_\lambda\left[\Gamma^{-1}(x;\lambda)\Gamma(z;\lambda)\right]\vec{f}_1(z)}{2\pi i\lambda(z-x)}
    \end{align}
    and plugging in our calculation of $\Delta_\lambda\left[\Gamma^{-1}(x;\lambda)\Gamma(z;\lambda)\right]$ gives the result. 
    
\end{proof}

\begin{remark}
    From Theorem \ref{resJump} we can immediately see that when $\lambda\in(-1/2,0)\cup(0,1/2)$ and $x,z\in(0,b_R)$, 
    \begin{equation}\label{RkernelRight}
        R(z,x;\lambda_+)-R(z,x;\lambda_-)=\frac{b_Lb_R(2a_-(-|\lambda|)+1)}{\emph{sgn}(\lambda)\pi xz(b_R-b_L)}d_R(x;-|\lambda|)d_R(z;-|\lambda|)
    \end{equation}
    and when $x,z\in(b_L,0)$,
    \begin{equation}\label{RkernelLeft}
        R(z,x;\lambda_+)-R(z,x;\lambda_-)=\frac{b_Lb_R(2a_-(-|\lambda|)+1)}{\emph{sgn}(\lambda)\pi xz(b_R-b_L)}d_L(x;-|\lambda|)d_L(z;-|\lambda|),
    \end{equation}
    where $d_R,d_L$ are defined in \eqref{drdl}.
\end{remark}

\begin{proposition}\label{noEigen}
    The operators $\mathcal{H}_R^*\mathcal{H}_R, \mathcal{H}_L^*\mathcal{H}_L$ do not have eigenvalues.
\end{proposition}

\begin{proof}
    We will prove this statement for $\mathcal{H}_R^*\mathcal{H}_R$ only, as the proof for $\mathcal{H}_L^*\mathcal{H}_L$ is similar.  We show that the resolvent of $\mathcal{H}_R^*\mathcal{H}_R$ has no poles in the $\lambda$ plane.  Recall from \eqref{HstarHRres} that 
    \begin{equation}
        \mathcal{R}_R(\lambda^2)=\left(I-\frac{1}{\lambda^2}\mathcal{H}_R^*\mathcal{H}_R\right)^{-1}=I+\pi_R\hat{R}(\lambda/2)\pi_R,
    \end{equation}
    and using \eqref{resolventKernel}, we can see that the kernel of $\mathcal{R}_R$ is 
    \begin{equation}\label{kernelHRstarHRres}
        \frac{[\Gamma^{-1}(x;\lambda/2)\Gamma(z;\lambda/2)]_{1,2}}{\pi \lambda(x-z)}\chi_R(x)\chi_R(z).
    \end{equation}
    Using Proposition \ref{GammaNoLambdaPoles}, we know that $\Gamma(z;\lambda/2)$ can (potentially) have a pole only when $\lambda=0,\pm1$.  Since \eqref{kernelHRstarHRres} is single-valued for $x,z\in(0,b_R)$ and pole-free for $\lambda\in\overline{\mathbb{C}}\setminus\{0,\pm1\}$, $\mathcal{R}_R(\lambda^2)$ is pole-free for $\lambda^2\in\overline{\mathbb{R}}\setminus\{0,1\}$.  We can see that $\lambda^2=0$ is not an eigenvalue of $\mathcal{H}_R^*\mathcal{H}_R$ because the null space of $\mathcal{H}_R$ contains only the zero vector. Likewise, it is easy to see that $\lambda^2=1$ is not an eigenvalue as well. Otherwise $\mathcal{H}_R^*\mathcal{H}_Rf=f$ for some $f\in L^2([0,b_R])$, $f\not\equiv0$, and the contradiction
\begin{equation}
\Vert f\Vert_{L^2([0,b_R])}
=\Vert \mathcal{H}_R^*\mathcal{H}_Rf \Vert_{L^2([0,b_R])}
<\Vert \CH^*_R\mathcal{H}_Rf \Vert_{\br}=\Vert \mathcal{H}_Rf \Vert_{L^2([b_L,0])}
<\Vert \mathcal{H}_Rf \Vert_{\br}=\Vert f \Vert_{L^2([0,b_R])}
\end{equation}
proves the desired assertion. Here we used the convention that whenever the norm $\Vert\cdot\Vert_\br$ is computed, the left-most Hilbert transform inside the norm is evaluated over the entire line.
\end{proof}

For convenience we define 
\begin{align}\label{DRDL}
    D_R(z;\lambda):=d_R(z;-|\lambda|/2), ~~  D_L(z;\lambda):=d_L(z;-|\lambda|/2),
\end{align}
where $d_L,d_R$ are defined in \eqref{drdl}.  Recall from Appendix \ref{dldr_appendix} that $d_L(z;-|\lambda|), d_R(z;-|\lambda|)$ are single-valued for $\lambda\in(-1/2,1/2)$.  Now for $y\in[b_L,0], x\in[0,b_R], \lambda^2\in[0,1]$,  define the kernels
\begin{equation}\label{phiLR}
    \phi_L(y,\lambda^2):=\frac{D_L(y;\lambda)}{\pi y|\lambda|D_R(\infty;\lambda)}, ~~~ \phi_R(x,\lambda^2):=\frac{-D_R(x;\lambda)}{\pi x|\lambda|D_L(\infty;\lambda)}.
\end{equation}
Using Proposition \ref{dldr_AppendixProp} we can see that $\phi_L(y,\lambda^2)=\BigO{1}$ as $y\to b_L$ and $\phi_R(x,\lambda^2)=\BigO{1}$ as $x\to b_R$.  We can see from \eqref{hhprime}, \eqref{ssprime} that $h_\infty'(M_1(z))=\BigO{\sqrt{z}}$, $s_\infty'(M_1(z))=\BigO{\sqrt{z}}$ as $z\to0$ whenever $\lambda^2\in[0,1]$.  Since both $D_R, D_L$ are linear combinations of $h_\infty', s_\infty'$, it follows that $\phi_L(y,\lambda^2)=\BigO{y^{-1/2}}$ as $y\to0^-$ and $\phi_R(x,\lambda^2)=\BigO{x^{-1/2}}$ as $x\to0^+$.  Also define the weights
\begin{align}\label{sigmaPrime}
    \frac{d\sigma_L(\lambda^2)}{d\lambda^2}=\frac{|b_L|b_R(a+1/2)}{i(b_R-b_L)}D_R^2(\infty;\lambda), ~~~ \frac{d\sigma_R(\lambda^2)}{d\lambda^2}=\frac{|b_L|b_R(a+1/2)}{i(b_R-b_L)}D_L^2(\infty;\lambda),
\end{align}
where $a:=a_-(-|\lambda|/2)$.  Notice that both $\sigma_L'(\lambda^2), \sigma_R'(\lambda^2)$ (here ' denotes differentiation with respect to $\lambda^2$) are non-negative real analytic for $\lambda^2\in(0,1)$, by Propositions \ref{aProp} and \ref{dldr_AppendixProp}.  With \eqref{ResOfIdGeneral} in mind, we can now prove the following theorem.

\begin{theorem}\label{ResOfIdThm} The spectrum of $\mathcal{H}_R^*\mathcal{H}_R$, $\mathcal{H}_L^*\mathcal{H}_L$ is $\lambda^2\in[0,1]$. Moreover, for $0\leq\lambda^2\leq1$, $g\in L^2([b_L,0])$, $f\in L^2([0,b_R])$, the operators 
    \begin{align}
        \hat{E}_{R,\lambda^2}[f](x)&=\int_0^{\lambda^2}\int_0^{b_R}\phi_R(x,\mu^2)\phi_R(z,\mu^2)f(z)~dzd\sigma_R(\mu^2), \label{ER_ResOfId} \\
        \hat{E}_{L,\lambda^2}[g](x)&=\int_0^{\lambda^2}\int_{b_L}^0\phi_L(x,\mu^2)\phi_L(z,\mu^2)g(z)~dzd\sigma_L(\mu^2), \label{EL_ResOfId}
    \end{align}
     where $\phi_L,\phi_R$ and $\sigma_L,\sigma_R$ are defined in \eqref{phiLR}, \eqref{sigmaPrime}, respectively, are the resolution of the identity for $\mathcal{H}_R^*\mathcal{H}_R$, $\mathcal{H}_L^*\mathcal{H}_L$, respectively.
\end{theorem}

\begin{proof}
We will consider only the operator $\mathcal{H}_R^*\mathcal{H}_R$. The proofs for $\mathcal{H}_L^*\mathcal{H}_L$ are completely analogous. It follows from Proposition \ref{HstarH_res} and Theorem \ref{resJump} that
    \begin{equation}\label{specLoc}
        \mathcal{R}_{R+}(\lambda^2)-\mathcal{R}_{R-}(\lambda^2)\begin{cases}
        =0, &\lambda^2\in{\mathbb{R}}\setminus[0,1] \\
        \neq0, &\lambda^2\in[0,1],
        \end{cases}
    \end{equation}
    where $\mathcal{R}_R$ is the resolvent of $\mathcal{H}_R^*\mathcal{H}_R$, see \eqref{HstarHRres}.  Thus the spectral set is $\lambda^2\in[0,1]$.  It was shown in Proposition \ref{noEigen} that $\mathcal{H}_R^*\mathcal{H}_R$ has no eigenvalues.
        
Next, we construct $\hat{E}_{R,\lambda^2}$.  It was shown in Theorem \ref{resJump} that $\hat{R}(\lambda)$ is single-valued for $\lambda\in\mathbb{C}\setminus[-1/2,1/2]$, thus from \eqref{ResOfIdGeneral} and Proposition \ref{HstarH_res} we have 
    \begin{equation}\label{hatEproof}
        \hat{E}_{R,\lambda^2}=\lim_{\epsilon\to0^+}\lim_{\delta\to0^+}\frac{-1}{2\pi i}\int_{\delta}^{\lambda^2-\delta}\frac{1}{\mu^2}\left[\mathcal{R}_R(\mu^2+i\epsilon)-\mathcal{R}_R(\mu^2-i\epsilon)\right]d\mu^2.
    \end{equation}
    We have shown in Proposition \ref{noEigen} that $\mathcal{H}_R^*\mathcal{H}_R$ has no eigenvalues.  According to \cite{AG80} section 82, the lack of eigenvalues guarantees that $\hat{E}_{R,\lambda^2}$ has no points of discontinuity.  Now returning to \eqref{hatEproof}, we can take $\delta=0$ and we can move the $\epsilon$ limit inside the integral as the kernel of $\mathcal{R}_R$ has analytic continuation above and below the interval $(0,1)$.  So from \eqref{ResOfIdGeneral}, \eqref{HstarHRres}, and Remark \ref{resForm} we obtain 
    \begin{align}
        \hat{E}_{R,\lambda^2}&=\frac{-1}{2\pi i}\int_0^{\lambda^2}\frac{1}{\mu^2}\left[\mathcal{R}_{R+}(\mu^2)-\mathcal{R}_{R-}(\mu^2)\right]d\mu^2 \cr
        &=\int_0^{\lambda^2}\frac{-\text{sgn}(\mu)}{2\pi i\mu^2}\pi_R\left[\hat{R}_+(\mu/2)-\hat{R}_-(\mu/2)\right]\pi_Rd\mu^2,
    \end{align}
    and now plugging in \eqref{RkernelRight} gives the result.  Note that $\Delta_{\lambda^2}\mathcal{R}_R(\lambda^2)=\text{sgn}(\lambda)\Delta_\lambda\hat{R}(\lambda/2)$, because when $\lambda$ is on the upper shore of $(-1,0)$, $\lambda^2$ is on the lower shore of $(0,1)$.

\end{proof}

\subsection{Nature of the spectrum of $\mathcal{H}_R^*\mathcal{H}_R$ and $\mathcal{H}_L^*\mathcal{H}_L$}

In this subsection, we show that the spectrum of $\mathcal{H}_R^*\mathcal{H}_R$ and $\mathcal{H}_L^*\mathcal{H}_L$ is simple and purely absolutely continuous.  We will prove statements in this section for $\mathcal{H}_R^*\mathcal{H}_R$ only because the statements and ideas for proofs are nearly identical for $\mathcal{H}_L^*\mathcal{H}_L$.  Notice that the resolution of the identity of $\mathcal{H}_R^*\mathcal{H}_R$ (see \eqref{ER_ResOfId}) can be compactly written as
\begin{equation}\label{ResOfIdWithUR}
    \hat{E}_{R,\lambda^2}[f](x)=\int_0^{\lambda^2}\phi_R(x,\mu^2)U_R
    [f](\mu^2)~d\sigma_R(\mu^2),
\end{equation}
where $\phi_R, \sigma_R$ are defined in \eqref{phiLR}, \eqref{sigmaPrime}, respectively and the operator $U_R:C_0^\infty([0,b_R])\to C^\infty((0,1))$ is defined as 
\begin{equation}\label{UR}
U_R[f](\mu^2):=\int_0^{b_R}\phi_R(z,\mu^2)f(z)dz.
\end{equation}
For any interval $\Delta\subset [0,1]$, which is at a positive distance from $0$ and $1$ (to avoid the singularities of $\phi_R(z,\mu^2)$ as $\mu^2\to0$ or 1), we have
\begin{equation}\label{isoDelta}
 \big\Vert \hat{E}_{R,\Delta}[f]\big\Vert^2_{L^2([0,b_R])} =  \langle \hat{E}_{R,\Delta}[f](\cdot),\hat{E}_{R,\Delta}[f](\cdot)\rangle=\langle \hat{E}_{R,\Delta}^2[f](\cdot),f\rangle=\langle \hat{E}_{R,\Delta}[f](\cdot),f\rangle,\ f\in C_0^\infty([0,b_R]),
\end{equation}
where we have used that $\hat{E}_{R,\Delta}$ is a self-adjoint projection operator, see \cite{AG80} p.214.  Using \eqref{ResOfIdWithUR} and \eqref{UR}, it is easy to show that 
\begin{equation}\label{sigmafUR_v1}
    \langle \hat{E}_{R,\Delta}[f],f\rangle=\int_{\Delta}\left|U_R[f](\mu^2)\right|^2~d\sigma_R(\mu^2)=\big\Vert U_R[f]\big\Vert^2_{L^2(\Delta,\sigma_R)},\ f\in C_0^\infty([0,b_R]).
\end{equation}
By \eqref{isoDelta} and \eqref{sigmafUR_v1} and by continuity, we can extend $U_R$ to all of $L^2([0,b_R])$ and 
\begin{equation}\label{UR-prp}
 \big\Vert \hat{E}_{R,\Delta}[f]\big\Vert^2_{L^2([0,b_R])} =  \big\Vert U_R[f]\big\Vert^2_{L^2(\Delta,\sigma_R)},\ f\in L^2([0,b_R]).
\end{equation}
Taking the limit $\Delta \to [0,1]$ in \eqref{isoDelta}, \eqref{sigmafUR_v1} and using that $\hat{E}_{R,\lambda^2}$ is the resolution of the identity, the spectrum is confined to $[0,1]$, and there are no eigenvalues, we prove the following Lemma.

\begin{lemma}\label{isometry}
    The operator $U_R$ extends to an isometry from $L^2([0,b_R])\to L^2([0,1],\sigma_R)$.
    
\end{lemma}

We are now ready to conclude this section.

\begin{theorem}\label{SimpSpec}
    The spectrum of $\mathcal{H}_R^*\mathcal{H}_R$, $\mathcal{H}_L^*\mathcal{H}_L$ is simple and purely absolutely continuous.
\end{theorem}

\begin{proof}
   To prove that the spectrum is simple we will show that 
    \begin{equation}\label{generatingVec}
        g(x):=\chi_R(x)
    \end{equation}
    is a generating vector.  So for any $f\in L^2([0,b_R])$ we want to show that
\begin{equation}
\lim_{n\to\infty}    \bigg\Vert f-\sum_{j=1}^n\alpha_{jn}\hat{E}_{R,I_{jn}}[g](x) \bigg\Vert_{L^2([0,b_R])}=0,
\end{equation}
where $\alpha_{jn}$ are some constants, and $I_{jn}:=[(j-1)/n,j/n)$.  Thus, the intervals $I_{jn}$, $1\le j\le n$, partition the spectral interval $[0,1]$. Using the properties of $\hat{E}_{R,\lambda^2}$, we calculate 
\begin{align}\label{fByG}
    f(x)-\sum_{j=1}^n\alpha_{jn}\hat{E}_{R,I_{jn}}[g](x)&=\int_0^1\phi_R(x,\mu^2)\left\{U_R[f](\mu^2)-\tilde{\phi}_n(\mu^2)U_R[g](\mu^2)\right\}~d\sigma_R(\mu^2)
\end{align}
where $\tilde{\phi}_n$ is the simple function
\begin{equation}\label{tildePhiN}
    \tilde{\phi}_n(\mu^2)=\sum_{j=1}^n \alpha_{jn}\chi_{I_{jn}}(\mu^2).
\end{equation}
  Using the properties of $\hat{E}_{R,\lambda^2}$, we can write the left hand side of \eqref{fByG} as 
\begin{equation}\label{hatEhelp}
    \sum_{j=1}^n\hat{E}_{R,I_{jn}}[f-\alpha_{jn}g](x).
\end{equation}
Now using  \eqref{UR-prp}, \eqref{tildePhiN},\eqref{hatEhelp}, and the fact that $\hat{E}_{R,\Delta_j}\hat{E}_{R,\Delta_k}=0$ whenever $\Delta_j\cap\Delta_k=\emptyset$ (see \cite{AG80}, p.214), we see that
\begin{align}
    &&&\bigg\Vert f-\sum_{j=1}^n\alpha_{jn}\hat{E}_{R,I_{jn}}[g](x) \bigg\Vert_{L^2([0,b_R])}^2=\bigg\Vert \sum_{j=1}^n\hat{E}_{R,I_{jn}}[f-\alpha_{jn}g](x)\bigg\Vert_{L^2([0,b_R])}^2= \nonumber \\
    &&& \sum_{j=1}^n\big\Vert\hat{E}_{R,I_{jn}}[f-\alpha_{jn}g](x)\big\Vert_{L^2([0,b_R])}^2 
    =\sum_{j=1}^n\big\Vert U_R[f]-\alpha_{jn}U_R[g]\big\Vert_{L^2(I_{jn},\sigma_R)}^2 
    =\big\Vert U_R[f]-\tilde{\phi}_n U_R[g]\big\Vert_{L^2([0,1],\sigma_R)}^2,
\end{align}
since the intervals $I_{jn}$ are disjoint. Now our goal is to show that any $U_R[f]\in L^2([0,1],\sigma_R)$ can be approximated by $\tilde{\phi}_n U_R[g]$.  Using statement (6) of Proposition~\ref{dldr_AppendixProp} in Appendix \ref{dldr_appendix} and \eqref{DRDL}, \eqref{phiLR}, it follows that 
\begin{equation}
        U_R[g](\lambda^2)=\int_0^{b_R}\phi_R(x,\lambda^2)g(x)dx=1.
    \end{equation}
    It is clear that any $U_R[f]$ can be approximated by a sequence of simple function $\tilde{\phi}_n$, so we have 
    \begin{equation}
    \bigg\Vert f-\sum_{j=1}^n\alpha_{jn}\hat{E}_{R,I_{jn}}[g](x) \bigg\Vert_{L^2([0,b_R])}^2=\big\Vert U_R[f]-\tilde{\phi}_n\big\Vert_{L^2([0,1],\sigma_R)}^2\to0,\ n\to\infty,
\end{equation}
   as desired.  Thus, the spectrum of $\mathcal{H}_R^*\mathcal{H}_R$ is simple and $g=\chi_R$ is a generating vector.  \\
    
Lastly, to show that the spectrum of $\mathcal{H}^*_R\mathcal{H}_R$ is purely absolutely continuous, we need to show that the function
\begin{equation}\label{sigmaf}
    \sigma_f(\lambda^2):=\langle \hat{E}_{R,\lambda^2}[f],f\rangle
\end{equation}    
is absolutely continuous for any $f\in C_0^\infty([0,b_R])$ (such functions are dense in $L^2([0,b_R])$), see \cite{AG80}, Vol. 2, Section 95.  Similarly to the proof of Lemma~\ref{isometry}, we have
    \begin{equation}
        \sigma_f(\lambda^2)=\int_0^{\lambda_0^2}\left|U_R[f](\mu^2)\right|^2d\sigma_R(\mu^2)+\int_{\lambda_0^2}^{\lambda^2}\left|U_R[f](\mu^2)\right|^2d\sigma_R(\mu^2)
    \end{equation}
for any $\lambda_0\in (0,1)$. The operator $\mathcal{H}^*_R\mathcal{H}_R$ has no eigenvalues, so $\sigma_f(\lambda^2)$ is continuous. Hence it suffices to show that $d\sigma_f(\lambda^2)/d\lambda^2$ is continuous on any interval $[\epsilon,1-\epsilon]$, $0<\epsilon<1$. Since both $d\sigma_R(\mu^2)/d\mu^2$ and the kernel of $U_R$ are real analytic for $\mu^2\in(0,1)$ (see Proposition \ref{dldr_AppendixProp}) and $f\in C_0^\infty([0,b_R])$, the desired assertion follows immediately.

\end{proof}

\subsection{Diagonalization of $\mathcal{H}_R^*\mathcal{H}_R$ and $\mathcal{H}_L^*\mathcal{H}_L$}

We are now ready to use Theorem \ref{AGdiag} and build the unitary operators which will diagonalize $\mathcal{H}_R^*\mathcal{H}_R$ and $\mathcal{H}_L^*\mathcal{H}_L$.  Recall from Theorem \ref{SimpSpec} that $\chi_R, \chi_L$ are generating vectors for $\mathcal{H}_R^*\mathcal{H}_R, \mathcal{H}_L^*\mathcal{H}_L$, respectively.  Following Theorem \ref{AGdiag} and Remark \ref{diagShort}, we define $U_L^*:L^2([0,1],\sigma_{\chi_L})\to L^2([b_L,0])$ and ${U_R^*:L^2([0,1],\sigma_{\chi_R})\to L^2([0,b_R])}$ by
\begin{align}\label{ULstarURstar}
    U_L^*[\tilde{g}](y):=\int_0^1 \tilde{g}(\lambda^2)~d\hat{E}_{L,\lambda^2}[\chi_L](y), ~~~ U_R^*[\tilde{f}](x):=\int_0^1 \tilde{f}(\lambda^2)~d\hat{E}_{R,\lambda^2}[\chi_R](x),
\end{align}
where $\hat{E}_{L,\lambda^2},\hat{E}_{R,\lambda^2}$ are the resolutions of the identity for $\mathcal{H}_L^*\mathcal{H}_L, \mathcal{H}_R^*\mathcal{H}_R$, respectively, see \eqref{ER_ResOfId}, \eqref{EL_ResOfId}.  The spectral measures are $\sigma_{\chi_L}(\lambda^2):=\langle \hat{E}_{L,\lambda^2}[\chi_L],\chi_L\rangle$ and $\sigma_{\chi_R}(\lambda^2):=\langle \hat{E}_{R,\lambda^2}[\chi_R],\chi_R\rangle$.

\begin{remark}\label{sigmaRemark}
    Using Proposition \ref{dldr_AppendixProp}, it can be verified that 
    \begin{equation}
        \frac{d\sigma_{\chi_L}(\lambda^2)}{d\lambda^2}=\frac{d\sigma_L(\lambda^2)}{d\lambda^2}, ~~~ \frac{d\sigma_{\chi_R}(\lambda^2)}{d\lambda^2}=\frac{d\sigma_R(\lambda^2)}{d\lambda^2},
    \end{equation}
    where $\sigma_L'(\lambda^2), \sigma_R'(\lambda^2)$ (here $'$ denotes differentiation with respect to $\lambda^2$) are defined in \eqref{sigmaPrime}.
\end{remark}

\begin{remark}\label{ULURRemark}
    Again, using Proposition \ref{dldr_AppendixProp}, it can be verified that 
    \begin{align}
        U_L^*[\tilde{g}](y)=\int_0^1\phi_L(y,\lambda)\tilde{g}(\lambda^2)~d\sigma_L(\lambda^2), ~~~ U_R^*[\tilde{f}](x)=\int_0^1\phi_R(x,\lambda)\tilde{f}(\lambda^2)~d\sigma_R(\lambda^2),
    \end{align}
    where $\phi_L, \phi_R$ are defined in \eqref{phiLR}.  It is now clear that the adjoint of $U_R$, defined in \eqref{UR}, is $U_R^*$, and 
    \begin{equation}\label{UL}
        U_L[g](\lambda^2)=\int_{b_L}^0\phi_L(y,\lambda^2)g(y)~dy.
    \end{equation}
\end{remark}

We conclude this section with the following main result.

\begin{theorem}\label{HstarHdiagHHstardiag}

    The operators $U_R:L^2([0,b_R])\to L^2([0,1],\sigma_R)$, $U_L:L^2([b_L,0])\to L^2([0,1],\sigma_L)$, defined in \eqref{UR}, \eqref{UL}, respectively, are unitary and  
    \begin{equation}
        U_R\mathcal{H}_R^*\mathcal{H}_RU_R^*=\lambda^2, ~~ U_L\mathcal{H}_L^*\mathcal{H}_LU_L^*=\lambda^2
    \end{equation}
    in the sense of operator equality on $L^2([0,1],\sigma_R)$, $L^2([0,1],\sigma_L)$, respectively, where $\lambda^2$ is to be understood as a multiplication operator.
    
\end{theorem}

\begin{proof}
    
    We state the proof for $\mathcal{H}_R^*\mathcal{H}_R$ only as the proof for $\mathcal{H}_L^*\mathcal{H}_L$ is nearly identical.  The resolution of the identity of $\mathcal{H}_R^*\mathcal{H}_R$ was constructed in Theorem \ref{ResOfIdThm} and in Theorem \ref{SimpSpec} it was shown that the spectrum of $\mathcal{H}_R^*\mathcal{H}_R$ is simple and $\chi_R$ is a generating vector.  Notice that both $U_R^*$, defined in \eqref{ULstarURstar}, and $\sigma_{\chi_R}=\sigma_R$, defined in \eqref{sigmaPrime}, were constructed in accordance with Theorem \ref{AGdiag} and Remark \ref{diagShort}.  The combination of Remarks \ref{sigmaRemark}, \ref{ULURRemark} and Theorem \ref{AGdiag} complete the proof.

\end{proof}

We later obtain a different proof of this Theorem, see Corollary \ref{HLdiagCor} and Theorem \ref{diagEquiv}.

\section{Diagonalization of $\mathcal{H}_R, \mathcal{H}_L$ via Titchmarsh-Weyl Theory}\label{secODEDiag}

Using recent developments in the Titchmarsh-Weyl theory obtained in \cite{Ful08}, it was shown in \cite{KT16} that the operator
\begin{equation}\label{DiffOpL}
    Lf(x):=\left[P(x)f'(x)\right]'+2\left(x-\frac{b_R+b_L}{4}\right)^2f(x), ~~ P(x):=x^2(x-b_L)(x-b_R)
\end{equation}
has only continuous spectrum and commutes with the FHTs $\mathcal{H}_L,\mathcal{H}_R$, defined in \eqref{HLHRdef}.  We now state the main result of \cite{KT16} and refer the reader to this paper for more details. 
\begin{theorem}\label{KT16mainResult}
    The operators $U_1:L^2([b_L,0])\to L^2(J,\rho_1)$ and $U_2:L^2([0,b_R])\to L^2(J,\rho_2)$, where $J=\left[(b_L^2+b_R^2)/8,\infty\right)$, are isometric transformations.  Moreover, in the sense of operator equality on $L^2(J,\rho_2)$ one has
    \begin{equation}
        U_2\mathcal{H}_LU_1^*=\sigma(\omega),
    \end{equation}
    where
    \begin{equation}
        \sigma(\omega)=\frac{-b_R}{b_L\cosh(\mu(\omega)\pi)}\left(1+\BigO{\epsilon^{\frac{1}{2}-\delta}}\right), ~~~ \omega\to\infty,~~~  \epsilon=\omega^{-1/2},
    \end{equation}
    $\mu(\omega)=\sqrt{\frac{\omega-(b_L+b_R)^2/8}{-b_Lb_R}-\frac{1}{4}}$, $\rho_1'(\omega)=\frac{1}{b_L^2(b_R-b_L)}\left(1+\BigO{\epsilon^{\frac{1}{2}-\delta}}\right)$, $\rho_2'(\omega)=\frac{1}{b_R^2(b_R-b_L)}\left(1+\BigO{\epsilon^{\frac{1}{2}-\delta}}\right)$,  and $0<\delta<<1$ is fixed.
    
\end{theorem}

There is a minor typo in this theorem in \cite{KT16}; when describing $\sigma(\lambda)$, the factor $\frac{a_2^3}{a_1}$ is incorrect and should be $-\frac{a_2}{a_1}$.  The operators $U_1, U_2$ in Theorem \ref{KT16mainResult} were obtained asymptotically when $\omega\to\infty$.  Here we obtain these operators explicitly.  According to \cite{Ful08}, the kernels of $U_1, U_2$ and the spectral measures $\rho_1, \rho_2$ are defined through particular solutions of $Lf=\omega f$.  Such solutions will be constructed in the following subsections.

\subsection{Interval $[0,b_R]$}

Define the function
\begin{align}\label{f}
    f_R(x,\omega):=\frac{b_R(b_R-b_L)}{x(b_R+b_L)-2b_Rb_L}M_4(x)^{-\frac{1}{2}+i\mu}\pFq{2}{1}{\frac{1}{4}+\frac{i\mu}{2},\frac{3}{4}+\frac{i\mu}{2}}{1+i\mu}{M_4^2(x)},
\end{align}
where $M_4(x)$ and $\mu$ are defined in Remark \ref{mobius} and Theorem \ref{KT16mainResult}, respectively.  Notice that $M_4(x)$ maps $b_L\to-1, 0\to0, b_R\to1$.  If we take $b_R=-b_L=a$ in \eqref{f}, where $a$ is a constant, we obtain (4.9) of \cite{KT16}.  Now define 
\begin{align}\label{phi2theta2}
    \varphi_2(x,\omega):=kf_R(x,\omega)+\overline{k}\overline{f}_R(x,\omega), ~~~ \vartheta_2(x,\omega):=l_2f_R(x,\omega)+\overline{l}_2\overline{f}_R(x,\omega),
\end{align}
where 
\begin{align}
    k&:=\frac{\Gamma(-i\mu)}{\Gamma\left(\frac{1}{4}-\frac{i\mu}{2}\right)\Gamma\left(\frac{3}{4}-\frac{i\mu}{2}\right)}, ~~~ l_2:=\frac{k}{b_R^2(b_R-b_L)}\left[2\gamma+2\Psi\left(\frac{1}{2}-i\mu\right)+\ln\left(\frac{-b_L}{b_R(b_R-b_L)}\right)\right]. \label{kl2}
\end{align}
Here $\gamma$ is Euler's constant and $\Psi$ is the Digamma function, see \cite{AS64} 6.3.1.

\begin{remark}\label{kTwoMod}
    Using properties of the Gamma functions, see \cite{AS64} 6.1.30, it can be shown that
    \begin{equation}
        |k|^2=\frac{\coth(\mu\pi)}{2\pi\mu},
    \end{equation}
    provided $\mu\geq0$.
\end{remark}

\begin{remark}\label{phi2DR}
    Notice that if we take 
    \begin{equation}\label{muOmegaLambda}
        i\mu(\omega)=a+\frac{1}{2}\implies \omega=\frac{(b_R+b_L)^2}{8}+b_Lb_Ra(a+1)
    \end{equation}
    for $\lambda\in[-1,1]$ (this implies $\mu\in[0,\infty)$ iff $\omega\in[(b_L^2+b_R^2)/8,\infty)$), where $a:=a_-(-|\lambda|/2)$, it can be verified via \cite{AS64} 15.3.16 that 
    \begin{equation}\label{fhprime}
        \frac{-b_R\alpha(-|\lambda|/2)}{x\sqrt{\pi}}h_\infty'\left(M_1(x)\right)=kf_R(x,\omega), ~~ \frac{-b_R\beta(-|\lambda|/2)}{x\sqrt{\pi}}s_\infty'\left(M_1(x)\right)=\overline{k}\overline{f}_R(x,\omega)
    \end{equation}
    where $M_1(x)$ is defined in Remark \ref{mobius},  $\alpha,\beta$ and $h_\infty',s_\infty'$ are defined in \eqref{alphabeta} and \eqref{hhprime},\eqref{ssprime}, respectively.  This relation immediately implies that 
    \begin{equation}
        \varphi_2(x,\omega)=\frac{-b_R}{x\sqrt{\pi}}D_R(x,\lambda)
    \end{equation}
    for $\lambda\in[-1,1]$, see \eqref{DRDL} for $D_R(x,\lambda)$.
\end{remark}

\begin{theorem}\label{phi2theta2Prop}
    The functions $\varphi_2, \vartheta_2$ defined in \eqref{phi2theta2} satisfy the following properties:
    \begin{enumerate}
        \item For $x\in[0,b_R]$ and $\omega\in[(b_L^2+b_R^2)/8,\infty)$, $\varphi_2(x,\omega),\vartheta_2(x,\omega)$ are linearly independent solutions of $Lf=\omega f$, where $L$ is defined in \eqref{DiffOpL},
        \item $\varphi_2(x,\omega),\vartheta_2(x,\omega)\in\mathbb{R}$, for all $x\in[0,b_R]$, $\omega\in\mathbb{R}$,
        \item $P(x)\varphi_2'(x,\omega)\to0$ as $x\to b_R^-$, 
        \item $P(x)W_x(\vartheta_2(x,\omega),\varphi_2(x,\omega))=1$ for all $x\in[0,b_R]$, $\omega\in\mathbb{C}$,
        \item $\ds{\lim_{x\to b_R^-}P(x)W_x(\vartheta_2(x,\omega),\varphi_2(x,\omega'))=1}$ for all $\omega,\omega'\in\mathbb{C}$.
    \end{enumerate}
\end{theorem}

\begin{proof}
\begin{enumerate}
    \item By definition, $\varphi_2, \vartheta_2$ are linear combinations of $f_R$ and $\overline{f}_R$, which can be expressed in terms of $h_\infty', s_\infty'$, see Remark \ref{phi2DR}.  Recall that $h_\infty(x), s_\infty(x)$ are solutions of \eqref{hsODE}.  It can now be verified that $\frac{1}{x}h_\infty'\left(M_1(x)\right), \frac{1}{x}s_\infty'\left(M_1(x)\right)$ are solutions of
    \begin{equation}
        L[g](x)=\left[\frac{(b_L+b_R)^2}{8}+b_Lb_R\cdot a(a+1)\right]g(x),
    \end{equation}
    where $a:=a_-(-|\lambda|/2)$.  Thus $\varphi_2, \vartheta_2$ solve $Lf=\omega f$ for $x\in[0,b_R]$ and $\omega\in[(b_L^2+b_R^2)/8,\infty)$.  To show $\varphi_2, \vartheta_2$ are linearly independent, we compute their Wronskian.  It is a simple exercise to show that 
    \begin{equation}
        W_x[\vartheta_2,\varphi_2]=2i\Im[\overline{k}l_2]W_x[f_R,\overline{f}_R].
    \end{equation}
    Using Remark \ref{kTwoMod}, \cite{AS64} 6.3.12 and Remark \ref{phi2DR}, \eqref{hypergeoODE}, \eqref{appendix_detHatGamma}, it can be verified that
    \begin{equation}\label{l2k_fRWron}
        \Im[\overline{k}l_2]=\frac{-1}{2\mu b_R^2(b_R-b_L)}, ~~~ W_x[f_R,\overline{f}_R]=\frac{i\mu b_R^2(b_R-b_L)}{P(x)}.
    \end{equation}
    Thus $\varphi_2, \vartheta_2$ are linearly independent since $W_x[\vartheta_2,\varphi_2]=1/P(x)$.
    
    \item This is clear by the definition of $\vartheta_2,\varphi_2$, see \eqref{phi2theta2}.
    
    \item This follows from Remark \ref{phi2DR} and Proposition \ref{dldr_AppendixProp} because $D_R(x,\lambda)$ is analytic at $x=b_R$.  Moreover, using \cite{AS64} 15.3.6, it can be shown that $\varphi_2(b_R,\omega)=1$.
    
    \item We have previously shown that $P(x)W_x[\vartheta_2(x,\omega),\varphi_2(x,\omega)]=1$ for $x\in(0,b_R)$, ${\omega\in[(b_L^2+b_R^2)/8,\infty)}$.  This can be extended to the entire complex $\omega$ plane.
    
    \item Using \cite{AS64} 15.3.6, 15.3.10 and Remark \ref{phi2DR}, for $\omega\in((b_L^2+b_R^2)/8,\infty)$ (which implies $\mu\in(0,\infty)$), we have 
    \begin{align}
        \varphi_2(x,\omega)&=1+\BigO{x-b_R}, ~~~ \vartheta_2(x,\omega)=-2\Re[i\mu l_2\overline{k}]\ln(b_R-x)+\text{o}(1)
    \end{align}
    as $x\to b_R^-$.  Thus for $\omega,\omega'\in((b_L^2+b_R^2)/8,\infty)$ we have 
    \begin{align}\label{wron_xtobr}
        W_x\left[\vartheta_2(x,\omega),\varphi_2(x,\omega')\right]=\frac{2\Re[i\mu l_2\overline{k}]}{b_R-x}+\BigO{\ln(b_R-x)}
    \end{align}
    as $x\to b_R^-$.  So by \eqref{wron_xtobr},
    \begin{equation}
        \lim_{x\to b_R^-}P(x)W_x\left[\vartheta_2(x,\omega),\varphi_2(x,\omega')\right]=1,
    \end{equation}
    which holds for $\omega,\omega'\in((b_L^2+b_R^2)/8,\infty)$, and can be extended to any $\omega,\omega'\in\mathbb{C}$.
    \end{enumerate}    
\end{proof}

\subsection{Interval $[b_L,0]$}

This subsection will be similar to the last so many proofs will be omitted, as the ideas have been previously presented.  Define the function
\begin{align}\label{fL}
    f_L(x,\omega):=\frac{-b_L(b_R-b_L)}{x(b_R+b_L)-2b_Rb_L}(-M_4(x))^{-\frac{1}{2}+i\mu}\pFq{2}{1}{\frac{1}{4}+\frac{i\mu}{2},\frac{3}{4}+\frac{i\mu}{2}}{1+i\mu}{M_4^2(x)},
\end{align}
where $M_4(x)$ and $\mu(\omega)$ are defined in Remark \ref{mobius} and Theorem \ref{KT16mainResult}, respectively.  Now define
\begin{align}\label{phi1theta1}
    \varphi_1(x,\omega)&:=kf_L(x,\omega)+\overline{k}\overline{f}_L(x,\omega), ~~~ \vartheta_1(x,\omega):=l_1f_L(x,\omega)+\overline{l}_1\overline{f}_L(x,\omega),
\end{align}
where $k$ is defined in \eqref{kl2} and 
\begin{equation}\label{l1}
    l_1:=\frac{k}{b_L^2(b_R-b_L)}\left[2\gamma+2\Psi\left(\frac{1}{2}-i\mu\right)+\ln\left(\frac{-b_R}{b_L(b_R-b_L)}\right)\right].
\end{equation}

\begin{remark}\label{phi1DL}
    The functions $f_L, f_R$, defined in \eqref{fL}, \eqref{f}, respectively, share the relation
    \begin{equation}
        \frac{b_L^2f_R\left(M_2(x),\omega\right)}{x(b_R+b_L)-b_Lb_R}=f_L(x,\omega),
    \end{equation}
    where $M_2(x)$ is defined in \eqref{mobius}.  This relation combined with Remark \ref{phi2DR} shows that
    \begin{equation}
        \varphi_1(x,\omega)=\frac{b_LM_2(x)}{b_Rx}\varphi_2(M_2(x),\omega)=\frac{-b_L}{x\sqrt{\pi}}D_L(x,\lambda),
    \end{equation}
    where $D_L, \varphi_2$ are defined in \eqref{DRDL}, \eqref{phi2theta2}, respectively, and the relation between $\omega$ and $\lambda$ is described in \eqref{muOmegaLambda}.
\end{remark}

\begin{theorem}\label{phi1theta1Prop}
    The functions $\varphi_1, \vartheta_1$ defined in \eqref{phi2theta2} satisfy the following properties:
    \begin{enumerate}
        \item For $x\in[b_L,0]$ and $\omega\in[(b_L^2+b_R^2)/8,\infty)$, $\varphi_1(x,\omega), \vartheta_1(x,\omega)$ are linearly independent solutions of $Lf=\omega f$, where $L$ is defined in \eqref{DiffOpL},
        \item $\varphi_1(x,\omega),\vartheta_1(x,\omega)\in\mathbb{R}$, for all $x\in[b_L,0]$, $\omega\in\mathbb{R}$,
        \item $P(x)\varphi_1'(x,\omega)\to0$ as $x\to b_L^+$, 
        \item $-P(x)W_x(\vartheta_1(x,\omega),\varphi_1(x,\omega))=1$ for all $x\in[b_L,0]$, $\omega\in\mathbb{C}$,
        \item $\ds{\lim_{x\to b_L^+}-P(x)W_x(\vartheta_1(x,\omega),\varphi_1(x,\omega'))=1}$ for all $\omega,\omega'\in\mathbb{C}$.
    \end{enumerate}
\end{theorem}

\subsection{Diagonalization of $\mathcal{H}_L, \mathcal{H}_R$}\label{HLdiagSec}

According to the spectral theory developed in \cite{Ful08}, we have gathered nearly all the necessary ingredients to diagonalize $\mathcal{H}_L, \mathcal{H}_R$.  It remains to construct two functions $m_1(\omega)$ and $m_2(\omega)$ so that 
\begin{equation}
    \vartheta_1(x,\omega)+m_1(\omega)\varphi_1(x,\omega)\in L^2([b_L,0]), ~~~ \vartheta_2(x,\omega)+m_2(\omega)\varphi_2(x,\omega)\in L^2([0,b_R])
\end{equation}
whenever $\Im\omega>0$.  It can be verified that 
\begin{equation}\label{mj}
    m_j=-\frac{l_j}{k}, ~~ j=1,2
\end{equation}
where $l_1$ and $l_2, k$ are defined in \eqref{l1}, \eqref{kl2}, respectively.  The spectral measures $\rho_1,\rho_2$ are constructed via the formula 
\begin{equation}\label{SpecMeasFormula}
    \rho_j(\omega_2)-\rho_j(\omega_1)=\lim_{\epsilon\to0^+}\frac{1}{\pi}\int_{\omega_1}^{\omega_2}\Im m_j(s+i\epsilon)~ ds,
\end{equation}
for $j=1,2$ (see \cite{Ful08} for more details).  From \eqref{SpecMeasFormula} we obtain
\begin{align}\label{rho12Prime}
     \rho_1'(\omega)=\frac{\tanh(\mu\pi)}{b_L^2(b_R-b_L)}, ~~~ \rho_2'(\omega)=\frac{\tanh(\mu\pi)}{b_R^2(b_R-b_L)},
\end{align}
where we have used \eqref{mj}, Remark \ref{kTwoMod}, and \eqref{l2k_fRWron}.  Define the operators $U_1:L^2([b_L,0])\to L^2(J,\rho_1)$ and $U_2:L^2([0,b_R])\to L^2(J,\rho_2)$, where $J=\left(\frac{b_R^2+b_L^2}{8},\infty\right)$, as
\begin{align}
    U_1[f](\omega)&:=\int_{b_L}^0\varphi_1(x,\omega)f(x)~dx, ~~~ U_1^*[f](\omega)=\int_J\varphi_1(x,\omega)\tilde{f}(\omega)~d\rho_1(\omega), \label{U1} \\
    U_2[f](\omega)&:=\int_0^{b_R}\varphi_2(x,\omega)f(x)~dx, ~~~ U_2^*[\tilde{f}](\omega)=\int_J\varphi_2(x,\omega)\tilde{f}(\omega)~d\rho_2(\omega), \label{U2}
\end{align}
where $\varphi_1,\varphi_2$ are defined in \eqref{phi1theta1}, \eqref{phi2theta2}, respectively.  We are now ready to prove the main result of this section.

\begin{theorem}\label{HLdiag}
    The operators $U_1, U_2$, defined in \eqref{U1}, \eqref{U2}, are unitary and in the sense of operator equality on $L^2(J,\rho_2)$, where $J=\left((b_R^2+b_L^2)/8,\infty\right)$, one has 
    \begin{equation}
        U_2\mathcal{H}_LU_1^*=-\frac{b_R}{b_L}\emph{sech}(\mu\pi),
    \end{equation}
    where $\rho_2'$ is defined in \eqref{rho12Prime}.
\end{theorem}

\begin{proof}
    First, the operators $U_1, U_2$ are unitary by \cite{Ful08}.  According to Proposition \ref{dldr_AppendixProp} and Remark \ref{phi1DL}, 
\begin{align}\label{HLphi1}
    \mathcal{H}_L[\varphi_1](y,\omega)&=\mathcal{H}_L\left[\frac{-b_LD_L(x;\lambda)}{x\sqrt{\pi}}\right](y,\omega)=\frac{-b_L}{b_R}\text{sech}(\mu\pi)\varphi_2(x,\omega),
\end{align}
since \eqref{muOmegaLambda} implies that $|\lambda|=\text{sech}(\mu\pi)$.  Using \eqref{HLphi1}, we calculate that
\begin{align}
    \mathcal{H}_LU_1^*[\tilde{f}](x)&=U_2^*\left[\frac{-b_L\rho_1'(\omega)}{b_R\rho_2'(\omega)}\text{sech}(\mu\pi)\tilde{f}(\omega)\right](x)
\end{align}
for any $\tilde{f}\in L^2(J,\rho_1)$, which is equivalent to 
\begin{equation}
    U_2\mathcal{H}_LU_1^*=-\frac{b_R}{b_L}\text{sech}(\mu\pi).
\end{equation}
\end{proof}

Since the adjoint of $\mathcal{H}_L$ is $-\mathcal{H}_R$, we have an immediate Corollary.

\begin{corollary}\label{HLdiagCor}

    In the sense of operator equality on $L^2(J,\rho_1)$ one has
    \begin{equation}
        U_1\mathcal{H}_RU_2^*=\frac{b_L}{b_R}\emph{sech}(\mu\pi), ~~ U_1\mathcal{H}_L^*\mathcal{H}_LU_1^*=\emph{sech}^2(\mu\pi),
    \end{equation}
    and in the sense of operator equality on $L^2(J,\rho_2)$ one has
    \begin{equation}
        U_2\mathcal{H}_R^*\mathcal{H}_RU_2^*=\emph{sech}^2(\mu\pi).
    \end{equation}
\end{corollary}

\begin{proof}
    
    The proof follows quickly from Theorem \ref{HLdiag} because 
    \begin{align}
        (U_2\mathcal{H}_LU_1^*)^*=U_1\mathcal{H}_L^*U_2^*
    \end{align}
    and (what follows is the multiplication operator)
    \begin{align}
        \left(-\frac{b_R}{b_L}\text{sech}(\mu\pi)\right)^*=-\frac{b_R}{b_L}\text{sech}(\mu\pi)\cdot\frac{\rho_2'(\omega)}{\rho_1'(\omega)}=-\frac{b_L}{b_R}\text{sech}(\mu\pi).
    \end{align}
    
\end{proof}

This corollary can be used to recover Theorem \ref{HstarHdiagHHstardiag}.  We have now obtained two (seemingly) different diagonalizations of $\mathcal{H}^*_R\mathcal{H}_R$ and $\mathcal{H}^*_L\mathcal{H}_L$ in Theorem \ref{HstarHdiagHHstardiag} and Corollary \ref{HLdiagCor}.  We show that these diagonalizations are equivalent in the sense of change of spectral variable.

\begin{theorem}\label{diagEquiv}

The two diagonalizations of $\mathcal{H}^*_L\mathcal{H}_L$ obtained in Theorem \ref{HstarHdiagHHstardiag} and Corollary \ref{HLdiagCor} are equivalent; that is, 
\begin{equation}
    U_1^*\emph{sech}^2(\mu\pi)U_1=U_L^*\lambda^2U_L
\end{equation}
in the sense of operator equality on $L^2([b_L,0])$.  The operators $U_1,U_L$ are defined in \eqref{U1},\eqref{UL}, respectively and $\emph{sech}^2(\mu\pi),\lambda^2$ are to be understood as multiplication operators.  An identical statement about $U_2$ and $U_R$, defined in \eqref{U2}, \eqref{UR}, respectively, can be made.

\end{theorem}

\begin{proof}

    We will relate the operators $U_L,U_1$ by using the change of variable $\lambda\to\omega$ in \eqref{muOmegaLambda}, which implies that 
    \begin{equation}
        \text{sech}^2(\mu(\omega)\pi)=\lambda^2, ~~~  \frac{b_R(a+1/2)}{i\pi\lambda^2b_L(b_R-b_L)}d\lambda^2=d\rho_1(\omega).
    \end{equation}
    Now using this change of variable, we see that
    \begin{align}\label{UL*U1*}
        U_L^*[\tilde{f}](x)&=U_1^*[c(\omega)\tilde{f}(\text{sech}^2(\mu\pi)](x),
    \end{align}
    where $c(\omega)=-b_L\sqrt{\pi}|\lambda|D_R(\infty;\lambda)$.  Similarly, 
    \begin{align}\label{ULU1}
        U_L[f](\lambda^2)&=\frac{1}{c(\omega)}U_1[f](\omega).
    \end{align}
    So using \eqref{UL*U1*} and \eqref{ULU1} we obtain
    \begin{align}
        U_L^*\lambda^2U_L[f](x)&=U_1^*\text{sech}^2(\mu\pi)U_1[f](x),
    \end{align}
    as desired.

\end{proof}

\section{Small $\lambda$ asymptotics of $\Gamma(z;\lambda)$}\label{secGammaAsymp}

In this section we only consider the symmetric scenario when $b_R=-b_L=1$, but the results can be obtained for general endpoints via M\"obius transformations.  The main results of this section are Theorem \ref{GammaAsmptotics} and Theorem \ref{GammaAsympMainResult}, which describes the small $\lambda$ asymptotics of $\Gamma(z;\lambda)$ first in a small annulus around $z=0$ and then in the rest of $\C$ respectively. The proof of Theorem \ref{GammaAsmptotics} is based on the asymptotics of hypergeometric functions that appear in $\Gamma(z;\lambda)$, whereas the prove of Theorem \ref{GammaAsympMainResult} is based on Theorem \ref{GammaAsmptotics} and the Deift-Zhou nonlinear steepest descent method (\cite{DZ93}).

\begin{remark}
    Everywhere in this section we consider  $\lambda\in\mathbb{C}\setminus[-1/2,1/2]$, where $\l$ on upper/lower shores of $[-1/2,1/2]$ is also allowed; such values will be denoted by   $\lambda_\pm$ respectively.
\end{remark}

\subsection{Modified saddle point  method uniform with respect to parameters}

According to \eqref{RHPSolution}, we are interested in $h_\infty(\eta)$, where
\begin{equation}\label{eta}
    \eta=\frac{z+1}{2z}.
\end{equation}
In view of the integral representation (\cite{AS64}, 15.3.1) 
\begin{align}\label{hinf-int}
    h_\infty(\eta):=e^{a\pi i}\eta^{-a}\pFq{2}{1}{a,a+1}{2a+2}{\frac{1}{\eta}}=e^{a\pi i}\eta^{-a}\frac{\Gamma(2a+2)}{\Gamma(a+1)^2}\int_0^1\left(\frac{t(1-t)}{1-t/\eta}\right)^{a}dt
\end{align}
of $h_\infty(\eta)$, given by \eqref{hhprime}, where $a(\l)\to\infty$ as $\l\to 0$, we want to use the saddle point method to find the small $\l$ asymptotics of $h_\infty(\eta)$.  We start with the case $\Im\lambda\geq0$ which implies $\Im[a]\to-\infty$ as $\lambda\to0$, see Appendix \ref{aAppendix} for more information about $a(\lambda)$.  For $\Im\lambda\leq0$ the results are similar, see Remark \ref{lambdaLHP}.  With that in mind, define function
\begin{equation}\label{S}
S_\eta(t)=S\left(t,\eta\right):=-i\ln\left(\frac{t(1-t)}{1-\frac{t}{\eta}}\right)
\end{equation}
where the branch cuts of $S_\eta(t)$ in $t$ variable are chosen to be $(-\infty,0)$, $(1,\infty)$, and the ray from $t=\eta$ to $t=\infty$ with angle $\arg{\eta}$.  The integral from \eqref{hinf-int} can be now written as
\begin{align}\label{int2}
\int_0^1\left(\frac{t(1-t)}{1-\frac{t}{\eta}}\right)^{a(\lambda)}~dt=\int_0^1e^{i\Re[a(\lambda)]S_\eta(t)}e^{-\Im[a(\lambda)]S_\eta(t)}~dt.
\end{align}
Define closed regions
\begin{align}
    \Omega&:=\left\{\eta=\frac{z+1}{2z}:~M\leq\eta\leq2M\right\}, \label{Omega} \\
    \Omega_+&:=\left\{\eta=\frac{z+1}{2z}:~M\leq\eta\leq2M, ~ 0\leq\arg(\eta)\leq\pi\right\}, \label{OmegaPlus}
\end{align}
where $M$ is a large, positive, fixed number that is to be determined.  Notice that the set of all $z$ such that $\frac{z+1}{2z}\in\Omega$ is a small annulus about the origin.  The large $a$ asymptotics of the integral in \eqref{int2} that is uniform in  $\eta\in\Omega$ is technically not covered by standard saddle point theorems (see, for example, \cite{Olv74}, \cite{Tem14}, \cite{Fed87}). Therefore, in Appendix \ref{appendix_saddlePt} we present a proof of Theorem \ref{result} for such integrals, that will be used later for the small lambda asymptotics of hypergeometric functions $h_\infty(\eta), s_\infty(\eta)$ and their derivatives.  The obtained results in Theorem \ref{result}  leading order term of the hypergeometric function is consistent with the results of Paris \cite{Par13}, where the formal asymptotic expansion in the large parameter $a(\l)$ was derived, but the error estimates and uniformity in $\eta$ were not addressed.

The following proposition identifying the saddle points of $S_\eta(t)$ is a simple exercise. We need the saddle point $ t_-^*(\eta)$ to state Theorem \ref{result}.

\begin{proposition}\label{saddlepts}
    For $\eta\in\Omega_+$, the function $S_\eta(t)$ has exactly two simple saddle points $t_\pm^*(\eta)$ defined by $S_\eta'(t_\pm^*(\eta))=0$.  Explicitly, 
    \begin{align}
    t_+^*(\eta)&=\eta+\sqrt{\eta^2-\eta}=2\eta+\BigO{1} ~ \text{ as } \eta\to\infty, \\
    t_-^*(\eta)&=\eta-\sqrt{\eta^2-\eta}=\frac{1}{2}+\BigO{\eta^{-1}} ~ \text{ as } \eta\to\infty,
\end{align}
where the branchcut for $t_\pm^*(\eta)$ is $[0,1]$.  Moreover,
\begin{align}\label{Satsaddle}
S_\eta(t^*_\pm(\eta))=-2i\ln\left(t_\pm^*(\eta)\right) ~ \text{ and } ~~ S_\eta ''(t^*_-(\eta))=\frac{2i}{t^*_-(\eta)\left(1-t^*_-(\eta)\right)}.
\end{align}
\end{proposition}

Let $B(\zeta,r)$, $B^0(\zeta,r)$ denote an open disc and a punctured open disc
respectively of radius $r>0$ 
centered at $\zeta\in \C$.

\begin{theorem}\label{result}
Fix a sufficiently small $\epsilon>0$, a sufficiently large $M$ (see \eqref{Omega}) and suppose $F(t,\eta,\lambda)$ satisfies the following properties:
\begin{enumerate}
    \item  For every $(\eta,\lambda)\in\Omega_+\times {B}^0(0,\epsilon)$
    , $F(t,\eta,\lambda)$ is analytic in   $t\in B(1/2,1/2)$;
    \item  $F(t,\eta,\lambda)$ is continuous in all variables   in $B(1/2,1/2)\times \Omega_+\times {B^0}(0,\epsilon)$ and 
        for every $t\in B(1/2,1/2)$ it is bounded in 
    $(\eta,\lambda)\in\Omega_+\times {B^0}(0,\epsilon)$;
     \item   $F(t,\eta,\lambda)=\BigO{t^{c_0}}$ as $t\to0$, where $c_0>-1$   uniformly in  $(\eta,\lambda)\in\Omega_+\times {B^0}(0,\epsilon)$;
        \item  $F(t,\eta,\lambda)=\BigO{(1-t)^{c_1}}$ as $t\to1$, where $c_1>-1$   uniformly in  $(\eta,\lambda)\in\Omega_+\times {B^0}(0,\epsilon)$;
    \item $|F(t_-^*(\eta),\eta,\lambda)|$ is separated from zero   for all $(\eta,\lambda)\in\Omega_+\times {B^0}(0,\epsilon)$.
\end{enumerate}
Then    
\begin{align}
    &\int_0^1F(t,\eta,\lambda)e^{-\Im[a(\lambda)]S_\eta(t)}dt=e^{-\Im[a(\lambda)]S_\eta(t^*_-(\eta))}F\left(t^*_-(\eta),\eta,\lambda\right)\sqrt{\frac{2\pi}{\Im[a(\lambda)]S''_\eta(t^*_-(\eta))}}\left[1+\BigO{\frac{M^2}{\Im[a(\lambda)]}}\right], \label{int}
\end{align}
as $\lambda\to0$, where  $a(\lambda)$ is defined in \eqref{aFunc}.  This approximation is uniform for $\eta\in\Omega_+$.
\end{theorem}

The idea of the proof is as follows:  we deform the contour of integration in \eqref{int} from $[0,1]$ to a path we call $\gamma_\eta$, which passes through a relevant saddle point $t_-^*(\eta)$ of $S_\eta(t)$.  We then show that the leading order contribution in \eqref{int} comes from a small neighborhood of $t_-^*(\eta)$.

\begin{remark}
 One can simplify equation \eqref{int} in Theorem \ref{result} by substituting \eqref{Satsaddle}.
\end{remark}

\subsection{Small $\lambda$ asymptotics of $\Gamma(z;\lambda)$ for $z\in\tilde{\Omega}$}

In this subsection we use Theorem \ref{result} to calculate the leading order asymptotics of $\Gamma(z;\lambda)$ given by \eqref{RHPSolution}
as $\l\to 0$, provided that $z\in\tilde{\Omega}$, where
\begin{align}\label{OmegaTilde}
    \tilde{\Omega}&:=\left\{z:~\frac{z+1}{2z}\in\Omega\right\}, ~~~ \tilde{\Omega}_+:=\left\{z:~\frac{z+1}{2z}\in\Omega_+\right\}.
\end{align}
Notice that $\tilde{\Omega}$ is a small annulus about the origin.  In this section, we will often use the variable 
\begin{equation}
    \varkappa=-\ln\lambda
\end{equation}
instead of $\l$ and the function
\begin{equation}\label{gFunc}
g(z):=a\left(-\frac{z}{2}\right)=\frac{1}{i\pi}\ln\left(\frac{i+\sqrt{z^2-1}}{-z}\right).
\end{equation}
The properties of the $g$-function can be found in Proposition \ref{gproperties}.

\begin{corollary}\label{hhprimessprime}
    In the limit $\lambda=e^{-\varkappa}\to0$
\begin{align}
    h_\infty\left(\frac{z+1}{2z}\right)&=
    i4^a\frac{\sqrt{2z}(1-z^2)^{1/4}}{1+\sqrt{1-z^2}}e^{\pm \varkappa\left(\frac1{2} -g(z)\right) }\left(1+\BigO{\frac{M^2}{\varkappa}}\right), \\
    h_\infty'\left(\frac{z+1}{2z}\right)&=\mp
    4^a\frac{\varkappa(2z)^{\frac{3}{2}}e^{\pm\varkappa\left(\frac{1}{2}-g(z)\right)}}{\pi(1-z^2)^{1/4}(1+\sqrt{1-z^2})}\left(1+\BigO{\frac{M^2}{\varkappa}}\right), \\
    s_\infty\left(\frac{z+1}{2z}\right)&=
    i4^{-a}\frac{(1-z^2)^{1/4}(1+\sqrt{1-z^2})}{(2z)^{\frac{3}{2}}}e^{\pm\varkappa\left( g(z)-\frac{1}{2}\right)}\left(1+\BigO{\frac{M^2}{\varkappa}}\right), \\
    s_\infty'\left(\frac{z+1}{2z}\right)&=\pm
    4^{-a}\frac{\varkappa(1+\sqrt{1-z^2})}{\pi\sqrt{2z}(1-z^2)^{1/4}}e^{\pm\varkappa\left(g(z)-\frac{1}{2}\right)}\left(1+\BigO{\frac{M^2}{\varkappa}}\right),      
\end{align}
provided $\pm\Im\lambda\geq0$, where each approximation is uniform in $z\in\tilde{\Omega}_+$.  The functions $h_\infty,h_\infty'$ and $s_\infty,s_\infty'$ are defined in \eqref{hhprime},\eqref{ssprime}, respectively.  The functions $\sqrt{1-z^2}$ and $(1-z^2)^{1/4}$ have branch cuts on $(-1,1)$ and $(-\infty,1)$ respectively.
\end{corollary}

\begin{proof}
According to Theorem \ref{result},
\begin{equation}\label{hInfproof}
h_\infty\left(\frac{z+1}{2z}\right)=\frac{e^{a\pi i}\Gamma(2a+2)}{\Gamma^2(a+1)}\sqrt{\frac{\pi}{a}}\left(\frac{1-z}{1+z}\right)^{\frac{1}{4}}\left(\frac{2z}{z+1}\right)^{a}\left(1+\sqrt{\frac{1-z}{1+z}}\right)^{-1-2a}\left[1+\BigO{\frac{M^2}{a}}\right],
\end{equation}
where we have used   $F(t,\eta,\lambda)=e^{i\Re[a(\lambda)]S_\eta(t)}$ and $\eta=\frac{z+1}{2z}$. 
By  Proposition \ref{aProp}
\begin{align}
\left(\frac{z+1}{2z}\right)^{-a}\left(1+\sqrt{\frac{1-z}{1+z}}\right)^{-2a}&=\exp\left[-a\pi ig(z)-\frac{a\pi i}{2}\right] \\
&=\frac{\sqrt{z(1+z)}}{\sqrt{2}(1+\sqrt{1-z^2})}e^{-\varkappa g(z)-\varkappa/2}\left(1+\BigO{\lambda^2}\right) \label{hInfeq1}
\end{align}
as $\lambda\to0$ with $\Im\lambda\geq0$ and $z\in\tilde{\Omega}_+$.  Using \cite{DLMF} 5.5.5, 5.11.13 we have
\begin{equation}\label{hInfeq2}
\frac{\Gamma(2a+2)}{\Gamma(a+1)^2}=\frac{4^{a+1/2}}{\sqrt{\pi}}\cdot\frac{\Gamma\left(a+\frac{3}{2}\right)}{\Gamma(a+1)}=4^{a+1/2}\sqrt{\frac{a}{\pi}}\left(1+\BigO{\frac{1}{a}}\right)
\end{equation}
as $a\to -i\infty$.  Now plugging \eqref{hInfeq1}, \eqref{hInfeq2} into \eqref{hInfproof}, we obtain the result for $h_\infty\left(\frac{z+1}{2z}\right)$ when $\Im\lambda\geq0$.  The approximation when $\Im\lambda\leq0$ can be found in an similar manner.

For $h_\infty'(\eta)$ we  again  use Theorem \ref{result} with $F(t,\eta,\lambda)=\frac{e^{i\Re[a(\lambda)]S_\eta(t)}}{1-t/\eta}$ to obtain
\begin{align}
h_\infty'(\eta)&=-e^{a\pi i}\frac{a\Gamma(2a+2)}{\eta^{a+1}\Gamma(a+1)^2}\int_0^1\frac{1}{1-t/\eta}\left(\frac{t(1-t)}{1-t/\eta}\right)^adt \cr
&=-2ae^{a\pi i}4^a\eta^{-a-1}\left(1-\frac{1}{\eta}\right)^{-1/4}\left(1+\sqrt{1-\frac{1}{\eta}}\right)^{-1-2a(\lambda)}\left(1+\BigO{\frac{M^2}{a}}\right) \label{hInfPrimeAsymp},
\end{align}
which is equivalent to the stated result.  Notice that the functions $s_\infty(\eta),s_\infty'(\eta)$, as written, only have the integral representation \cite{AS64} 15.3.1 for $-1/2\leq\Re[a(\lambda)]<0$.  In this case, we obtain the stated result immediately via the observation ${h_\infty(\eta)\big|_{a\to-a-1}=s_\infty(\eta)}$.  To obtain the results when $0\leq\Re[a(\lambda)]\leq1/2$, use \cite{DLMF} 15.5.19 with $z\to1/\eta$, $a\to-a$, $b\to-a-1$, and $c\to-2a$ to obtain
\begin{align}
&\frac{a(a-1)}{\eta}\left(1-\frac{1}{\eta}\right)\pFq{2}{1}{-a+2,-a+1}{-2a+2}{\frac{1}{\eta}}+2a(2a-1)\left(1-\frac{1}{\eta}\right)\pFq{2}{1}{-a+1,-a}{-2a+1}{\frac{1}{\eta}} \nonumber \\
&=2a(2a-1)\pFq{2}{1}{-a,-a-1}{-2a}{\frac{1}{\eta}}.
\end{align}
Now with $z\to1/\eta$, $a\to-a+1$, $b\to-a$, and $c\to-2a+1$, we have
\begin{align}
&\frac{a(a-2)}{\eta}\left(1-\frac{1}{\eta}\right)\pFq{2}{1}{-a+3,-a+2}{-2a+3}{\frac{1}{\eta}}+2(1-a)\left(1-2a+\frac{2(a-1)}{\eta}\right)\pFq{2}{1}{-a+2,-a+1}{-2a+2}{\frac{1}{\eta}} \nonumber \\
&=2(1-a)(1-2a)\pFq{2}{1}{-a+1,-a}{-2a+1}{\frac{1}{\eta}}.
\end{align}
Combining the two previous equations, we see that 
\begin{align}
\pFq{2}{1}{-a,-a-1}{-2a}{\frac{1}{\eta}}&=\frac{a(a-2)}{2\eta(a-1)(2a-1)}\left(1-\frac{1}{\eta}\right)^2\pFq{2}{1}{-a+3,-a+2}{-2a+3}{\frac{1}{\eta}} \nonumber \\
&+a\left(1-\frac{1}{\eta}\right)\left[2(2a-1)+\frac{3(1-a)}{\eta}\right]\pFq{2}{1}{-a+2,-a+1}{-2a+2}{\frac{1}{\eta}}.
\end{align}
The perk of this equation is that the right hand side has an integral representation for ${-1/2\leq\Re[a]\leq1/2}$.  Thus we can apply Theorem \ref{result} twice and obtain the leading order asymptotics.  So we have shown
\begin{align}
s_\infty(\eta)&=-e^{-a\pi i}\eta^{a+1}\pFq{2}{1}{-a-1,-a}{-2a}{\frac{1}{\eta}} \cr
&=-\frac{1}{2}e^{-a\pi i}4^{-a}\eta^{a+1}\left(1+\sqrt{1-\frac{1}{\eta}}\right)^{1+2a}\left(1-\frac{1}{\eta}\right)^{1/4}\left[1+\BigO{\frac{M^2}{a}}\right].
\end{align}
A similar process can be repeated for $s_\infty '(\eta)$ and we obtain
\begin{align}
    s_\infty'(\eta)&=\frac{a+1}{-e^{a\pi i}}\eta^a\pFq{2}{1}{-a,-a}{-2a}{\frac{1}{\eta}} \cr
    &=-2(a+1)e^{-a\pi i}4^{-a-1}\eta^{a}\left(1-\frac{1}{\eta}\right)^{-1/4}\left(1+\sqrt{1-\frac{1}{\eta}}\right)^{1+2a(\lambda)}\left[1+\BigO{\frac{M^2}{a}}\right]. \label{sInfPrimeAsymp}
\end{align}

\end{proof}

We have an immediate Corollary.

\begin{corollary}\label{GammaHat}
In the limit $\lambda=e^{-\varkappa}\to0$,
\begin{align}
    \hat{\Gamma}\left(\frac{z+1}{2z}\right)&=i(1-z^2)^{\sigma_3/4}\begin{bmatrix} 1 & 0 \\ 0 & \pm\frac{2z\varkappa}{i\pi} \end{bmatrix}(\1+i\sigma_2)\left(\frac{\sqrt{2z}}{1+\sqrt{1-z^2}}\right)^{\sigma_3}\begin{bmatrix} 1 & 0 \\ 0 & \frac{1}{2z} \end{bmatrix}\times \nonumber \\
    &\times\left(\1+\BigO{\frac{M^2}{\varkappa}}\right)4^{a\sigma_3}e^{\pm\varkappa\left(\frac{1}{2}-g(z)\right)\sigma_3},
\end{align}
provided $\pm\Im\lambda\geq0$, which is uniform for $z\in\tilde{\Omega}_+$.

\end{corollary}

It remains to find the small $\lambda$ leading order asymptotics of the remaining factors of $\Gamma(z;\lambda)$.  This is a tedious, but straightforward exercise.

\begin{lemma}\label{hatGammaOneHalf}
    
    We have 
    \begin{equation}
        \hat{\Gamma}^{-1}\left(\frac{1}{2}\right)=e^{-\frac{a\pi i}{2}\sigma_3}4^{-a\sigma_3}\begin{bmatrix} \frac{\Gamma\left(-\frac{a}{2}\right)\Gamma\left(-\frac{1}{2}-a\right)}{4\Gamma\left(\frac{1}{2}-\frac{a}{2}\right)\Gamma(-1-a)} & \frac{i\Gamma\left(\frac{1}{2}-\frac{a}{2}\right)\Gamma\left(-\frac{1}{2}-a\right)}{8\Gamma\left(1-\frac{a}{2}\right)\Gamma(-a)} \\ \frac{i\Gamma(\frac{a}{2}+\frac{1}{2})\Gamma(a+\frac{1}{2})}{\Gamma(\frac{a}{2}+1)\Gamma(a)} & \frac{-\Gamma\left(\frac{a}{2}+1\right)\Gamma\left(a+\frac{1}{2}\right)}{\Gamma\left(\frac{a}{2}+\frac{1}{2}\right)\Gamma(a+2)} \end{bmatrix}.
    \end{equation}  
    Moreover, in the limit $\lambda=e^{-\varkappa}\to0$,
    \begin{equation}
        \hat{\Gamma}^{-1}\left(\frac{1}{2}\right)=\sqrt{\frac{2}{i}}e^{\mp\frac{\varkappa}{2}\sigma_3}4^{-a\sigma_3}\begin{bmatrix}\frac{1}{4} & 0 \\ 0 & 1 \end{bmatrix}\left(\1+i\sigma_2\right)\left(\1+\BigO{\varkappa^{-1}}\right)\begin{bmatrix} 1 & 0 \\ 0 & \pm \frac{\pi}{2\varkappa} \end{bmatrix},
    \end{equation}
    provided $\pm\Im\lambda\geq0$.
\end{lemma}

\begin{proof}

Using \cite{AS64} 15.3.15 then 15.1.20, we have 
    \begin{align}
        \pFq{2}{1}{a,a+1}{2a+2}{2}=(-1)^{-a/2}\pFq{2}{1}{\frac{a}{2},\frac{a}{2}+1}{a+\frac{3}{2}}{1}&=e^{-a\pi i/2}\frac{\sqrt{\pi}\Gamma\left(a+\frac{3}{2}\right)}{\Gamma\left(\frac{a}{2}+\frac{3}{2}\right)\Gamma\left(\frac{a}{2}+\frac{1}{2}\right)},
    \end{align}
    where  
    $(-1)^{-a/2}=e^{-a\pi i/2}$. Thus, in view of Theorem \ref{ThmRHPSol} and   \eqref{hhprime}, we  obtain the  (1,1) entry of $\hat{\Gamma}\left(\frac{1}{2}\right)$.  Repeating this process for $h_\infty',s_\infty,s_\infty'$, we obtain our explicit result.  The asymptotics directly follow from the use of Stirling's formula and Proposition \ref{aProp}.
    
\end{proof}

According to Theorem \ref{ThmRHPSol},
\begin{equation}
Q(\lambda)=D(\1+ie^{a\pi i}\sigma_2),~~~\text{
where}~~~
D:=\begin{bmatrix} -\tan(a\pi) & 0 \\ 0 & 4^{2a+1}e^{a\pi i}\frac{\Gamma(a+3/2)\Gamma(a+1/2)}{\Gamma(a)\Gamma(a+2)} \end{bmatrix}.
\end{equation}
Combining that with Proposition \ref{aProp} and  5.11.13 from \cite{DLMF}, we obtain the following Lemma.

\begin{lemma}\label{Dasymp}
In the limit $\lambda=e^{-\varkappa}\to0$,
\begin{equation}
    D=i\begin{bmatrix} \pm1 & 0 \\ 0 & 4^{2a+1}e^{\pm\varkappa} \end{bmatrix}\left(\1+\BigO{\varkappa^{-1}}\right),
\end{equation}
provided $\pm\Im\lambda\geq0$.
\end{lemma}

We are ready to put the pieces from this section together and obtain the asymptotics of $\Gamma(z;\lambda)$ as $\lambda\to0$ for $z\in\tilde{\Omega}$.  Define the matrix 
\begin{align}\label{Phi}
\Phi(z)&:=\frac{1}{2\sqrt{z}(z^2-1)^{1/4}}\begin{bmatrix} i+z+\sqrt{z^2-1} & -i-z+\sqrt{z^2-1} \\ i-z+\sqrt{z^2-1} & -i+z+\sqrt{z^2-1}\end{bmatrix}=\left(\1+\frac{i}{2z}(\sigma_3-i\sigma_2)\right)\left(\frac{z^2-1}{z^2}\right)^{\sigma_1/4}.
\end{align}
This matrix is a particular solution of 
 the so-called model RHP  \ref{modelRHP}  corresponding to $x=y=i/2$ in \eqref{solmod}. 
 As we will see in Theorem \ref{GammaAsmptotics} below, the limit of $\Gamma(z;\lambda)$ as $\lambda\to0$, $\Im \lambda >0$ distinguishes $\Phi(z)$ among 
 all other solutions of the model RHP.

\begin{lemma}\label{leadingOrderModelProb}
In the limit $\lambda=e^{-\varkappa}\to0$
\begin{equation}
    D^{-1}\hat{\Gamma}^{-1}\left(\frac{1}{2}\right)\begin{bmatrix} 1 & \frac{-1}{2z(a+1)} \\ 0 & 1 \end{bmatrix}\hat{\Gamma}\left(\frac{z+1}{2z}\right)e^{\varkappa g(z)\sigma_3}D=
    \begin{cases}
        \Phi(z)\left(\1+\BigO{\frac{M^2}{\varkappa}}\right), &\Im\lambda\geq0 \\
        \sigma_3\Phi(z)\sigma_3\left(\1+\BigO{\frac{M^2}{\varkappa}}\right), &\Im\lambda\leq0
    \end{cases}
\end{equation}
uniformly in  $z\in\tilde{\Omega}_+$.

\end{lemma}

\begin{proof}

First take $\Im\lambda\geq0$; the leading order term of $D^{-1}\hat{\Gamma}^{-1}\left(\frac{1}{2}\right)\begin{bmatrix} 1 & \frac{-1}{2z(a+1)} \\ 0 & 1 \end{bmatrix}\hat{\Gamma}\left(\frac{z+1}{2z}\right)e^{\varkappa g\sigma_3}D$ from Lemmas and Corollaries \ref{GammaHat}, \ref{hatGammaOneHalf}, \ref{Dasymp}, we have 
\begin{align*}
    &\frac{\sqrt{2i}}{4}(\1+i\sigma_2)\begin{bmatrix} 1 & -1 \\ 0 & -iz \end{bmatrix}(1-z^2)^{\sigma_3/4}(\1+i\sigma_2)\left(\frac{\sqrt{2z}}{1+\sqrt{1-z^2}}\right)^{\sigma_3}\begin{bmatrix} 1 & 0 \\ 0 & \frac{2}{z} \end{bmatrix} \\
    &=\frac{(z^2-1)^{-1/4}}{2z^{3/2}}\left[\left(\1-\sigma_1\right)+iz\left(\sigma_3-i\sigma_2\right)+\sqrt{1-z^2}\left(\sigma_3+i\sigma_2\right)\right]\cdot\left[\1-\sigma_3\sqrt{1-z^2}\right] \\
    &=\Phi(z).
\end{align*}
When $\Im\lambda\leq0$, observe that the leading order term of $D^{-1}\hat{\Gamma}^{-1}\left(\frac{1}{2}\right)\begin{bmatrix} 1 & \frac{-1}{2z(a+1)} \\ 0 & 1 \end{bmatrix}\hat{\Gamma}\left(\frac{z+1}{2z}\right)e^{\varkappa g\sigma_3}D$ is now 
\begin{equation}
    \sigma_3\cdot\frac{\sqrt{2i}}{4}(\1+i\sigma_2)\begin{bmatrix} 1 & -1 \\ 0 & -iz \end{bmatrix}(1-z^2)^{\sigma_3/4}(\1+i\sigma_2)\left(\frac{\sqrt{2z}}{1+\sqrt{1-z^2}}\right)^{\sigma_3}\begin{bmatrix} 1 & 0 \\ 0 & \frac{2}{z} \end{bmatrix}\cdot\sigma_3
\end{equation}
and thus we have the leading order term.  Since $\Phi(z)$ is uniformly bounded away from $0$ when $z\in\tilde{\Omega}$, we immediately obtain the lower order term.

\end{proof}

Define matrix 
\begin{align}\label{Psi}
\Psi_0(z;\varkappa)=\begin{cases}
\Phi(z), &\text{ for } \Im[\varkappa]\leq0, \\
\sigma_1\Phi(z)\sigma_1, &\text{ for } \Im[\varkappa]\geq0.
\end{cases}
\end{align}
Note that the matrix $\sigma_1\Phi(z)\sigma_1$ is also a solution to RHP \ref{modelRHP} with  $x=y=-i/2$ in \eqref{solmod}. 
Now we are ready to prove one of the main results of this section.

\begin{theorem}\label{GammaAsmptotics}

Let $\theta\in(0,\pi/2)$ be fixed. Then
\begin{equation}
\Gamma(z;\lambda) =
  \begin{cases}
   \Psi_0(z;\varkappa)\left(\1+\BigO{\frac{M^2}{\varkappa}}\right)e^{-\varkappa g(z)\sigma_3}, ~ &z\in\tilde{\Omega}, \theta<|\arg(z)|<\pi-\theta, \\
   
   \Psi_0(z;\varkappa)\left(\1+\BigO{\frac{M^2}{\varkappa}}\right)\begin{bmatrix} 1 & 0 \\ \pm ie^{\varkappa(2g(z)-1)} & 1 \end{bmatrix}e^{-\varkappa g(z)\sigma_3}, &z\in\tilde{\Omega}, ~ \pi-\theta<\arg(z)<\pi+\theta, \\
   
   \Psi_0(z;\varkappa)\left(\1+\BigO{\frac{M^2}{\varkappa}}\right)\begin{bmatrix} 1 & \mp ie^{-\varkappa(2g(z)+1)} \\ 0 & 1 \end{bmatrix}e^{-\varkappa g(z)\sigma_3}, &z\in\tilde{\Omega}, ~ -\theta<\arg(z)<\theta
  \end{cases} \nonumber 
\end{equation}
as   $\lambda=e^{-\varkappa}\to0$ uniformly in $z\in\tilde{\Omega}$, 
provided $\pm\Im\lambda\geq0$.  See Figure \ref{figLenseOmegaTilde} for $\theta, \tilde{\Omega}$.
\end{theorem}

\begin{proof}

First assume $\Im\lambda\geq0$ and $z\in\tilde{\Omega}_+$ (this implies $z\in\tilde{\Omega}$ and $\Im z\leq0$).  The following calculation
\begin{align}
e^{-\varkappa g\sigma_3}Q&=e^{-\varkappa g\sigma_3}D\left(\1+ie^{a\pi i}\sigma_2\right)e^{-\varkappa g\sigma_3}e^{\varkappa g\sigma_3}=D\left(e^{-2\varkappa g\sigma_3}+ie^{a\pi i}\sigma_2\right)e^{\varkappa g\sigma_3}
\end{align}
and use of Lemma \ref{leadingOrderModelProb} give us
\begin{align}
\Gamma(z;\lambda)&=\frac{\sigma_2}{1+e^{2a\pi i}}(\1-ie^{a\pi i}\sigma_2)D^{-1}\hat{\Gamma}^{-1}(\infty)\begin{bmatrix} 1 & \frac{-1}{2z(a+1)} \\ 0 & 1 \end{bmatrix}\hat{\Gamma}\left(\frac{z+1}{2z}\right)e^{\varkappa g\sigma_3}D(\sigma_2e^{2\varkappa g\sigma_3}+ie^{a\pi i}I)e^{-\varkappa g\sigma_3} \nonumber \\
&=\Phi(z)\left(\1+\BigO{\frac{M^2}{\varkappa}}\right)\begin{bmatrix} 1 & ie^{-\varkappa(2g+1)} \\ -ie^{\varkappa(2g-1)} & 1 \end{bmatrix}e^{-\varkappa g\sigma_3},
\end{align}
as desired.  Now take $\Im\lambda\leq0$ and proceed similar to above.  We use the calculation 
\begin{align}
e^{\varkappa g\sigma_3}Q&=e^{\varkappa g\sigma_3}D(\1+ie^{a\pi i}\sigma_2)e^{-\varkappa g\sigma_3}e^{\varkappa g\sigma_3}=D(\1+ie^{a\pi i}e^{2\varkappa g\sigma_3}\sigma_2)e^{\varkappa g\sigma_3}
\end{align}
and Lemma \ref{leadingOrderModelProb} to obtain
\begin{align}
\Gamma(z;\lambda)&=\sigma_1\Phi(z)\sigma_1\left(\1+\BigO{\frac{M^2}{\varkappa}}\right)\begin{bmatrix} 1 & ie^{-\varkappa(2g+1)} \\ -ie^{\varkappa(2g-1)} & 1 \end{bmatrix}e^{-\varkappa g\sigma_3}.
\end{align}
The results for $z\in\tilde{\Omega}$ with $\Im z\geq0$ are immediate via use of the symmetry $\overline{\Gamma(\bar{z};\bar{\lambda})}=\Gamma(z;\lambda)$.  Recall that ${\overline{\Phi(\bar{z})}=\sigma_1\Phi(z)\sigma_1}$ from Remark \ref{modelProbSym}, $\overline{g(\bar{z})}=g(z)$ and $\varkappa(\lambda)=\overline{\varkappa(\bar{\lambda})}$.

\end{proof}

\subsection{Deift-Zhou steepest descent method}

The $g$-function, defined in \eqref{gFunc}, will play an important role so we list its relevant properties, all of which follow directly from 
Proposition \ref{aProp}.

\begin{proposition}\label{gproperties}
$g(z)$ has the following properties:
\begin{enumerate}
\item $g(z)$ is analytic on $\mathbb{\bar C}\setminus [-1,1]$,  Schwarz symmetric and  $g(\infty)=0$,
\item $g_+(z)+g_-(z)=1$ for $z\in [-1,0]$,  \,  $g_+(z)+g_-(z)=-1$ for $z\in [0,1]$,
\item $\Re{(2g(z)-1)}=0$ for $z\in[-1,0]$ and $\Re{(2g(z)-1)}<0$ for $z\in\overline{\mathbb{C}}\setminus[-1,0]$,
\item $\Re{(2g(z)+1)}=0$ for $z\in[0,1]$ and $\Re{(2g(z)+1)}>0$ for $z\in\overline{\mathbb{C}}\setminus[0,1]$,
\end{enumerate}
\end{proposition}

\subsubsection{Transformation $\Gamma(z;\lambda)\to Z(z;\varkappa)$}\label{RHPtrans}

Our first transformation will be  
\begin{equation}\label{Y}
Y(z;\varkappa):=\Gamma(z;e^{-\varkappa})e^{\varkappa g(z)\sigma_3},
\end{equation}
where $\Gamma(z;\lambda)$ was defined in \eqref{RHPSolution}.  Since $\Gamma(z;\lambda)$ is the solution of RHP \ref{Gamma4RHP}, it is easy to show that $Y(z;\varkappa)$ solves the following RHP.

\begin{problem}\label{YRHP}
Find a matrix $Y(z;\varkappa)$, $e^{-\varkappa}=\lambda\in\mathbb{C}\setminus\{0\}$, analytic for $z\in\bar{\mathbb{C}}\setminus[-1,1]$ and satisfying the following conditions:
\begin{align}
Y(z_+;\varkappa)=&Y(z_-;\varkappa)\begin{bmatrix}e^{\varkappa(g_+-g_-)} & -ie^{-\varkappa(g_++g_--1)} \\ 0 & e^{-\varkappa(g_+-g_-)}\end{bmatrix}, ~ z\in(-1,0) \label{rhpY1} \\
Y(z_+;\varkappa)=&Y(z_-;\varkappa)\begin{bmatrix}e^{\varkappa(g_+-g_-)} & 0 \\ ie^{\varkappa(g_++g_-+1)} & e^{-\varkappa(g_+-g_-)}\end{bmatrix}, ~ z\in(0,1) \label{rhpY3} \\
Y(z;\varkappa) =& 1+\BigO{z^{-1}} \text{ as } z\to\infty, \\
Y(z;\varkappa)=&\begin{bmatrix} \BigO{1} & \BigO{\log(z+1)} \end{bmatrix} \text{ as } z\to-1, \label{rhpYendpt2}\\
Y(z;\varkappa)=&\begin{bmatrix} \BigO{\log(z-1)} & \BigO{1} \end{bmatrix} \text{ as } z\to1, \\
Y(z;\varkappa)\in&L^2_{loc} \text{ as } z\to0.
\end{align}
The endpoint behavior is listed column-wise.

\end{problem}

The jumps for $Y(z;\varkappa)$ on $(-1,0)$ and $(0,1)$ can be written as
\begin{align}
Y(z_+;\varkappa)&=Y(z_-;\varkappa)\begin{bmatrix} 1 & 0 \\ ie^{\varkappa(2g_-(z)-1)} & 1 \end{bmatrix}(-i\sigma_1)\begin{bmatrix} 1 & 0 \\ ie^{\varkappa(2g_+(z)-1)} & 1 \end{bmatrix}, ~ z\in(-1,0) \\
Y(z_+;\varkappa)&=Y(z_-;\varkappa)\begin{bmatrix} 1 & \frac{1}{i}e^{-\varkappa(2g_-(z)+1)} \\ 0 & 1 \end{bmatrix}(i\sigma_1)\begin{bmatrix} 1 & \frac{1}{i}e^{-\varkappa(2g_+(z)+1)} \\ 0 & 1 \end{bmatrix}, ~ z\in(0,1).
\end{align}
This decomposition can be verified by direct matrix multiplication and by using the jump properties of $g(z)$ in Proposition \ref{gproperties}.  We define the `lense' regions $\mathcal{L}_{L,R}^{(\pm)}$ as in Figure \ref{lenseFigure}.

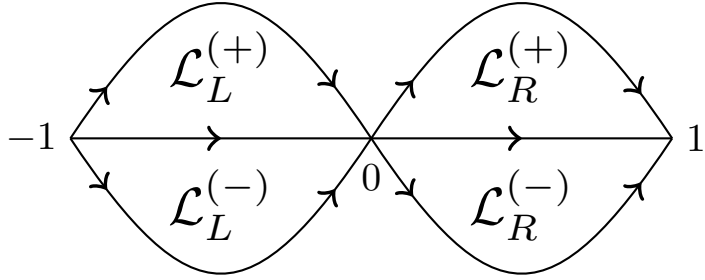
\begin{figure}
\begin{center}
\begin{tikzpicture}[scale=0.8]

\draw[black,thick,postaction = {decorate, decoration = {markings, mark = at position .25 with {\arrow[black,thick,scale=2]{>};}}},,postaction = {decorate, decoration = {markings, mark = at position .75 with {\arrow[black,thick,scale=2]{>};}}}] (-5,0) -- (5,0);

\draw[black,thick,smooth,postaction = {decorate, decoration = {markings, mark = at position .15 with {\arrow[black,thick,scale=2]{>};}}},postaction = {decorate, decoration = {markings, mark = at position .85 with {\arrow[black,thick,scale=2]{>};}}}] (-5,0) .. controls (-3,3) and (-2,3) .. (0,0);
\draw[black,thick,smooth,postaction = {decorate, decoration = {markings, mark = at position .15 with {\arrow[black,thick,scale=2]{>};}}},postaction = {decorate, decoration = {markings, mark = at position .85 with {\arrow[black,thick,scale=2]{>};}}}] (-5,0) .. controls (-3,-3) and (-2,-3) .. (0,0);

\draw[black,thick,smooth,postaction = {decorate, decoration = {markings, mark = at position .17 with {\arrow[black,thick,scale=2]{>};}}},postaction = {decorate, decoration = {markings, mark = at position .87 with {\arrow[black,thick,scale=2]{>};}}}] (0,0) .. controls (2,3) and (3,3) .. (5,0);
\draw[black,thick,smooth,postaction = {decorate, decoration = {markings, mark = at position .17 with {\arrow[black,thick,scale=2]{>};}}},postaction = {decorate, decoration = {markings, mark = at position .87 with {\arrow[black,thick,scale=2]{>};}}}] (0,0) .. controls (2,-3) and (3,-3) .. (5,0);

\node[left,scale=1.5] at (-5,0) {$-1$};
\node[below,scale=1.5] at (0,-0.2) {$0$};
\node[right,scale=1.5] at (5,0) {$1$};

\node[scale=2] at (-2.5,1.2) {$\mathcal{L}_L^{(+)}$};
\node[scale=2] at (-2.5,-1.1) {$\mathcal{L}_L^{(-)}$};
\node[scale=2] at (2.5,1.2) {$\mathcal{L}_R^{(+)}$};
\node[scale=2] at (2.5,-1.1) {$\mathcal{L}_R^{(-)}$};

\end{tikzpicture}

\end{center}
\vspace{-10mm}
\caption{\hspace{.05in}Lense regions $\mathcal{L}_{L,R}^{(\pm)}$.} \label{lenseFigure}
\end{figure}

Recall, from Proposition \ref{gproperties}, that $\Re[2g(z)+1]\geq0$ with equality only for $z\in(0,1)$ and $\Re[2g(z)-1]\leq0$ with equality only for $z\in(-1,0)$.  Our second and final transformation is
\begin{equation}\label{Z}
Z(z;\varkappa) :=
  \begin{cases}
   Y(z;\varkappa), & \text{$z$ outside the lenses} \\
   Y(z;\varkappa)\begin{bmatrix} 1 & 0 \\ \mp ie^{\varkappa(2g(z)-1)} & 1 \end{bmatrix}, & z\in\mathcal{L}_L^{(\pm)} \\
   Y(z;\varkappa)\begin{bmatrix} 1 & \mp\frac{1}{i}e^{-\varkappa(2g(z)+1)} \\ 0 & 1 \end{bmatrix}, & z\in\mathcal{L}_R^{(\pm)}.
  \end{cases}
\end{equation}
Since $Y(z;\varkappa)$ solves RHP \ref{YRHP}, it is a direct calculation to show $Z(z;\varkappa)$ solves the following RHP.

\begin{problem}\label{ZRHP}
Find a matrix $Z(z;\varkappa)$, analytic on the complement of the arcs of Figure \ref{lenseFigure}, satisfying the jump conditions
\begin{equation}
Z(z_+;\varkappa)=Z(z_-;\varkappa)
   \begin{cases}
   \begin{bmatrix} 1 & 0 \\ ie^{\varkappa(2g-1)} & 1 \end{bmatrix} & z\in \partial\mathcal{L}_L^{(\pm)}\setminus\mathbb{R}, \\
   \begin{bmatrix} 1 & \frac{1}{i}e^{-\varkappa(2g+1)} \\ 0 & 1 \end{bmatrix} & z\in\partial\mathcal{L}_R^{(\pm)}\setminus\mathbb{R}, \\
   -i\sigma_1 & z\in (-1,0), \\
   i\sigma_1 & z\in (0,1),
   \end{cases}
\end{equation}
normalized by
\begin{equation}
Z(z;\varkappa)=1+\BigO{z^{-1}}, \text{   as   }  z\to\infty,
\end{equation}
and with the same endpoint behavior as $Y(z;\varkappa)$ near the endpoints $z=0,\pm1$, see $\eqref{rhpYendpt2}$.

\end{problem}

The jumps for $Z(z;\varkappa)$ on $\partial\mathcal{L}_{L,R}^{(\pm)}$ will be exponentially small as long as $z$ is a fixed distance away from $0,\pm1$ due to Proposition $\ref{gproperties}$.  If we `ignore' the jumps on the lenses of the RHP for $Z(z;\varkappa)$, we obtain the so-called model RHP.

\begin{problem}\label{modelRHP}

Find a matrix $\Psi(z)$, analytic on $\overline{\mathbb{C}}\setminus [-1,1]$, and satisfying
\begin{align}
&\Psi_+(z)=\Psi_-(z)(-i\sigma_1), \text{   for   } z\in [-1,0], \\
&\Psi_+(z)=\Psi_-(z)(i\sigma_1), \text{   for   } z\in [0,1],  \\
&\Psi(z)=\BigO{|z\mp1|^{-\frac{1}{4}}}, \text{ as } z\to \pm1,  \\ \label{zero-beh}
&\Psi(z)=\BigO{|z|^{-\frac{1}{2}}}, \text{ as } z\to 0,       \\ 
&\Psi(z)=\1+\BigO{z^{-1}} \text{  as  } z\to\infty.
\end{align}

\end{problem}

Note that condition \eqref{zero-beh} does not guarantee the uniqueness of $\Psi(z)$.

\begin{theorem}\label{model_sol}
$\Psi(z)$ is a solution to RHP \ref{modelRHP} if and only if  there exist $x,y\in\mathbb{C}$ such that
\begin{equation}\label{solmod}
\Psi(z)=\left(\1+\frac{1}{z}\begin{bmatrix} x & -x \\ y & -y \end{bmatrix}\right)\left(\frac{z^2-1}{z^2}\right)^{\sigma_1/4}.
\end{equation}
\end{theorem}

\begin{proof}

The Sokhotski-Plemelj formula (see \cite{Gak66}) can be applied to this problem to obtain the solution 
\begin{equation}
\Psi_1(z)=\beta(z)^{\sigma_1}, ~ \text{ where } ~  \beta(z)=\left(\frac{z^2-1}{z^2}\right)^{1/4}.
\end{equation}
Take any solution to RHP \ref{modelRHP} (different from $\Psi_1(z)$) and call it $\Psi_2(z)$.  Then it can be seen that the matrix $\Psi_2(z)\Psi_1^{-1}(z)$ has no jumps in the complex plane, $\Psi_2(z)\Psi_1^{-1}(z)=\1+\BigO{z^{-1}}$ as $z\to\infty$ and $\Psi_2(z)\Psi_1^{-1}(z)=\BigO{z^{-1}}$ as $z\to0$.  Then it must be that 
\begin{equation}
\Psi_2(z)\Psi_1^{-1}(z)=\1+\frac{A}{z},
\end{equation}
where $A$ is a constant matrix.  Notice that 
\begin{equation}
\Psi_1(z)=\frac{\beta(z)}{2}\begin{bmatrix} 1 & 1 \\ 1 & 1 \end{bmatrix}+\frac{1}{2\beta(z)}\begin{bmatrix} 1 & -1 \\ -1 & 1 \end{bmatrix},
\end{equation}
so we have
\begin{equation}
\Psi_2(z)=\left(\1+\frac{A}{z}\right)\Psi_1(z)=\left(\1+\frac{A}{z}\right)\left(\frac{\beta(z)}{2}\begin{bmatrix} 1 & 1 \\ 1 & 1 \end{bmatrix}+\frac{1}{2\beta(z)}\begin{bmatrix} 1 & -1 \\ -1 & 1 \end{bmatrix}\right). 
\end{equation}
Since $\Psi_2(z)$ is a solution of RHP \ref{modelRHP}, it must be true that $\Psi_2(z)=\BigO{z^{-1/2}}$ as $z\to0$.  Thus the matrix $A$ must satisfy
\begin{equation}
A\cdot\begin{bmatrix} 1 & 1 \\ 1 & 1 \end{bmatrix}=\begin{bmatrix} 0 & 0 \\ 0 & 0 \end{bmatrix} ~ \implies ~ A=\begin{bmatrix} x & -x \\ y & -y \end{bmatrix}.
\end{equation}
It is easy now to check that \eqref{solmod} with any $x,y\in\mathbb{C}$ satisfies RHP \ref{modelRHP}.

\end{proof}

\begin{remark}\label{modelProbSym}
Assume $\Psi(z)$ is a solution of the RHP \ref{modelRHP}.  Then $\det\Psi(z)\equiv 1$ if and only if $y=x$ in the representation \eqref{solmod}.  If, additionally, $x\in i\mathbb{R}$ in this representation then $\Psi(z)$ has the symmetry
\begin{equation}
    \overline{\Psi(\overline{z})}=\sigma_1\Psi(z)\sigma_1.
\end{equation}
Both properties can be easily verified.
\end{remark}

\subsubsection{Approximation of $Z(z;\varkappa)$ and Main Result}

We will construct a piecewise (in $z$) approximation of $Z(z;\varkappa)$ when $\varkappa\to\infty$.  Our approach is very similar to that in \cite{BKT16}.  Denote by
$\mathbb{D}_j$ a disc of small radius $l$ centered at $j$, $j=0,\pm1$, where $l$ is chosen so that $\partial\mathbb{D}_0\subset\tilde{\Omega}$.  The idea is as follows: on the lenses $\mathcal{L}_{L,R}^{(\pm)}$ (see Figure \ref{lenseFigure}) outside the discs $\mathbb{D}_j$, $j=0,\pm1$, the jumps of $Z(z;\varkappa)$ are uniformly close to the identity matrix thus $\Psi_0(z;\varkappa)$ (a solution to model RHP, see \eqref{Psi}) is a `good' approximation of $Z(z;\varkappa)$.  Inside $\mathbb{D}_j$, $j=0,\pm1$, we construct local approximations that are commonly called `parametrices'.  The solution of the so-called Bessel RHP is necessary.

\begin{problem}\label{BesselRHP}
Let $\nu\in(0,\pi)$ be any fixed number.  Find a matrix $\mathcal{B}_\nu(\zeta)$ that is analytic off the rays $\mathbb{R}_-$, $e^{\pm i\theta}\mathbb{R}^+$ and satisfies the following conditions.
\begin{align}
&\mathcal{B}_{\nu+}(\zeta)=\mathcal{B}_{\nu-}(\zeta)\begin{bmatrix} 1 & 0 \\ e^{-4\sqrt{\zeta}\pm i\pi\nu} & 1 \end{bmatrix}, ~ \zeta\in e^{\pm i\theta}\mathbb{R}_+ \\
&\mathcal{B}_{\nu+}(\zeta)=\mathcal{B}_{\nu-}(\zeta)\begin{bmatrix} 0 & 1 \\ -1 & 0 \end{bmatrix}, ~ \zeta\in\mathbb{R}_- \\
&\mathcal{B}_{\nu}(\zeta)=\BigO{\zeta^{-\frac{|\nu|}{2}}} \text{ for } \nu\neq0 \text{ or } \BigO{\log{\zeta}} \text{ for } \nu=0 \text{ as } \zeta\to0, \\
&\mathcal{B}_{\nu}(\zeta)=F(\zeta)\left(1+\BigO{\frac{1}{\sqrt{\zeta}}}\right) \text{ as } \zeta\to\infty, ~~ \text{where}~~ F(\zeta):=(2\pi)^{-\sigma_3/2}\zeta^{-\frac{\sigma_3}{4}}\frac{1}{\sqrt{2}}\begin{bmatrix} 1 & -i \\ -i & 1 \end{bmatrix}. \label{Binf_behavior}
\end{align}
\end{problem}

This RHP has an explicit solution in terms of Bessel functions and can be found in \cite{Van03}.  Define local coordinates at points $z=\pm1$ as 
\begin{align}
-4\sqrt{\xi_{-1}(z)}=&\varkappa(2g(z)-1), ~ \text{ for }z\in\mathbb{D}_{-1}, \label{l_coord-1} \\
4\sqrt{\xi_1(z)}=&\varkappa(2g(z)+1), ~ \text{ for }z\in\mathbb{D}_1. \label{l_coord1}
\end{align}
We call $\tilde{Z}(z;\varkappa)$ our approximation of $Z(z;\varkappa)$ and define
\begin{equation}\label{tildeZ}
\tilde{Z}(z;\varkappa):=
\begin{cases}
\Psi_0(z;\varkappa), &  z \in \mathbb{C}\setminus \ds{\bigcup_{j=-1}^1}\mathbb{D}_j, \\
       \Psi_0(z;\varkappa)i^{-\frac{\sigma_3}{2}}F^{-1}(\xi_{-1})\mathcal{B}_0(\xi_{-1})i^{\frac{\sigma_3}{2}}, &  z \in \mathbb{D}_{-1}, \\
       Z(z;\varkappa), &  z \in \mathbb{D}_0, \\
       \Psi_0(z;\varkappa)i^{-\frac{\sigma_3}{2}}\sigma_1F^{-1}(\xi_1)\mathcal{B}_0(\xi_1)\sigma_1i^{\frac{\sigma_3}{2}}, &  z \in \mathbb{D}_1.
\end{cases}
\end{equation}

\begin{remark}\label{tildeZjumps}
    The matrix $\tilde{Z}(z;\varkappa)$ was constructed to have exactly the same jumps as $Z(z;\varkappa)$ when $z\in\mathbb{D}_{0,\pm1}\cup[-1,1]$.  For 
    more details see \cite{BKT16}, section 4.3.
\end{remark}

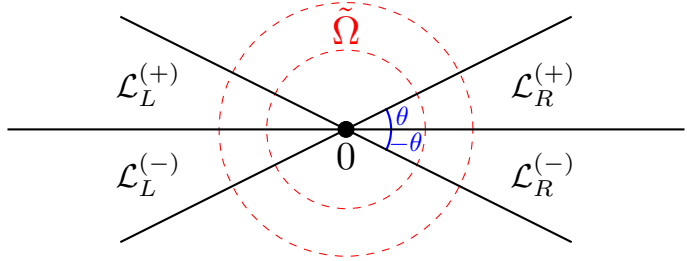
\begin{figure}[H]
\begin{center}
\begin{tikzpicture}[scale=1.5]

\draw[blue,thick] (0.4,0) arc (0:28:0.4);
\node[blue,scale=1.0] at (0.5,0.12) {$\theta$};

\draw[blue,thick] (0.4,0) arc (0:-28:0.4);
\node[blue,scale=1.0] at (0.53,-0.12) {$-\theta$};

\draw[dashed,red] (0,0) circle [radius=32pt];
\draw[dashed,red] (0,0) circle [radius=20pt];

\draw[black,thick] (-3,0) -- (3,0);

\draw[black,thick] (-2,1) -- (2,-1);
\draw[black,thick] (-2,-1) -- (2,1);

\filldraw[black] (0,0) circle [radius=2pt];

\node[below,scale=1.5] at (0,-0.0) {$0$};
\node[red,scale=1.5] at (0,0.92) {$\tilde{\Omega}$};
\node[scale=1.25] at (-1.75,0.4) {$\mathcal{L}_L^{(+)}$};
\node[scale=1.25] at (-1.75,-0.4) {$\mathcal{L}_L^{(-)}$};
\node[scale=1.25] at (1.75,0.4) {$\mathcal{L}_R^{(+)}$};
\node[scale=1.25] at (1.75,-0.4) {$\mathcal{L}_R^{(-)}$};

\end{tikzpicture}
\end{center}
\caption{\hspace{.05in}The set $\tilde{\Omega}$ and lenses.}\label{figLenseOmegaTilde}
\end{figure}

\begin{corollary}\label{ZAsymptotics}
\begin{equation}
    Z(z;\varkappa)=\Psi_0(z;\varkappa)\left(\1+\BigO{\frac{M^2}{\varkappa}}\right)~~~\text{as}~~~\varkappa\to\infty
\end{equation}
uniformly in $z\in\tilde{\Omega}$,  

\end{corollary}

\begin{proof}
    In Theorem \ref{GammaAsmptotics}, we obtained the leading order behavior of $\Gamma(z;\lambda)$ for $z\in\tilde{\Omega}$ as $\lambda\to0$.  This Theorem can easily be written in terms of $Z(z;\varkappa)$ instead of $\Gamma(z;\lambda)$ by applying the transformations (see section \ref{RHPtrans}) $\Gamma\to Y\to Z$.
\end{proof}

Define the error matrix as 
\begin{equation}\label{errorMatrix}
\mathcal{E}(z;\varkappa):=Z(z;\varkappa)\tilde{Z}^{-1}(z;\varkappa),
\end{equation}
It is clear that $\mathcal{E}(z;\varkappa)=\1+\BigO{z^{-1}}$ as $z\to\infty$ since both $Z(z;\varkappa),\tilde{Z}(z;\varkappa)$ have this behavior.  $\mathcal{E}(z;\varkappa)$ has no jumps inside $\mathbb{D}_{-1,0,1}$ because $\tilde{Z}(z;\varkappa)$ was constructed to have the same jumps as $Z(z;\varkappa)$ inside $\mathbb{D}_{-1,0,1}$, see Remark \ref{tildeZjumps}.  Thus $\mathcal{E}(z;\varkappa)$ will have jumps on $\partial\mathbb{D}_{-1,0,1}$, $\partial\mathcal{L}_{L,R}^{(\pm)}\setminus\mathbb{D}_{0,\pm1}$, and be analytic elsewhere.  Explicitly,  
\begin{equation}\label{Ejumps}
    \mathcal{E}(z_+;\varkappa)=\mathcal{E}(z_-;\varkappa)
    \begin{cases}
        \Psi_0(z;\varkappa)i^{-\frac{\sigma_3}{2}}F^{-1}(\xi_{-1})B_0(\xi_{-1})i^{\frac{\sigma_3}{2}}\Psi_0^{-1}(z;\varkappa), &z\in\partial\mathbb{D}_{-1}, \\
        
        \1+\Psi_0(z;\varkappa)\begin{bmatrix} 0 & 0 \\ ie^{\varkappa(2g(z)-1)} & 0 \end{bmatrix}\Psi_0^{-1}(z;\varkappa), &z\in\partial\mathcal{L}_L^{(\pm)}\setminus\mathbb{D}_{-1,0}, \\
        
        Z(z;\varkappa)\Psi_0^{-1}(z;\varkappa), &z\in\partial\mathbb{D}_0, \\
        
        \1+\Psi_0(z;\varkappa)\begin{bmatrix} 0 & -ie^{-\varkappa(2g(z)+1)} \\ 0 & 0 \end{bmatrix}\Psi_0^{-1}(z;\varkappa), &z\in\partial\mathcal{L}_R^{(\pm)}\setminus\mathbb{D}_{0,1}, \\
        
        \Psi_0(z;\varkappa)i^{-\frac{\sigma_3}{2}}\sigma_1F^{-1}(\xi_{1})B_0(\xi_{1})\sigma_1i^{\frac{\sigma_3}{2}}\Psi_0^{-1}(z;\varkappa), &z\in\partial\mathbb{D}_{1}.
    \end{cases}
\end{equation}
Call $\Sigma$ the collection of arcs where $\mathcal{E}(z;\varkappa)$ has a jump, as described in Figure \ref{errorJumpsFig}.

\begin{figure}[H]
\begin{center}
\begin{tikzpicture}[scale=1.25]

\filldraw[black] (-3,0) circle [radius=1.5pt] node[anchor=south east] {-1};
\filldraw[black] (0,0) circle [radius=1.5pt] node[anchor=south east] {0};
\filldraw[black] (3,0) circle [radius=1.5pt] node[anchor=south east] {1};

\draw[black,dashed,postaction = decorate, decoration = {markings, mark = at position .54 with {\arrow[black,thick]{<};}}] (-3,0) circle [radius=14pt];
\draw[black,dashed,postaction = decorate, decoration = {markings, mark = at position .54 with {\arrow[black,thick]{<};}}] (0,0) circle [radius=14pt];
\draw[black,dashed,postaction = decorate, decoration = {markings, mark = at position .54 with {\arrow[black,thick]{<};}}] (3,0) circle [radius=14pt];

\draw [smooth,dashed,postaction = decorate, decoration = {markings, mark = at position .54 with {\arrow[black,thick]{>};}}] (-2.8,0.5) .. controls (-1.5,1.5) .. (-0.2,0.5);
\draw [smooth,dashed,postaction = decorate, decoration = {markings, mark = at position .54 with {\arrow[black,thick]{>};}}] (0.2,0.5) .. controls (1.5,1.5) .. (2.8,0.5);
\draw [smooth,dashed,postaction = decorate, decoration = {markings, mark = at position .54 with {\arrow[black,thick]{>};}}] (-2.8,-0.5) .. controls (-1.5,-1.5) .. (-0.2,-0.5);
\draw [smooth,dashed,postaction = decorate, decoration = {markings, mark = at position .54 with {\arrow[black,thick]{>};}}] (0.2,-0.5) .. controls (1.5,-1.5) .. (2.8,-0.5);

\end{tikzpicture}
\end{center}
\vspace{-5mm}
\caption{\hspace{.05in}The contour $\Sigma$, where $\mathcal{E}(z;\varkappa)$ has jumps.}\label{errorJumpsFig}
\end{figure}
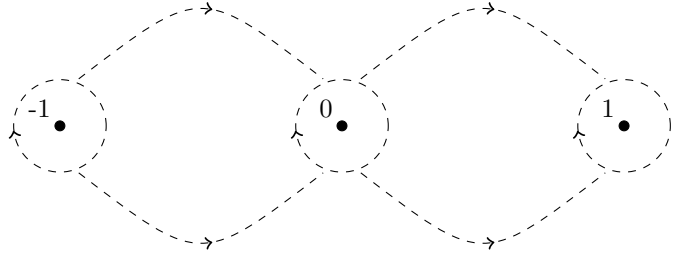

\begin{remark}
    The factors $e^{-\varkappa(2g(z)+1)}$ and $e^{\varkappa(2g(z)-1)}$ in \eqref{EjumpLargeKappa} are exponentially small for all $z$ in the corresponding set, in light of Proposition \ref{gproperties}.
\end{remark}

Now revisiting the jumps of $\mathcal{E}(z;\varkappa)$ in \eqref{Ejumps}, we have another Corollary.

\begin{corollary}\label{EjumpAsympCor}
For any $z\in\Sigma$,
    \begin{equation}\label{EjumpLargeKappa}
        \mathcal{E}(z_+;\varkappa)=\mathcal{E}(z_-;\varkappa)
        \begin{cases}
            \1+\BigO{\frac{M^2}{\varkappa}}, &z\in \partial\mathbb{D}_{-1,0,1}, \\
            \1+\BigO{e^{\varkappa(2g(z)-1)}}, &z\in\partial\mathcal{L}_L^{(\pm)}\setminus\mathbb{D}_{-1,0}, \\
            \1+\BigO{e^{-\varkappa(2g(z)+1)}}, &z\in\partial\mathcal{L}_R^{(\pm)}\setminus\mathbb{D}_{0,1},
        \end{cases}~~~\text{as}~~~\varkappa\to\infty
    \end{equation}
    uniformly in $z\in\Sigma$.
\end{corollary}

\begin{proof}
    The behavior for $z\in\partial\mathbb{D}_{\pm1},\partial\mathbb{D}_{0}$ is a direct consequence of \eqref{Binf_behavior}, Corollary \ref{ZAsymptotics}, respectively.  The behavior on the lenses is clear via inspection of \eqref{Ejumps}.
\end{proof}

\begin{corollary}\label{ZAsympUni}
Let $J$ be a compact subset of  $\mathbb{C}\setminus\{-1,0,1\}$, where we distinguish the points on the upper and lower sides of $(-1,0)$ and $(0,1)$. Then we have the approximation
    \begin{equation}
        Z(z;\varkappa)=\Psi_0(z;\varkappa)\left(\1+\BigO{\frac{M^2}{\varkappa}}\right) ~~~\text{as}~~\varkappa\to\infty
    \end{equation}
    uniformly in $z\in J$.
\end{corollary}

\begin{proof}
    Given $J$, the disks $\mathbb{D}_{0,\pm1}$ can be taken sufficiently small in order to not intersect $J$.  Corollaries \ref{ZAsymptotics}, \ref{EjumpAsympCor} and the so-called small norm theorem, see \cite{Dei00} Theorem 7.171, can now be applied to conclude that $\mathcal{E}(z;\varkappa)=\1+\BigO{\frac{M^2}{\varkappa}}$ uniformly for $z\in J$.  This is equivalent to the stated result.    
    
\end{proof}

We are now ready to prove the main result of this section.

\begin{theorem}\label{GammaAsympMainResult}
In the limit $\lambda=e^{-\varkappa}\to0$\emph{:} 
\begin{enumerate}
    \item For $z$ in compact subsets of $\mathbb{C}\setminus[-1,1]$ we have the uniform approximation 
    \begin{equation}
        \Gamma(z;\lambda) =
        \Psi_0(z;\varkappa)\left(\1+\BigO{\frac{M^2}{\varkappa}}\right)e^{-\varkappa g(z)\sigma_3};
    \end{equation}
    \item For $z$ in compact subsets of $(-1,0)\cup(0,1)$ we have the uniform approximation 
    \begin{equation}
        \Gamma(z_\pm;\lambda) =
        \begin{cases}
        \Psi_0(z_\pm;\varkappa)\left(\1+\BigO{\frac{M^2}{\varkappa}}\right)\begin{bmatrix} 1 & 0 \\ \pm ie^{\varkappa(2g(z_\pm)-1)} & 1 \end{bmatrix}e^{-\varkappa g(z_\pm)\sigma_3}, & z\in(-1,0), \\
        \Psi_0(z_\pm;\varkappa)\left(\1+\BigO{\frac{M^2}{\varkappa}}\right)\begin{bmatrix} 1 & \mp ie^{-\varkappa(2g(z_\pm)+1)} \\ 0 & 1 \end{bmatrix}e^{-\varkappa g(z_\pm)\sigma_3}, &z\in(0,1),
        \end{cases}
    \end{equation}
\end{enumerate}
where $\pm$ denotes the upper/lower shore of the real axis in the $z$-plane.
\end{theorem}

\begin{proof}
    
    This Theorem is a direct consequence of Corollary \ref{ZAsympUni}.  We simply need to revert the transforms that took $\Gamma\to Z$.  Doing so, we find that
    \begin{equation}
        \Gamma(z;\lambda)=\begin{cases}
            Z(z;\varkappa)e^{-\varkappa g(z)\sigma_3}, & \text{$z$ outside lenses} \\
            Z(z;\varkappa)\begin{bmatrix} 1 & 0 \\ \pm ie^{\varkappa(2g(z)-1)} & 1 \end{bmatrix} e^{-\varkappa g(z)\sigma_3}, & z\in\mathcal{L}_L^{(\pm)} \\
            Z(z;\varkappa)\begin{bmatrix} 1 & \mp ie^{-\varkappa(2g(z)+1)} \\ 0 & 1 \end{bmatrix} e^{-\varkappa g(z)\sigma_3}, & z\in\mathcal{L}_R^{(\pm)}.
        \end{cases}
    \end{equation}
    Since $z$ is in a compact subset of $\mathbb{C}\setminus[-1,1]$ or $(-1,0)\cup(0,1)$, we apply Corollary \ref{ZAsympUni} to obtain the result.

\end{proof}






\appendix

\section{Construction of $\Gamma(z;\lambda)$}\label{appendixGammaConstruct}

Recall that the Hypergeometric ODE (see \cite{DLMF} 15.10.1) is 
\begin{equation}
\eta(1-\eta)\frac{{\mathrm{d}}^{2}w}{{\mathrm{d}\eta}^{2}}+\left(c-(a+b+1)\eta\right)\frac%
{\mathrm{d}w}{\mathrm{d}\eta}-abw=0,\label{hypergeoODE}
\end{equation}
which has exactly three regular singular points at $\eta=0,1,\infty$.  The idea is to choose parameters $a,b,c$ so that the monodromy matrices of the fundamental matrix solution solution of the ODE will match (up to similarity transformation) the jump matrices of RHP \ref{Gamma4RHP}.  We orient the real axis of the $\eta-$plane as described in Figure \ref{figRealAxisEta}.

\begin{figure}[h]
\begin{center}
\begin{tikzpicture}[scale=1.0]

\draw[blue,dashed,postaction = {decorate, decoration = {markings, mark = at position .54 with {\arrow[red,thick]{>};}}}] (-5,0) -- (-3,0);

\node[above] at (-3,0) {$0$};

\draw[blue,dashed,postaction = {decorate, decoration = {markings, mark = at position .54 with {\arrow[red,thick]{<};}}}] (-3,0) -- (-1,0);
\node[above] at (-1,0) {$1$};

\draw[blue,dashed,postaction = {decorate, decoration = {markings, mark = at position .54 with {\arrow[red,thick]{<};}}}] (-1,0) -- (5,0);
\node[above] at (5,0) {$\infty$};

\draw[blue,dashed,postaction = {decorate, decoration = {markings, mark = at position .54 with {\arrow[red,thick]{>};}}}] (5,0) -- (6,0);

\filldraw[black] (-3,0) circle [radius=1.5pt];
\filldraw[black] (-1,0) circle [radius=1.5pt];
\filldraw[black] (5,0) circle [radius=1.5pt];

\end{tikzpicture}
\end{center}
\caption{\hspace{.1in}Orientation of the real axis of the $\eta$-plane.}\label{figRealAxisEta}
\end{figure}
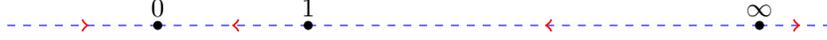

\subsection{Solutions of ODE \eqref{hypergeoODE} near Regular Singular Points and Connection Formula}\label{ODEsolConnec}

According to \cite{DLMF} 15.10.11 - 15.10.16, three pairs of linearly independent solutions of ODE \eqref{hypergeoODE} when $\eta=0,1,\infty$, respectively, are
\begin{align}
    &h_0(\eta)=\pFq{2}{1}{a , b}{c}{\eta},~ s_0(\eta)=\eta^{1-c}\pFq{2}{1}{a-c+1,b-c+1}{2-c}{\eta}, \label{0pair} \\
    &h_1(\eta)=\pFq{2}{1}{a,b}{a+b+1-c}{1-\eta},~ s_1(\eta)=(1-\eta)^{c-a-b}\pFq{2}{1}{c-a,c-b}{c-a-b+1}{1-\eta}, \label{1pair} \\
    &h_\infty(\eta)=e^{a\pi i}\eta^{-a}\pFq{2}{1}{a,a-c+1}{a-b+1}{\frac{1}{\eta}},~ s_\infty(\eta)=e^{b\pi i}\eta^{-b}\pFq{2}{1}{b,b-c+1}{b-a+1}{\frac{1}{\eta}}. \label{inftypair}
\end{align}
From \cite{DLMF} 15.10.7, we have that 
\begin{equation}\label{hypergeo_det_inf}
W_\eta\left[h_\infty(\eta),s_\infty(\eta)\right]=e^{(a+b)\pi i}(a-b)\eta^{-c}(1-\eta)^{c-a-b-1}.
\end{equation}
Kummer's 20 connection formula are listed in \cite{DLMF} 15.10.17 - 15.10.36.  We will list only what is necessary in this construction.  The connection between solutions at $\eta=0$ and $\eta=\infty$ is (see \cite{DLMF} 15.10.19, 15.10.20, 15.10.25, 15.10.26)
\begin{align}
\begin{bmatrix} h_0(\eta) & s_0(\eta) \end{bmatrix}&=\begin{bmatrix} h_\infty(\eta) & s_\infty(\eta) \end{bmatrix}C_{\infty 0}, \\
\begin{bmatrix} h_\infty(\eta) & s_\infty(\eta) \end{bmatrix}&=\begin{bmatrix} h_0(\eta) & s_0(\eta) \end{bmatrix}C_{0 \infty},
\end{align}
where
\begin{align}
C_{\infty0}&=\begin{bmatrix} \frac{\Gamma\left(c\right)\Gamma\left(b-a\right)}{\Gamma\left(b\right)\Gamma\left(c-a\right)} & e^{(1-c)\pi i}\frac{\Gamma\left(2-c\right)\Gamma\left(b-a\right)}{\Gamma\left(1-a\right)\Gamma\left(b-c+1\right)} \\ \frac{\Gamma\left(c\right)\Gamma\left(a-b\right)}{\Gamma\left(a\right)\Gamma\left(c-b\right)} & e^{(1-c)\pi i}\frac{\Gamma\left(2-c\right)\Gamma\left(a-b\right)}{\Gamma\left(1-b\right)\Gamma\left(a-c+1\right)} \end{bmatrix}, \label{CInf0} \\
C_{\infty 0}^{-1}&=C_{0\infty}=\begin{bmatrix} \frac{\Gamma\left(1-c\right)\Gamma\left(a-b+1\right)}{\Gamma\left(a-c+1\right)\Gamma\left(1-b\right)} & \frac{\Gamma\left(1-c\right)\Gamma\left(b-a+1\right)}{\Gamma\left(b-c+1\right)\Gamma\left(1-a\right)} \\ e^{(c-1)\pi i}\frac{\Gamma\left(c-1\right)\Gamma\left(a-b+1\right)}{\Gamma\left(a\right)\Gamma\left(c-b\right)} & e^{(c-1)\pi i}\frac{\Gamma\left(c-1\right)\Gamma\left(b-a+1\right)}{\Gamma\left(b\right)\Gamma\left(c-a\right)} \end{bmatrix}. \label{C0Inf}
\end{align}
The connection between solutions at $\eta=1$ and $\eta=\infty$ is (see \cite{DLMF} 15.10.23, 15.10.24, 15.10.27, 15.10.28)
\begin{align}
\begin{bmatrix} h_1(\eta) & s_1(\eta) \end{bmatrix}&=\begin{bmatrix} h_\infty(\eta) & s_\infty(\eta) \end{bmatrix}C_{\infty1}, \\
\begin{bmatrix} h_\infty(\eta) & s_\infty(\eta) \end{bmatrix}&=\begin{bmatrix} h_1(\eta) & s_1(\eta) \end{bmatrix}C_{1\infty},
\end{align}
where
\begin{align}
    C_{\infty 1}&=\begin{bmatrix} e^{-a\pi i}\frac{\Gamma\left(a+b-c+1\right)\Gamma\left(b-a\right)}{\Gamma\left(b\right)\Gamma\left(b-c+1\right)} & e^{(b-c)\pi i}\frac{\Gamma\left(c-a-b+1\right)\Gamma\left(b-a\right)}{\Gamma\left(1-a\right)\Gamma\left(c-a\right)} \\ e^{-b\pi i}\frac{\Gamma\left(a+b-c+1\right)\Gamma\left(a-b\right)}{\Gamma\left(a\right)\Gamma\left(a-c+1\right)} & e^{(a-c)\pi i}\frac{\Gamma\left(c-a-b+1\right)\Gamma\left(a-b\right)}{\Gamma\left(1-b\right)\Gamma\left(c-b\right)} \end{bmatrix}, \label{CInf1} \\
    C_{\infty 1}^{-1}&=C_{1\infty}=\begin{bmatrix} e^{a\pi i}\frac{\Gamma\left(a-b+1\right)\Gamma\left(c-a-b\right)}{\Gamma\left(1-b\right)\Gamma\left(c-b\right)} & e^{b\pi i}\frac{\Gamma\left(b-a+1\right)\Gamma\left(c-a-b\right)}{\Gamma\left(1-a\right)\Gamma\left(c-a\right)} \\ e^{(c-b)\pi i}\frac{\Gamma\left(a-b+1\right)\Gamma\left(a+b-c\right)}{\Gamma\left(a\right)\Gamma\left(a-c+1\right)} & e^{(c-a)\pi i}\frac{\Gamma\left(b-a+1\right)\Gamma\left(a+b-c\right)}{\Gamma\left(b\right)\Gamma\left(b-c+1\right)} \end{bmatrix}. \label{C1Inf}
\end{align}

\subsection{Selection of Parameters $a,b,c$}

Define
\begin{equation}\label{hatGammaConstruct}
\hat{\Gamma}(\eta):=\eta^{\frac{c}{2}}(1-\eta)^{\frac{a+b-c+1}{2}}\begin{bmatrix} h_\infty(\eta) & s_\infty(\eta) \\ h'_\infty(\eta) & s'_\infty(\eta) \end{bmatrix}.
\end{equation}
Notice that for any $\eta\in\mathbb{C}$
\begin{equation}\label{appendix_detHatGamma}
    \det\left(\hat{\Gamma}(\eta)\right)=e^{(a+b)\pi i}(a-b)
\end{equation}
according to \eqref{hypergeo_det_inf}.  Our solution $\Gamma(z;\lambda)$ to RHP \ref{Gamma4RHP} has singular points at $z=b_L,0,b_R$.  Notice that the M\"obius transform 
\begin{equation}\label{m1Construct}
    \eta=M_1(z):=\frac{b_R(z-b_L)}{z(b_R-b_L)}
\end{equation}
maps $b_L\to 0$, $b_R\to1$, and $0\to\infty$ where the orientation of the $z-$axis is described in Figure \ref{figRealAxisZ}.  Thus we are interested in the matrix 
\begin{align}
\hat{\Gamma}\left(M_1(z)\right)&=\left(\frac{b_R(z-b_L)}{|b_L|(z-b_R)}\right)^{\frac{c}{2}}\left(\frac{|b_L|(z-b_R)}{z(b_R-b_L)}\right)^{\frac{a+b+1}{2}}\begin{bmatrix} h_\infty\left(M_1(z)\right) & s_\infty\left(M_1(z)\right) \\ h_\infty'\left(M_1(z)\right) & s_\infty'\left(M_1(z)\right) \end{bmatrix}.
\end{align}
\begin{figure}
\begin{center}
\begin{tikzpicture}[scale=1.0]

\draw[blue,dashed,postaction = {decorate, decoration = {markings, mark = at position .54 with {\arrow[red,thick]{>};}}}] (-5,0) -- (-3,0);

\draw[blue,dashed,postaction = {decorate, decoration = {markings, mark = at position .54 with {\arrow[red,thick]{<};}}}] (-3,0) -- (-1,0);

\draw[blue,dashed,postaction = {decorate, decoration = {markings, mark = at position .54 with {\arrow[red,thick]{>};}}}] (-1,0) -- (2,0);

\draw[blue,dashed,postaction = {decorate, decoration = {markings, mark = at position .54 with {\arrow[red,thick]{>};}}}] (2,0) -- (6,0);

\filldraw[black] (-1,0) circle [radius=1.5pt];
\node[above] at (-1,0) {$0$};
\filldraw[black] (2,0) circle [radius=1.5pt];
\node[above] at (2,0) {$b_R$};
\filldraw[black] (-3,0) circle [radius=1.5pt];
\node[above] at (-3,0) {$b_L$};

\end{tikzpicture}
\end{center}
\caption{\hspace{.1in}Orientation of the real axis of the $z$-plane.}\label{figRealAxisZ}
\end{figure}
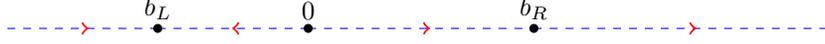
We need to determine parameters $a,b,c$ such that $\hat{\Gamma}(M_1(z))$ is $L^2_{\text{loc}}$ at $z=b_L,0,b_R$, so we are interested in the bi-resonant case, which is  
\begin{equation}
c\in\mathbb{Z}, \hspace{5mm} c-b-a\in\mathbb{Z}.
\end{equation}
When $z= b_L$, to guarantee that $\hat{\Gamma}$ is $L^2_{loc}$ we must have that (use the connection formula of section \ref{ODEsolConnec} to easily inspect the local behavior)
\begin{align}
\frac{c}{2}&>-\frac{1}{2}, ~~~ -\frac{c}{2}>-\frac{1}{2}.
\end{align}
Since $c\in\mathbb{Z}$, it must be so that $c=0$.  Now for $z=b_R$, we must have that 
\begin{align}
\frac{a+b-c+1}{2}=\frac{r+1}{2}>-\frac{1}{2}, \\
\frac{1-a-b}{2}=\frac{-r-1}{2}>-\frac{1}{2}
\end{align}
where $a+b=r\in\mathbb{Z}$.  Since $r\in\mathbb{Z}$, the only possibility is $r=-1$.  So we have that $b=-1-a$ and $c=0$.  Lastly, as $z\to0$,
\begin{align}
    h_\infty\left(M_1(z)\right)&=\BigO{z^a}, ~~~ h_\infty'\left(M_1(z)\right)=\BigO{z^{a+1}}, \\
    s_\infty\left(M_1(z)\right)&=-e^{-a\pi i}\left(\frac{-b_Rb_L}{z(b_R-b_L)}\right)^{a+1}+\BigO{z^{-a}}, \\
    s_\infty'\left(M_1(z)\right)&=-(a+1)e^{-a\pi i}\left(\frac{-b_Rb_L}{z(b_R-b_L)}\right)^{a}+\BigO{z^{-a+1}},
\end{align}
so we see that it is not possible for $\hat{\Gamma}(M_1(z))$ to have $L^2$ behavior at $z=0$.  On the other hand, observe that
\begin{align}
    s_\infty\left(M_1(z)\right)+\frac{b_Lb_R}{z(b_R-b_L)(a+1)}s_\infty'\left(M_1(z)\right)=\BigO{z^{-a}}, ~~ z\to0.
\end{align}
Thus the matrix 
\begin{equation}\label{hatGammaL2Loc}
    \begin{bmatrix} 1 & \frac{b_Lb_R}{z(b_R-b_L)(a+1)} \\ 0 & 1 \end{bmatrix}\hat{\Gamma}(M_1(z))
\end{equation}
is $L^2_{\text{loc}}$ as $z\to0$ provided that $|\Re(a)|<1/2$.  In the next section, we solve for $a$ explicitly in terms of $\lambda$ and the condition $|\Re(a)|<1/2$ will be met provided that $\lambda\notin[-1/2,1/2]$, see Appendix \ref{aAppendix}.

\subsection{Monodromy}

The monodromy matrices of $\hat{\Gamma}(z)$ about the singular points $z=0,1,\infty$ are 
\begin{equation}
M_0=C_{\infty0}e^{i\pi c\sigma_3}C_{0\infty},\hspace{0.5cm} M_1=C_{\infty1}e^{i\pi(a+b-c+1)\sigma_3}C_{1\infty},\hspace{0.5cm} M_\infty=e^{i\pi(b-a-1)\sigma_3},
\end{equation}
where $C_{\infty 0}, C_{0 \infty}, C_{\infty 1}, C_{1 \infty}$ are defined in \eqref{CInf0}, \eqref{C0Inf}, \eqref{CInf1}, \eqref{C1Inf}, respectively.  With some effort it can be shown that 
\begin{align}
M_0&=\begin{bmatrix} \cos\pi c\left(1-2i\frac{\sin\pi a\sin\pi b}{\sin\pi(b-a)}\right) & \frac{2\pi i\Gamma(b-a)\Gamma(b-a+1)}{\Gamma(b)\Gamma(c-a)\Gamma(b-c+1)\Gamma(1-a)} \\ \frac{2\pi i\Gamma(a-b)\Gamma(a-b+1)}{\Gamma(a)\Gamma(c-b)\Gamma(a-c+1)\Gamma(1-b)} & \cos\pi c\left(1+2i\frac{\sin\pi a\sin\pi b}{\sin\pi(b-a)}\right) \end{bmatrix}+\sin\pi c\frac{\sin\pi(a+b)}{\sin\pi(b-a)}\begin{bmatrix} 1 & 0 \\ 0 & 1 \end{bmatrix}, \\
M_1&=e^{\frac{i\pi}{2}(b-a)\sigma_3}M_0\big|_{c\to a+b-c+1}e^{-\frac{i\pi}{2}(b-a)\sigma_3}.
\end{align}
From the previous section, we take $c=0$ and $b=-1-a$.  It is important to note that the connection matrices $C_{\infty 0}, C_{0 \infty}, C_{\infty 1}, C_{1 \infty}$ are singular when $c=0$ and/or $b=-1-a$, but we can see that $M_0, M_1$ are not.  Taking $c=0$ and $b=-a-1$, we obtain
\begin{align}
M_0&=\begin{bmatrix} 1-i\tan(a\pi) & \frac{\tan^2(a\pi)\Gamma(a)\Gamma(a+2)}{i4^{2a+1}\Gamma(a+\frac{1}{2})\Gamma(a+\frac{3}{2})} \\ \frac{i4^{2a+1}\Gamma(a+\frac{1}{2})\Gamma(a+\frac{3}{2})}{\Gamma(a)\Gamma(a+2)} & 1+i\tan(a\pi) \end{bmatrix}, \\
M_1&=\sigma_3e^{-i\pi a\sigma_3}M_0e^{i\pi a\sigma_3}\sigma_3, \\
M_\infty&=e^{-2\pi ia\sigma_3}.
\end{align}
Let 
\begin{equation}\label{Q}
Q(\lambda)=\begin{bmatrix} -\tan(a\pi) & 0 \\ 0 & 4^{2a+1}e^{a\pi i}\frac{\Gamma(a+3/2)\Gamma(a+1/2)}{\Gamma(a)\Gamma(a+2)} \end{bmatrix}\begin{bmatrix} 1 & e^{a\pi i} \\ -e^{a\pi i} & 1 \end{bmatrix}
\end{equation}
so then we have 
\begin{align}
Q^{-1}M_0Q&=\begin{bmatrix} 1 & 0 \\ -\frac{e^{2\pi ia}-1}{e^{a\pi i}} & 1 \end{bmatrix}, ~~~ Q^{-1}M_1Q=\begin{bmatrix} 1 & -\frac{e^{2\pi ia}-1}{e^{a\pi i}} \\ 0 & 1 \end{bmatrix}, \label{M0M1jump} \\
Q^{-1}M_\infty Q&=\begin{bmatrix} \frac{e^{4\pi ia}-e^{2\pi ia}+1}{e^{2\pi ia}} & \frac{1-e^{2\pi ia}}{e^{a\pi i}} \\ \frac{1-e^{2\pi ia}}{e^{a\pi i}} & 1 \end{bmatrix}
\end{align}
The match requires 
\begin{equation}
\frac{e^{2\pi ia}-1}{e^{a\pi i}}=\frac{i}{\lambda}
\end{equation}
which implies
\begin{equation}\label{aConstruct}
a(\lambda)=\frac{1}{i\pi}\ln\left(\frac{i+\sqrt{4\lambda^2-1}}{2\lambda}\right).
\end{equation}
In Appendix \ref{aAppendix} we have listed all the important properties of $a(\lambda)$.

\subsection{Proof of Theorem \ref{ThmRHPSol}}

We will now construct $\Gamma(z;\lambda)$ so that it is the solution of RHP \ref{Gamma4RHP}.

\begin{proof}
Notice that the matrix 
\begin{equation}
    \begin{bmatrix} 1 & \frac{b_Lb_R}{z(b_R-b_L)(a+1)} \\ 0 & 1 \end{bmatrix}\hat{\Gamma}(M_1(z))Q(\lambda)\sigma_2,
\end{equation}
where $\hat{\Gamma}, M_1, a$ are defined in \eqref{hatGammaConstruct}, \eqref{m1Construct}, \eqref{aConstruct}, respectively, satisfies the following properties:

\begin{itemize}
    \item $L^2$ behavior at $z=0$, provided $\lambda\not\in[-1/2,1/2]$, due to \eqref{hatGammaL2Loc} and properties of $a(\lambda)$ (see Appendix \ref{aAppendix}),
    \item jump matrix $\begin{bmatrix} 1 & -\frac{i}{\lambda} \\ 0 & 1 \end{bmatrix}$ for $z\in(b_L,0)$ with positive orientation, see \eqref{M0M1jump} and Figure \ref{figRealAxisZ} for orientation,
    \item jump matrix $\begin{bmatrix} 1 & 0 \\ \frac{i}{\lambda} & 1 \end{bmatrix}$ for $z\in(0,b_R)$ with positive orientation, see \eqref{M0M1jump},
    \item column-wise behavior $\begin{bmatrix} \BigO{1} & \BigO{\ln(z-b_L)} \end{bmatrix}$ as $z\to b_L$, because the first column has no jump on $(b_L,0)$ and is analytic for $z\not\in[b_L,b_R]$ so the first column is $\BigO{1}$ as $z\to b_L$.  The Sokhotski-Plemelj formula can be used to inspect the behavior of the second column.
    \item column-wise behavior $\begin{bmatrix} \BigO{\ln(z-b_R)} & \BigO{1} \end{bmatrix}$ as $z\to b_R$, same idea as above,
    \item behavior $\hat{\Gamma}(M_1(\infty))Q(\lambda)\sigma_2(\1+\BigO{z^{-1}})$ as $z\to\infty$,
    \item analytic for $z\in\mathbb{C}\setminus[b_L,b_R]$, due to properties of hypergeometric functions.
\end{itemize}

Thus, we conclude that the matrix 
\begin{equation}
    \Gamma(z;\lambda):=\sigma_2Q^{-1}(\lambda)\hat{\Gamma}^{-1}(M_1(\infty))\begin{bmatrix} 1 & \frac{b_Lb_R}{z(b_R-b_L)(a+1)} \\ 0 & 1 \end{bmatrix}\hat{\Gamma}(M_1(z))Q(\lambda)\sigma_2
\end{equation}

is a solution of RHP \ref{Gamma4RHP} provided that $\lambda\notin[-1/2,1/2]$.  Uniqueness follows immediately from the $L^2$ behavior of $\Gamma(z;\lambda)$ at the endpoints $z=b_L,0,b_R$.
\end{proof}

\section{Definition and properties of $a(\lambda)$}\label{aAppendix}

Recall, from \eqref{aFunc}, that 
\begin{equation}
a(\lambda)=\frac{1}{i\pi}\ln\left(\frac{i+\sqrt{4\lambda^2-1}}{2\lambda}\right)
\end{equation}
where $\sqrt{4\lambda^2-1}=2\lambda+\BigO{1}$ as $\lambda\to\infty$ and the principle branch of the logarithm is taken.  The following proposition lists all relevant properties of $a(\lambda)$, none of which are difficult to prove.

\begin{proposition}\label{aProp}
The function $a(\lambda)$ has the following properties:

\begin{enumerate}

    \item $a(\lambda)$ is analytic and Schwarz symmetric for $\lambda\in\mathbb{C}\setminus[-1/2,1/2]$.
    
    \item $a_+(\lambda)+a_-(\lambda)=\begin{cases} 1, &\lambda\in(0,1/2) \\ -1, &\lambda\in(-1/2,0)\end{cases}$
    
    \item $\Re[a_\pm(\lambda)]=\begin{cases} \frac{1}{2}, &\lambda\in(0,1/2) \\ -\frac{1}{2}, &\lambda\in(-1/2,0) \end{cases}$ and $-\frac{1}{2}<\Re\left[a(\lambda)\right]<\frac{1}{2}$ for $\lambda\in\overline{\mathbb{C}}\setminus\left[-\frac{1}{2},\frac{1}{2}\right]$.
    
    \item  For $\lambda\in(-1/2,1/2)$, $\Im[a_+(\lambda)]<0$, $\Im[a_-(\lambda)]>0$, and $\Im[a_+(\lambda)]=-\Im[a_-(\lambda)]$.
    
    \item If $\lambda\to0$, provided that either $\Im\lambda\geq0$ or $\Im\lambda\leq0$, then 
        \begin{equation}\label{aAsymp}
            a(\lambda)=\begin{cases}
                -\frac{1}{i\pi}\ln(\lambda)+\frac{1}{2}+\BigO{\lambda^2}=\frac{\varkappa}{i\pi}+\frac{1}{2}+\BigO{e^{-2\varkappa}}, &\Im\lambda\geq0, \\
                \frac{1}{i\pi}\ln(\lambda)+\frac{1}{2}+\BigO{\lambda^2}=-\frac{\varkappa}{i\pi}+\frac{1}{2}+\BigO{e^{-2\varkappa}}, &\Im\lambda\leq0,
            \end{cases}
        \end{equation}
        
        where $\varkappa=-\ln\lambda$.
        
    \item As $\lambda\to0$, provided that either $\Im\lambda\geq0$ or $\Im\lambda\leq0$, then
        \begin{equation}\label{AAsymp}
            e^{i\pi a(\lambda)}=\begin{cases} ie^{\varkappa}(1+\BigO{e^{-2\varkappa}}), &\Im\lambda\geq0, \\ ie^{-\varkappa}(1+\BigO{e^{-2\varkappa}}), &\Im\lambda\leq0. \end{cases}
        \end{equation}
    
\end{enumerate}

\end{proposition}

\section{Properties of $d_L(z;\lambda),d_R(z;\lambda)$}\label{dldr_appendix}

In this Appendix we compute and list the important properties of the functions $d_L(z;\lambda)$ and $d_R(z;\lambda)$.  Recall that (from eq. \eqref{drdl})
\begin{align}
    d_R(z;\lambda)&:=\alpha(\lambda)h_\infty '\left(\frac{b_R(z-b_L)}{z(b_R-b_L)}\right)+\beta(\lambda)s_\infty '\left(\frac{b_R(z-b_L)}{z(b_R-b_L)}\right), \\ d_L(z;\lambda)&:=-e^{a\pi i}\alpha(\lambda)h_\infty '\left(\frac{b_R(z-b_L)}{z(b_R-b_L)}\right)+e^{-a\pi i}\beta(\lambda)s_\infty '\left(\frac{b_R(z-b_L)}{z(b_R-b_L)}\right),
\end{align}
where 
\begin{align}
    \alpha(\lambda):=\frac{e^{-a\pi i}\tan(a\pi)\Gamma(a)}{4^{a+1}\Gamma(a+3/2)}, ~~ \beta(\lambda):=\frac{4^{a}e^{a\pi i}\Gamma(a+1/2)}{\Gamma(a+2)}
\end{align}
and $h_\infty',s_\infty'$ are defined in \eqref{hhprime}, \eqref{ssprime}.

\begin{proposition}\label{dldr_AppendixProp}
    For $\lambda\in(-1/2,0)$,    
    \begin{enumerate}
    
        \item $d_R\left(M_2(x);\lambda\right)=d_L(x;\lambda)$, where $M_2(x)=\frac{b_Rb_Lx}{x(b_R+b_L)-b_Rb_L}$,
    
        \item $d_L(z;\lambda)$ and $d_R(z;\lambda)$ are single-valued in $\lambda$, 
        
        \item $d_L(z;\lambda)$ is analytic for $z\in\overline{\mathbb{C}}\setminus[0,b_R]$, $\overline{d_L(z;\lambda)}=d_L(\overline{z};\overline{\lambda})$ for $z\in\overline{\mathbb{C}}$, and 
        \begin{equation}
            d_L(z_+;\lambda)-d_L(z_-;\lambda)=\frac{i}{\lambda}d_R(z;\lambda), ~~~ \text{for $z\in(0,b_R)$,}
        \end{equation}

        \item $d_R(z;\lambda)$ is analytic for $z\in\overline{\mathbb{C}}\setminus[b_L,0]$, $\overline{d_R(z;\lambda)}=d_R(\overline{z};\overline{\lambda})$ for $z\in\overline{\mathbb{C}}$, and 
        \begin{equation}
            d_R(z_+;\lambda)-d_R(z_-;\lambda)=-\frac{i}{\lambda}d_L(z;\lambda), ~~~ \text{for $z\in(b_L,0)$,}
        \end{equation}

        \item We have the SVD system 
        \begin{align}
            \mathcal{H}_R\left[\frac{d_R(y;\lambda)}{y}\right](x)=2\lambda\frac{d_L(x;\lambda)}{x}, ~~~ \mathcal{H}_L\left[\frac{d_L(x;\lambda)}{x}\right](y)=2\lambda\frac{d_R(y;\lambda)}{y}
        \end{align}
        where $x\in(b_L,0)$ and $y\in(0,b_R)$.

        \item $\ds{\int_0^{b_R}\frac{d_R(x;\lambda)}{x}~dx=2\pi\lambda d_L(\infty;\lambda)}$, $\ds{\int_{b_L}^0\frac{d_L(x;\lambda)}{x}~dx=-2\pi\lambda d_R(\infty;\lambda)}$
        
    \end{enumerate}
    
    We obtain the corresponding identities for $D_L(z;\lambda), D_R(z;\lambda)$ when $\lambda\in(-1,1)$ via the relation $D_R(z;\lambda):=d_R(z;-|\lambda|/2)$ and $D_L(z;\lambda):=d_L(z;-|\lambda|/2)$.

\end{proposition}

\begin{proof}
    \begin{enumerate}
    
        \item Recall that $M_j(x)$, $j=1,2,3$, are defined in Remark \ref{mobius}.  It is easy to show that $1-M_3(x)=M_1(x)$, $1-\frac{1}{M_3(x)}=\frac{-M_1(x)}{M_3(x)}$.  From \cite{AS64} 15.5.7, 15.5.13, (taking $a:=a_-(\lambda)$ and $-1=e^{-i\pi}$) we obtain
\begin{align}
    \pFq{2}{1}{a+1,a+1}{2a+2}{\frac{1}{M_3(x)}}&=-e^{a\pi i}\left(\frac{M_1(x)}{M_3(x)}\right)^{-a-1}\pFq{2}{1}{a+1,a+1}{2a+2}{\frac{1}{M_1(x)}},
\end{align}
thus we have 
\begin{align}\label{hM3}
    h_\infty'(M_3(x))&=-a e^{a\pi i}M_3(x)^{-a-1}\pFq{2}{1}{a+1,a+1}{2a+2}{\frac{1}{M_3(x)}}=-e^{a\pi i}h_\infty'(M_1(x)).
\end{align}
Taking $a\to-a-1$ in the identity above yields 
\begin{equation}\label{sM3}
    s_\infty'(M_3(x))=e^{-a\pi i}s_\infty'(M_1(x)),
\end{equation}
because $h_\infty(\eta)|_{a\to-a-1}=s_\infty(\eta)$.  Now using the definition of $d_R(x;\lambda)$, \eqref{hM3}, \eqref{sM3}, and that $M_1(M_2(x))=M_3(x)$, we obtain the result.

        \item Notice that $h_\infty(z)\big|_{a\to-a-1}=s_\infty(z)$ so we have $h_\infty'(z,\lambda_+)=s_\infty'(z,\lambda_-)$.  To compute the jumps of the coefficients $\alpha,\beta$, we use \cite{DLMF} 5.5.3 to obtain $\alpha_+(\lambda)=\beta_-(\lambda)$ and, similarly, $\beta_+(\lambda)=\alpha_-(\lambda)$.  This can now be used to prove the statement.

        \item Recall, from the proof of Theorem \ref{GammaJumpLambdaThm}, that $d_L(z;\lambda), d_R(z;\lambda)$ was defined in terms of the $(2,1), (2,2)$ element, respectively, of the matrix 
        \begin{equation}
            M(z,\lambda)=\hat{\Gamma}(M_1(z),\lambda)Q(\lambda)\sigma_2.
        \end{equation}
        Using the definition of $\Gamma(z;\lambda)$, see \eqref{RHPSolution}, some simple algebra shows that 
        \begin{equation}
            M(z,\lambda)=\begin{bmatrix} 1 & \frac{-b_Lb_R}{z(b_R-b_L)(a+1)} \\ 0 & 1 \end{bmatrix}\hat{\Gamma}^{-1}(M_1(\infty))Q\sigma_2\Gamma(z;\lambda).
        \end{equation}
        Since $\Gamma(z;\lambda)$ is a solution of RHP \ref{Gamma4RHP}, we know $M(z,\lambda)$ is analytic for $z\in\overline{\mathbb{C}}\setminus[b_L,b_R]$ and 
        \begin{equation}
            M_+=M_-\begin{bmatrix} 1 & -\frac{i}{\lambda} \\ 0 & 1 \end{bmatrix},~~z\in(b_L,0), ~~~~~ M_+=M_-\begin{bmatrix} 1 & 0 \\ \frac{i}{\lambda} & 1 \end{bmatrix},~~z\in(0,b_R),
        \end{equation}
        which immediately gives the jumps and analyticity of both $d_L,d_R$.  For the symmetry, notice that, by definition,
        \begin{align}
            d_R(z;\lambda)&=\frac{-a\tan(a\pi)\Gamma(a)}{4^{a+1}\Gamma(a+\frac{3}{2})}M_1(z)^{-a-1}\pFq{2}{1}{a+1, a+1}{2a+2}{\frac{1}{M_1(z)}} \nonumber \\
            &-\frac{(a+1)4^a\Gamma(a+\frac{1}{2})}{\Gamma(a+2)}M_1(z)^a\pFq{2}{1}{-a,-a}{-2a}{\frac{1}{M_1(z)}},
        \end{align}
        and $a=a(\lambda), M_1(z)$ are Schwarz symmetric.  Thus $\overline{d_R(z;\lambda)}=d_R(\overline{z};\overline{\lambda})$ and the symmetry of $d_L(z;\lambda)$ follows from the relation $d_R(M_2(x);\lambda)=d_L(x;\lambda)$.

        \item This was proven above.

        \item Let $\gamma_R$ be the circle with center and radius of $b_R/2$ with negative orientation.  Then,
        \begin{align}
            \mathcal{H}_R\left[\frac{d_R(y;\lambda)}{y}\right](x)&=\frac{2\lambda}{2\pi i}\int_0^{b_R}\frac{\frac{i}{\lambda}d_R(y;\lambda)}{y(y-x)}~dy=\frac{2\lambda}{2\pi i}\int_0^{b_R}\frac{\Delta_yd_L(y;\lambda)}{y(y-x)}~dy \cr
            &=\frac{2\lambda}{2\pi i}\int_{\gamma_R}\frac{d_L(y;\lambda)}{y(y-x)}~dy=2\lambda d_L(x;\lambda)
        \end{align}
        where we have deformed $\gamma_R$ through $z=\infty$ and into the circle of center and radius $b_L/2$ with positive orientation and then apply residue theorem.  The remaining computation is nearly identical.

        \item The idea is similar to that of the last proof;
        \begin{align}
            \int_0^{b_R}\frac{d_R(x;\lambda)}{x}~dx&=\frac{2\pi\lambda}{2\pi i}\int_0^{b_R}\frac{\frac{i}{\lambda}d_R(x;\lambda)}{x}~dx=\frac{2\pi\lambda}{2\pi i}\int_{\gamma_R}\frac{d_L(x;\lambda)}{x}~dx=2\pi\lambda d_L(\infty;\lambda),
        \end{align}
        and the remaining identity is proven analogously.
    \end{enumerate}
\end{proof}

\section{Proof of Theorem \ref{result}}\label{appendix_saddlePt}

\subsection{Deformation of $[0,1]$ to $\gamma_\eta$}

For any $\eta\in\Omega_+$, we want the path $\gamma_\eta$ to have the property 
\begin{equation}
    \Re\left[S_\eta(t)-S_\eta(t^*_-(\eta))\right]\leq0
\end{equation}
for all $t\in\gamma_\eta$ with equality only when $t=t_-^*(\eta)$, so it is important to  understand the collection of the level   curves of 
\begin{equation}
    \Re\left[S_\eta(t)-S_\eta(t^*_-(\eta))\right]=0,
\end{equation}
which obviously depends on $\eta$.

\begin{lemma}\label{levSetProp}
    For each $\eta\in\Omega_+$ \nonumber \\ we have the following:
    \begin{enumerate}
        \item There is exactly one curve $l_0$ emitted from $t=0$.  Moreover, $l_0$ lies in the sector $|\arg(t)|\leq\pi/4$ for 
         sufficiently small $|t|$.
        \item There is exactly one curve $l_1$ emitted from $t=1$.  Moreover, $l_0$ lies in the sector $|\arg(1-t)|\leq\pi/4$ for 
         sufficiently small $|1-t|$.
        \item There exists exactly one $\theta=\theta_{u,\eta}\in(0,\pi)$ such that $\Re\left[S_\eta(1/2+1/2e^{i\theta_{u,\eta}})\right]=\Re\left[S_\eta(t^*_-(\eta))\right]$.  Moreover, for $M$ sufficiently large, $\theta_{u,\eta}\in(\pi/4,3\pi/4)$.
        \item There exists exactly one $\theta=\theta_{l,\eta}\in(-\pi,0)$ such that $\Re\left[S_\eta(1/2+1/2e^{i\theta_{l,\eta}})\right]=\left[S_\eta(t^*_-(\eta))\right]$.  Moreover, for $M$ sufficiently large, $\theta_{l,\eta}\in(-3\pi/4,-\pi/4)$.
    \end{enumerate}
\end{lemma}

\begin{figure}[h]
\begin{center} 

\begin{tikzpicture}[scale=0.75]

\draw[red,thick] (-3,0) .. controls (0,-0.25) .. (3,0);

\draw [thick,smooth,red] (0.5,-3.31) .. controls (-0.16,-0.19) .. (0.5,3.59);

\draw[blue,thick] (3,0) -- (4.5,0);
\draw[blue,thick] (-4.5,0) -- (-3,0);
\filldraw[black] (-3,0) circle [radius=2pt] node[anchor=north east] {0};
\filldraw[black] (0,-0.19) circle [radius=2pt] node[anchor=south east] {$t_-^*(\eta)$};
\filldraw[black] (3,0) circle [radius=2pt] node[anchor=north west] {1};

\draw[black,dashed] (0,0) circle [radius=85.4pt];

\node[] at (-2.5,0.25) {$+$};
\node[] at (-2.5,-0.25) {$-$};

\node[] at (2.5,-0.25) {$+$};
\node[] at (2.5,0.25) {$-$};

\end{tikzpicture}
\end{center}
\vspace{-5mm}
\caption{This is a visualization of Lemma \ref{levSetProp}.  The blue lines are the branch cuts of $S_\eta(t)$, the red curves form the level set $\Re\left[S_\eta(t)-S_\eta(t^*_-(\eta))\right]=0$ and the black dashed circle has center and radius $1/2$.  The $+$ denotes regions where $\Re\left[S_\eta(t)-S_\eta(t^*_-(\eta))\right]>0$ and $-$ denotes regions where $\Re\left[S_\eta(t)-S_\eta(t^*_-(\eta))\right]<0$.} \label{figZeroLevelSet}
\end{figure}
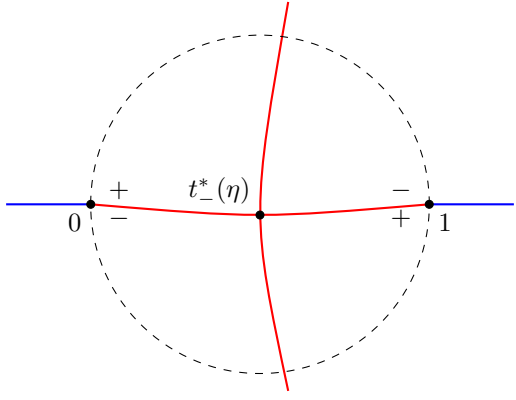

\begin{proof}

For brevity, define
\begin{align}\label{alpha}
    \alpha_\eta(t)&:=\Re\left[S_\eta(t)-S_\eta(t^*_-(\eta))\right]=\arg\left(\frac{\eta t(1-t)}{\eta-t}\right)-2\arg\left(t^*_-(\eta)\right).
\end{align}
The statements listed above are equivalent to
\begin{enumerate}
        \item For any $r>0$ sufficiently small, $\alpha_\eta(re^{i\theta})$ is increasing for $-\pi<\theta<\pi$ and ${\exists!\theta_{0,\eta,r}\in(-\pi/4,\pi/4)}$ such that ${\alpha_\eta(re^{i\theta_{0,\eta,r}})=0}$,
        
        \item For any $r>0$ sufficiently small, $\alpha_\eta(1-re^{i\theta})$ is increasing for $-\pi<\theta<\pi$ and ${\exists!\theta_{1,\eta,r}\in(-\pi/4,\pi/4)}$ such that ${\alpha_\eta(1-re^{i\theta_{1,\eta,r}})=0}$,
        
        \item $\alpha_\eta\left(\frac{1}{2}+\frac{1}{2}e^{i\theta}\right)$ is increasing for $0<\theta<\pi$ and $\exists!\theta_{u,\eta}\in(\pi/4,3\pi/4)$ such that ${\alpha_\eta\left(\frac{1}{2}+\frac{1}{2}e^{i\theta_{u,\eta}}\right)=0}$,
        
        \item $\alpha_\eta\left(\frac{1}{2}+\frac{1}{2}e^{i\theta}\right)$ is increasing for $-\pi<\theta<0$ and $\exists!\theta_{l,\eta}\in(-3\pi/4,-\pi/4)$ such that ${\alpha_\eta\left(\frac{1}{2}+\frac{1}{2}e^{i\theta_{l,\eta}}\right)=0}$.
        
    \end{enumerate}
The proofs of each of the four claims above are similar so we prove the first.
\begin{align}
\alpha_\eta(re^{i\theta})&=\arg(re^{i\theta})+\arg(1-re^{i\theta})-\arg(\eta-re^{i\theta})+\arg(\eta)-2\arg\left(t^*_-(\eta)\right) \nonumber \\
&\equiv\theta+\tan^{-1}\left(\frac{r\sin\theta}{r\cos\theta-1}\right)-\tan^{-1}\left(\frac{r\sin\theta-\Im\eta}{r\cos\theta-\Re\eta}\right)+\arg(\eta)-2\arg\left(t^*_-(\eta)\right) \pmod{\pi} \nonumber 
\end{align}
Differentiating with respect to $\theta$, we have
\begin{align}
\frac{d}{d\theta}\left[\alpha_\eta(re^{i\theta})\right]&=1+\frac{r^2-r\cos\theta}{(r\cos\theta-1)^2+r^2\sin^2\theta}-\frac{r^2-r(\Re(\eta)\cos\theta+\Im(\eta)\sin\theta)}{(r\cos\theta-\Re\eta)^2+(r\sin\theta-\Im\eta)^2}\to1
\end{align}
as $r\to0$.  So with $r$ small enough, $\frac{d}{d\theta}\left[\alpha_\eta(re^{i\theta})\right]>0$ for any $\theta$ and for any $\eta\in\Omega_+$.  Take $M$ sufficiently large so that 
\begin{equation}\label{Mcondit ion1}
    |\arg\left(t^*_-(\eta)\right)|<\pi/8
\end{equation}
for all $\eta\in\Omega_+$.  Thus when $r$ is sufficiently small,
\begin{align}
    \alpha_\eta(re^{i\pi/4})&=\frac{\pi}{4}+\tan^{-1}\left(\frac{r/\sqrt{2}}{r/\sqrt{2}-1}\right)-\arg\left(1-\frac{re^{i\pi/4}}{\eta}\right)-2\arg\left(t^*_-(\eta)\right)>0
\end{align}
and similarly it can be shown that $ \alpha_\eta(re^{-i\pi/4})<0$.  Intermediate Value Theorem can now be applied to show $\exists!\theta_{0,\eta,r}\in(-\pi/4,\pi/4)$ such that $\alpha_\eta(re^{i\theta_{0,\eta,r}})=0$.

\end{proof}

\begin{remark}
    According to Lemma \ref{levSetProp}, for every $\eta\in\Omega_+$, the set $t\in B(1/2,1/2)$ is split into 4 sectors by the level curve $\Re\left[S_\eta(t)-S_\eta(t^*_-(\eta))\right]=0$, see Figure \ref{figZeroLevelSet}.
\end{remark}

We have now proven that we can deform $[0,1]$ to an `appropriate' path $\gamma_\eta$, as described in the following Theorem.

\begin{theorem} 
    For each $\eta\in\Omega_+$, the path $[0,1]$ can be continuously deformed to a path $\gamma_\eta$ which begins at $t=0$ (with $\emph{arg}(t)<-\pi/4$ for small $|t|$), passes through $t=t_-^*(\eta)$, ends at $t=1$ (with $\emph{arg}(t-1)<3\pi/4$ for small $|t-1|$), and ${\gamma_\eta\subset B(1/2,1/2)}$.  Moreover,  $\Re\left[S_\eta(t)-S_\eta(t^*_-(\eta))\right]\leq0$ for all $t\in\gamma_\eta$ with equality only when $t=t_-^*(\eta)$.  See \eqref{OmegaPlus}, \eqref{S} for $\Omega_+, S_\eta(t)$, respectively.
\end{theorem}

\begin{proof}
 Let us fix some $\eta\in\Omega_+$. The function $\Re S_\eta(t)$ is a harmonic function in $t\in  B(1/2,1/2)$, so the  level curves 
 $\Re S_\eta(t)=\Re S_\eta(t^*_-(\eta))$ can not form closed loops there. Since, according to Lemma \ref{levSetProp}, there exactly four 
 level curves  $\Re S_\eta(t)=\Re S_\eta(t^*_-(\eta))$ entering the disc  $  B(1/2,1/2)$, and exactly four legs of this level curve
 emanating from $t=t^*_-(\eta)$, the only possible topology of these level curves is shown on Figure \ref{figZeroLevelSet}. We can now deform 
 $[0,1]$ to an `appropriate' path $\gamma_\eta$ as stated in the theorem. The new path of integration, $\gamma_\eta$, is the black curve in Figure \ref{figGammaPathSplit}.
\end{proof}

\subsection{Estimates}

Here we will split $\gamma_\eta$ into 3 pieces (see Figure \ref{figGammaPathSplit}) and show that the leading order contribution in Theorem \ref{result} comes from the piece containing $t^*_-(\eta)$.  Define radius
\begin{equation}\label{r}
r:=r_M=\max_{\eta\in\Omega_+}\left|\frac{1}{2}-t^*_-(\eta)\right|=O(M^{-1})
\end{equation}
so that $t^*_-(\eta)\in \overline{B}(1/2,r)$ for all $\eta\in\Omega_+$.   The latter estimate follows from Prposition \ref{saddlepts}.  Also define function 
\begin{equation}\label{v_func}
v_\eta(t):=v\left(t,\eta\right)=\sqrt{S_\eta(t^*_-(\eta))-S_\eta(t)}
\end{equation}
where the square root is defined so that $\Re v_\eta(t)>0$ along the path $\gamma_\eta$ lying between $t^*_-(\eta)$ and 1.  The function $v_\eta(t)$ will be an essential change of variables in the integral \eqref{int} so we mention its important properties.
\begin{lemma}\label{v_prop}
If $M>0$ is sufficiently large, the function $v_\eta(t)$ for every $\eta\in\Omega_+$ is: 1)  biholomorphic in  $t\in B(1/2,2r)$; 
2) satisfies the condition
\begin{equation}\label{v_biLip}
    \left|v_\eta(t_1)-v_\eta(t_2)\right|=  (2+O(M^{-1}))\left|t_1-t_2\right|
\end{equation}
for any $t_1,t_2\in B(1/2,2r)$. Moreover, if $I_\eta$ is the image of $B(1/2,2r)$ under the map $v_\eta(t)$,  then there exists a $\delta_*>0$ such that $\ds{{B(0,\delta_*)\subset\bigcap_{\eta\in\Omega_+}I_\eta}}$.
\end{lemma}

\begin{proof}
Statement 1) follows directly from 
\begin{equation}\label{veta}
   v_\eta(t)=\left(t-t^*_-(\eta)\right)\sqrt{-\frac{1}{2}S_\eta''(t^*_-(\eta))+\BigO{t-t^*_-(\eta)}}, ~~ t\in B(1/2,2r).
\end{equation}
To prove statement 2), we observe that Proposition \ref{saddlepts} implies $   S_\eta ''(t^*_-(\eta))=8i+O(M^{-1})  $. Then,
taking into account \eqref{veta},
\begin{equation}
 v'_\eta(t)=\frac{-S_\eta'(t)}{2v_\eta(t)}=\sqrt{-\frac{1}{2}S_\eta''(t^*_-(\eta))}+O(M^{-1}),
 \end{equation}
in $ B(1/2,2r)$, so that $|v'_\eta(t)|=2+O(M^{-1})$ there. That implies \eqref{v_biLip}.  To prove the last statement, define the function 
\begin{equation}
\delta_*(\eta)=\min \{v\, : ~ v\in \partial I_\eta\}
\end{equation}
for all  $\eta\in\Omega_+$. Since $\delta_*(\eta)$ is a continuous positive function on the compact boundary $\partial  \Omega_+$,
it must attain its minimal value $\delta_*>0$.
\end{proof}

Let $\delta=\frac{\delta_*}2$ and define
\begin{equation}
    t^*_L(\eta):=v^{-1}_\eta(-\delta), ~~~ t^*_R(\eta):=v^{-1}_\eta(\delta).
\end{equation}
With the previous Lemma in mind, we now write ${\gamma_\eta=\gamma_{L,\eta}+\gamma_{C,\eta}+\gamma_{R,\eta}}$, where $\gamma_{C,\eta}$ is the image of $[-\delta,\delta]$ under the map $v^{-1}_\eta$, $\gamma_{L,\eta}$ is the portion of $\gamma$ beginning at $t=0$ and ending at $t=t^*_L(\eta)$, and $\gamma_{L,\eta}$ is the portion of $\gamma$ beginning at $t=t^*_R(\eta)$ and ending at $t=1$.  The curves $\gamma_{L,\eta},\gamma_{C,\eta},\gamma_{R,\eta}$ are pictured in Figure \ref{figGammaPathSplit}.
\begin{figure}
\begin{center}
\begin{tikzpicture}

\draw[blue,thick] (-4.5,0) -- node[anchor=south,black] {$+$} node[anchor=north,black] {$-$} (-3,0);
\draw[red,thick] (-3,0) -- (3,0);
\draw[blue,thick] (3,0) -- node[anchor=south,black] {$-$} node[anchor=north,black] {$+$} (4.5,0);
\draw [thick,smooth,red] (0.5,-2) .. controls (-0.15,0) .. (0.5,2);
\draw [thick,smooth] (-3,0) .. controls (-1,-0.88) .. (0,0);
\draw [thick,smooth] (0,0) .. controls (1,0.88) .. (3,0);

\filldraw[black] (-3,0) circle [radius=1.5pt] node[anchor=south] {0};
\filldraw[black] (0,0) circle [radius=1.5pt] node[anchor=south east] {$t^*_-(\eta)$};
\filldraw[black] (3,0) circle [radius=1.5pt] node[anchor=north] {1};

\node[] at (-2.25,-0.65) {$\gamma_{L,\eta}$};
\node[] at (2.25,0.65) {$\gamma_{R,\eta}$};

\filldraw[black] (0.7,0.55) circle [radius=1.5pt] node[anchor=south] {$t^*_R(\eta)$};
\filldraw[black] (-0.7,-0.55) circle [radius=1.5pt] node[anchor=north] {$t^*_L(\eta)$};

\end{tikzpicture}
\end{center}
\caption{\hspace{.05in}$\gamma_\eta=\gamma_{L,\eta}+\gamma_{C,\eta}+\gamma_{R,\eta}$}\label{figGammaPathSplit}
\end{figure}
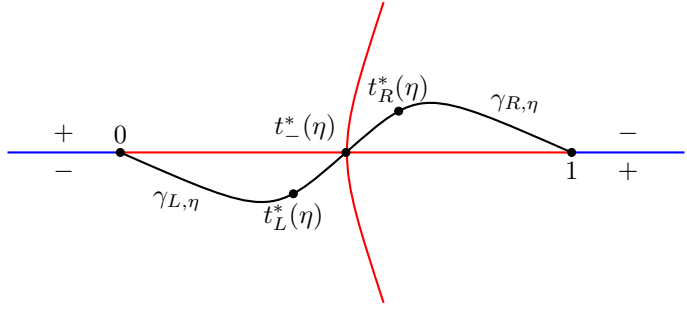

\begin{remark}
Notice that
\begin{equation}
    \left|t_-^*(\eta)-t_R^*(\eta)\right|=\left|v_\eta^{-1}(0)-v_\eta^{-1}(\delta)\right|\geq\frac{2}{9}|\delta-0|>0
\end{equation}
for all $\eta\in\Omega_+$ by use of \eqref{v_biLip}.  An identical statement holds for $t_L^*(\eta)$.
\end{remark}

\begin{lemma}\label{intLeading}

Choose sufficiently large $M$. Then,
under the hypotheses of Theorem \ref{result}
\begin{align}
    &\int_{\gamma_{C,\eta}}F(t,\eta,\lambda)e^{-\Im[a(\lambda)]S_\eta(t)}dt \nonumber \\
    &=e^{-\Im[a(\lambda)]S_\eta(t^*_-(\eta))}F\left(t^*_-(\eta),\eta,\lambda\right)\sqrt{\frac{2\pi}{\Im[a(\lambda)]S''_\eta(t^*_-(\eta))}}
    \left[1+\BigO{\frac{M^2}{\Im[a(\lambda)]}}\right]
\end{align}
as $\lambda\to0$, $\Im\lambda\geq0$,  uniformly in $\eta\in\Omega_+$.

\end{lemma}

\begin{proof}

By definition,
$\gamma_{C,\eta}$ is the path of steepest descent for $S_\eta(t)-S_\eta(t^*_-(\eta))$ and thus $\Im[S_\eta(t)-S_\eta(t^*_-(\eta))]=0$ and $\Re[S_\eta(t)-S_\eta(t^*_-(\eta))]\leq0$ on $\gamma_{C,\eta}$ with equality only when $t=t_-^*(\eta)$.  Using change of variables 
\begin{align}
v&=v_\eta(t)=\sqrt{S_\eta(t^*_-(\eta))-S_\eta(t)},
\end{align}
we obtain
\begin{align}
    \int_{\gamma_{C,\eta}}F(t,\eta,\lambda)e^{-\Im[a]S_\eta(t)}dt&=e^{-\Im[a]S_\eta(t^*_-)}\int_{\gamma_{C,\eta}}F(t,\eta,\lambda)e^{-\Im[a][S_\eta(t)-S_\eta(t^*_-)]}dt \cr
    &=e^{-\Im[a]S_\eta(t^*_-)}\int_{-\delta}^\delta e^{\Im[a]v^2}F(t(v),\eta,\lambda)\frac{dv}{v'_\eta(t(v))}. \label{changeOfVar1}
\end{align}
Since $v_\eta(t)$ is biholomorphic in $B(1/2,2r)$ for every $\eta\in\Omega_+$, both $\frac{dt}{dv}$ and $F(t(v),\eta,\lambda)$ are analytic for $v\in B(0,\delta_*)$ by Lemma \ref{v_prop}.  So 
\begin{equation}
    \frac{F(t(v),\eta,\lambda)}{v'_\eta(t(v))}=\sum_{n=0}^\infty b_{n}(\eta,\lambda)v^{n}=:b_0(\eta,\lambda)+vb_1(\eta,\lambda)+v^2\hat R_2(v,\eta,\lambda)
\end{equation}
where
\begin{align}
    b_0(\eta,\lambda)&=F(t^*_-(\eta),\eta,\lambda)\sqrt{\frac{2}{-S_\eta ''(t^*_-(\eta))}}, ~~~ b_{n}(\eta,\lambda)=\frac{1}{2\pi i}\int_{|w|=\delta}\frac{F(t(w),\eta,\lambda)}{v'_\eta(t(w)) w^{n+1}}dw
\end{align}
for $n\geq1$, where, by construction,  $\delta<\delta_*$. We have 
\begin{align}\label{R2eq}
    e^{-\Im[a]S_\eta(t^*_-)}\int_{-\delta}^\delta e^{\Im[a]v^2}F(t(v),\eta,\lambda) \frac{dv}{v'_\eta(t(v))} 
    =2e^{-\Im[a]S_\eta(t^*_-)}\int_0^\delta e^{\Im[a]v^2}\left[b_0(\eta,\lambda)+v^2 R_2(v,\eta,\lambda)\right]dv, 
\end{align}
where $R_2$ is the even part of  $\hat R_2$. 
By the assumptions of the Lemma,
\begin{equation}
K:=\max_{(\eta,\lambda)\in\Omega_+\times B^0(0,\delta_*)}\max_{v\in \partial{B}(0,\delta_*)}\left|\frac{F(t(v),\eta,\lambda)}{v'_\eta(t(v))}\right|.
\end{equation}
 $K$ is finite.  By Cauchy's estimate we now have 
\begin{equation}
    \left| b_{2n}(\eta,\lambda)\right|\leq\frac{K}{\delta_*^{2n}}
\end{equation}
so that for $v\in[0,\delta]$, 
\begin{align}
    \left|R_2(v,\eta,\lambda)\right|=\left|\sum_{n=1}^\infty b_{2n}(\eta,\lambda)v^{2(n-1)}\right|\leq\frac{K}{\delta_*^{2}}\sum_{n=1}^\infty\left(\frac{\delta}{\delta_*}\right)^{2(n-1)}\leq\frac{K}{\delta_*^2-\delta^2}=:R.
\end{align}
Note that $R=O(M^2)$. We fix a sufficiently large $M$.  Using a second change of variable  $-\tau=\Im[a]v^2$, we  rewrite \eqref{R2eq} as
\begin{align}
       &e^{-\Im[a]S_\eta(t^*_-)}\int_{-\delta}^\delta e^{\Im[a]v^2}F(t(v),\eta,\lambda) \frac{dv}{v'_\eta(t(v))}= \nonumber \\
    &\frac{e^{-\Im[a]S_\eta(t^*_-)}}{\sqrt{-\Im[a]}}\int_0^{-\Im[a]\delta^2}e^{-\tau}\left[\frac{b_0(\eta,\lambda)}{\sqrt{\tau}}-\frac{\sqrt{\tau}}{\Im[a]}R_2\left(t\left(\sqrt{\frac{-\tau}{\Im[a]}}\right),\eta,\lambda\right)\right]d\tau
\end{align}
Using now the asymptotics of the Incomplete Gamma function (see \cite{DLMF} 8.2.2, 8.2.11), we have 
\begin{equation}
    \int_0^{-\Im[a]\delta^2}e^{-\tau}\tau^{1/2-1}d\tau=\Gamma(1/2)-\Gamma(1/2,-\Im[a]\delta^2)=\sqrt{\pi}+\BigO{\sqrt{-\Im[a]}e^{\Im[a]\delta^2}}
\end{equation}
as $\lambda\rightarrow0$ and
\begin{align}
    &\left|\int_0^{-\Im[a]\delta^2}e^{-\tau}\tau^{3/2-1}R_2\left(t\left(\sqrt{\frac{-\tau}{\Im[a]}}\right),\eta,\lambda\right)d\tau\right| \cr
    &\leq R\int_0^{-\Im[a]\delta^2}e^{-\tau}\tau^{3/2-1}d\tau 
    =R\left[\Gamma(3/2)+\BigO{\left(-\Im[a]\right)^{3/2}e^{\Im[a]\delta^2}}\right]
\end{align}
with the error term being uniform with respect to $\eta,\lambda$.  Now we have
\begin{align}
     \int_{\gamma_{C,\eta}}F(t,\eta,\lambda)e^{-\Im[a(\lambda)]S_\eta(t)}dt=\frac{e^{-\Im[a]S_\eta(t^*_-)}}{\sqrt{-\Im[a]}}\left[b_0(\eta,\lambda)\sqrt{\pi}+\BigO{\frac{M^2}{\Im[a]}}\right],
\end{align}
which is equivalent to the statement of the Lemma since, by the hypotheses of the Lemma, $b_0(\eta,\lambda)$ is uniformly bounded away from 0.
\end{proof}

Next we show that the contribution away from the saddle point is negligible.

\begin{lemma}\label{intEstimate}
    Assume $F(t,\eta,\lambda)$ satisfies the hypotheses of Theorem \ref{result}.  Then, as $\lambda\to0$ with $\Im\lambda\geq0$, we have the bound 
    \begin{equation}
        \left|\int_{\gamma_{L+R,\eta}}F(t,\eta,\lambda)e^{-\Im[a(\lambda)]S_\eta(t)}~dt\right|\leq e^{-\Im[a(\lambda)]\Re[S_\eta(t^*_-(\eta))]}\cdot e^{-\Im[a(\lambda)]c_*} K_{LR},
    \end{equation}
    where $\gamma_{L+R,\eta}=\gamma_{L,\eta}\cup\gamma_{R,\eta}$  and constants $c_*<0$ and $ K_{LR}>0$ are $\eta,\lambda$ independent.
    
\end{lemma}

\begin{proof}

Define
\begin{equation}
    c_*(\eta)=\max_{t\in\gamma_{L+R,\eta}}\Re\left[S_\eta(t)-S_\eta(t_-^*(\eta))\right].
\end{equation}
Obviously  $c_*(\eta)$ is a continuous and negative function of  $\eta\in\Omega_+$.  Since $\Omega_+$ is a compact set,  the maximal value $c_*$ of $ c_*(\eta)$  on $\Omega_+$ is negative. Then
\begin{align}
    &\left|\int_{\gamma_{L+R,\eta}}F(t,\eta,\lambda)e^{-\Im[a(\lambda)]S_\eta(t)}~dt\right| 
    \leq e^{-\Im[a(\lambda)]\Re[S_\eta(t^*_-(\eta))]}\int_{\gamma_{L+R,\eta}}\left|F(t,\eta,\lambda)\right|e^{-\Im[a(\lambda)]\Re\left[S_\eta(t)-S_\eta(t^*_-(\eta))\right]}~dt \cr
    &\leq e^{-\Im[a(\lambda)]\Re[S_\eta(t^*_-(\eta))]}e^{-\Im[a(\lambda)]c_*}\int_{\gamma_{L+R,\eta}}\left|F(t,\eta,\lambda)\right|~dt 
    \leq e^{-\Im[a(\lambda)]\Re[S_\eta(t^*_-(\eta))]}\cdot e^{-\Im[a(\lambda)]c_*} K_{LR}
\end{align}
where the constant
\begin{equation}
    K_{LR}:=\max_{\eta\in\Omega_+}\max_{\lambda\in\overline{B}(0,\epsilon)}\int_{\gamma_{L+R,\eta}}\left|F(t,\eta,\lambda)\right|~dt.
\end{equation}
is finite by hypotheses of Theorem \ref{result}.

\end{proof}

The combination of Lemmas \ref{intLeading} and \ref{intEstimate} proves Theorem \ref{result} provided that $\lambda$ is in the upper half plane.

\begin{remark}\label{lambdaLHP}
When $\lambda\to0$ in the lower half plane, the key difference is that $\Im[a(\lambda)]\to\infty$ (see Proposition \ref{aProp}).  With that in mind, rewrite the integrand of $\eqref{int}$ as 
\begin{align}
F(t,\eta,\lambda)e^{-\Im[a(\lambda)]S_\eta(t)}=F(t,\eta,\lambda)e^{\Im[a(\lambda)][-S_\eta(t)]}.
\end{align}
We can replace $S$ with $-S$ and carry out the same analysis as before.  The only difference will be that the regions in the $t$ plane where $\Re[S_\eta(t)-S_\eta(t^*_-(\eta))]<0$ and $\Re[S_\eta(t)-S_\eta(t^*_-(\eta))]>0$ will swap, so we deform $[0,1]$ to a different contour, see Figure \ref{figLambdaLHP}.  All previous ideas from this section can now be applied using the new contour $\gamma$.

\end{remark}

\begin{figure}[h]
\begin{center}
\begin{tikzpicture}

\draw[blue,thick] (-4.5,0) -- node[anchor=south,black] {$-$} node[anchor=north,black] {$+$} (-3,0);
\draw[red,thick] (-3,0) -- (3,0);
\draw[blue,thick] (3,0) -- node[anchor=south,black] {$+$} node[anchor=north,black] {$-$} (4.5,0);
\draw [thick,smooth,red] (0.5,-2) .. controls (-0.15,0) .. (0.5,2);
\draw [thick,smooth] (-3,0) .. controls (-1,0.88) .. (0,0);
\draw [thick,smooth] (0,0) .. controls (1,-0.88) .. (3,0);

\filldraw[black] (-3,0) circle [radius=1.5pt] node[anchor=south] {0};
\filldraw[black] (0,0) circle [radius=1.5pt] node[anchor=north east] {$t^*_-(\eta)$};
\filldraw[black] (3,0) circle [radius=1.5pt] node[anchor=north] {1};

\filldraw[black] (0.72,-0.56) circle [radius=1.5pt] node[anchor=north] {$t^*_R(\eta)$};
\filldraw[black] (-0.72,0.56) circle [radius=1.5pt] node[anchor=south] {$t^*_L(\eta)$};

\end{tikzpicture}
\end{center}
\vspace{-5mm}
\caption{\hspace{.05in}Path of integration when $\lambda$ is in the lower half plane.} \label{figLambdaLHP}
\end{figure}
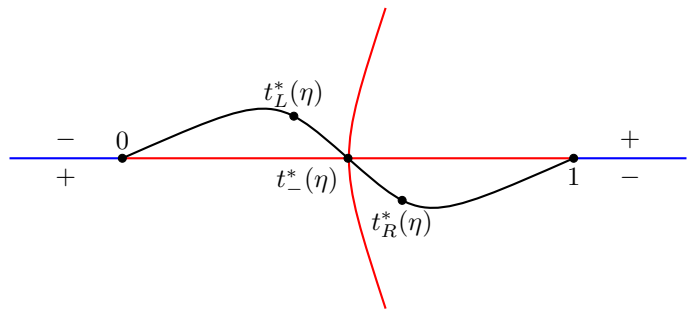



\bibliographystyle{amsplain}
\bibliography{4pt_NoGap.bib}

\end{document}